\documentclass[12pt]{article}
\pdfoutput=1
\usepackage[english]{babel}
\usepackage[utf8]{inputenc}
\usepackage{layouts}
\usepackage{amssymb}

\usepackage{graphicx}
\usepackage{float} 
\usepackage{subfigure}
\usepackage[colorlinks,
            linkcolor=black,
            anchorcolor=black,
            citecolor=black
            ]{hyperref}
\usepackage{times}
\usepackage{amsthm}
\usepackage{amsmath}
\usepackage{amscd}
\usepackage[all]{xy}
\usepackage{graphicx}
\usepackage{blindtext}
\usepackage{indentfirst}
\usepackage{enumerate}
\usepackage{enumitem}
\setlist[enumerate]{listparindent=\parindent}

\usepackage {nicematrix}
\usepackage[]{caption2} 
\renewcommand{\thesubfigure}{(\roman{subfigure})}
\makeatletter \renewcommand{\@thesubfigure}{\thesubfigure \space}
\renewcommand{\p@subfigure}{} \makeatother
 
\theoremstyle{plain}
\newtheorem{thm}{Theorem}[subsection]
\newtheorem{apt}[thm]{Assumption}
\newtheorem{prop}[thm]{Proposition}
\newtheorem{cor}[thm]{Corollary}
\newtheorem{lma}[thm]{Lemma}

\theoremstyle{definition}
\newtheorem{dfn}[thm]{Definition}
\newtheorem{ex}[thm]{Example}

\theoremstyle{remark} 

\newtheorem*{note}{Note}

\newcommand{\RR}{\mathbb R}
\newcommand{\NN}{\mathbb N}
\newcommand{\ZZ}{\mathbb Z}

\newcommand{\DC}{\mathcal{C}}

\newcommand{\DA}{\mathcal{A}}

\newcommand{\DP}{\mathcal{P}}
\newcommand{\DQ}{\mathcal{Q}}

\newcommand{\DT}{\mathcal{T}}
\newcommand{\DM}{\mathcal{M}}
\newcommand{\DI}{\mathcal{I}}
\newcommand{\DDJ}{\mathcal{J}}
\newcommand{\GJ}{\mathfrak{J}}
\newcommand{\Gf}{\mathfrak{f}}
\newcommand{\GC}{\mathfrak{C}}
\newcommand{\GM}{\mathfrak{M}}
\newcommand{\cc}{\ol{c}}
\newcommand{\eps}{\epsilon}
\newcommand{\veps}{\varepsilon}
\newcommand{\lam}{\lambda}
\newcommand{\Lam}{\Lambda}
\newcommand{\inte}{\mathrm{int}}
\newcommand{\cl}{\mathrm{cl}}

\newcommand{\OO}{\varnothing}
\newcommand{\ol}[1]{\overline{#1}}
\newcommand{\lz}{{\lam_{0}}}

\newcommand{\pspa}{N\times [0,1]}
\newcommand{\MD}[1]{\DM(S^{#1})}

\newcommand{\mset}[2]{\{{#1}\mid{#2}\}}
\newcommand{\sset}[1]{\{{#1}\}}
\newcommand{\wt}[1]{\widetilde{#1}}
\newcommand{\FO}[1]{(\DP^{\Gf}_{#1},<^{\Gf}_{#1})}
\newcommand{\for}[1]{{{#1} \in \NN}}
\newcommand{\gam}[0]{\gamma}
\newcommand{\CH}{C\!H}

\newcommand{\LTD}[2]{\{(#1_i,#1_i^{'})\}_{i\in #2}}
\newcommand{\HHL}[1]{\text{the Hausdorff limit }\lim_{\lam_0} M_{#1}^{\lam_n}}
\newcommand{\HL}[1]{\lim_{\lam_0} M_{#1}^{\lam_n}}
\newcommand{\HLL}[2]{\lim_{{\lam_0}^{#2}} M_{#1}^{\lam_n}}
\newcommand{\h}[1]{\hat{#1}}
\newcommand{\hse}[0]{\h{S}_\eps}
\newcommand{\MDp}[0]{(\DM(\hse),\h{\DP})}
\newcommand{\ti}[1]{\tilde{#1}}
\usepackage{geometry}
\newcommand{\bi}[1]{\textbf{\textit{#1}}}
\newcommand{\inv}[0]{\mathrm{Inv}}
\newcommand{\GO}{\mathfrak{O}}
\newcommand{\GMD}[0]{(\DM(S),\DP)}
\newcommand{\SOI}[0]{\DI(\DP)}
\newcommand{\pa}[0]{\partial}
\usepackage{mathrsfs}
\newcommand{\SH}[0]{\mathscr{H}}
\newcommand{\GK}{\mathfrak{K}}

\title{Transition Matrix without Continuation in the Conley Index Theory}
\author{YANGHONG YU}
\date{}

\begin{document}

\maketitle

\begin{abstract}
Given a one-parameter family of flows over a parameter interval $\Lam$, assuming there is a continuation of Morse decompositions over $\Lam$, Reineck \cite{Reineck1998connection} defined a singular transition matrix to show the existence of a connection orbit between some Morse sets at some parameter points in $\Lam$. This paper aims to extend the definition of a singular transition matrix in cases where there is no continuation of Morse decompositions over the parameter interval. This extension will help study the bifurcation associated with the change of Morse decomposition from a topological dynamics viewpoint.
\end{abstract}

\section{Introduction}\label{Introduction}
\par In Conley index theory, a transition matrix is used to detect the existence of a non-robust connecting orbit between Morse sets in a Morse decomposition. Suppose there is a parameterized flow $\phi_t^\lambda$ with the parameter $\lambda$ in an open interval. In order to obtain the existence of a connecting orbit that occurs by bifurcation at some parameter point, it is often helpful to introduce a slow parameter drift of speed $\eps$ on the parameter interval and to consider a connection matrix for the combined slow-fast system on the product of the phase and parameter spaces with $\eps$ tending to $0$. Such an idea was proposed by Reineck \cite{Reineck1998connection}, introducing the notion of a singular transition matrix.
Reineck \cite{Reineck1998connection} assumed there is a continuation of Morse decompositions over the parameter interval and constructed the singular transition matrix. Suppose a Morse set undergoes bifurcation at some point in the parameter interval. In that case, the continuation of Morse decompositions may break down, and the conventional singular transition matrix can not be defined.
\par In this paper, we assume there is a breakdown of the continuation of Morse decompositions in the parameter interval and generalize the results of Reineck \cite{Reineck1998connection} by extending the singular transition matrix under this situation. The information on the existence of connecting orbits obtained from the extended singular transition matrix can be used to infer the bifurcation information over the parameter interval.
\par Instead of adding a slow parameter drift, a transition matrix between two parameter points has been defined axiomatically by Franzosa-Mischaikow \cite{FranzosaMischai1998}, Kokubu \cite{Kokubu2000ontransitionmatrices} and others. This definition still assumes the continuation of Morse decompositions over the parameter interval between these two parameter points. In this paper, we also extend the axiomatic definition under the situation that the continuation of Morse decompositions breaks down at a point in the parameter interval.
\par The rest of the paper begins with Section \ref{Dynamicalsystem} when basic background definitions of a dynamical system are given. 
In Section \ref{Conleyindex} and Section \ref{Connectionmatrix}, we very briefly review the Conley index and the connection matrix. In Section \ref{Transition Matrix without Continuation}, we define the breakdown of the continuation of Morse decompositions in Subsection \ref{The Finest Decomposition of a Continuable Interval Pair}. In order to obtain information on connection orbits at some parameter, we study connecting orbits in the extended slow-fast flow $\Phi_t^\epsilon$ when $\epsilon$ limit to $0$ in Subsection \ref{limit} and Subsection \ref{CO}. A singular transition matrix without continuation is constructed in Subsection \ref{STM}, and the axiomatic one is defined in Subsection \ref{An AxiomaticExtensionApproachtotheSingularity}.



\section{Dynamical System}\label{Dynamicalsystem}
\par In this section, we introduce basic concepts for dynamical systems, especially the notions of attractor-repeller decomposition and Morse decomposition, which describe the gradient-like information in dynamics, and then introduce their robustness.

\subsection{Basic Notations}
\par The dynamical system concerns long-term behaviors in ODEs or iterated maps. In the case of continuous time, it is formulated as a flow on a phase space.
Let $(X,d)$ be a locally compact metric space with metric $d$. 
\begin{dfn}[Flow]\label{Flow}
	A continuous map $\phi:\RR \times X\to X$ is a \textbf{\textit{flow}} on $X$ if it satisfies the following:
	\begin{enumerate}
		\item $\phi(0,x)=x$.
		\item $\phi(t_1,\phi(t_2,x))=\phi(t_1+t_2,x)$ for any $t_1,t_2\in \RR$.
	\end{enumerate}
\end{dfn}
\par  Usually, $t$ in a flow $\phi(t,x)$ is viewed as time. Let $\phi_t(x):=\phi(t,x)$ and
$\phi(I,Y):=\mset{\phi(x,t)}{t\in I \text{ and } x\in Y}$
for any subset $Y\subset X$ and any interval of time $I\subset \RR$. 
\begin{dfn}[Invariant set]\label{InvariantSet}
A subset $S\subset X$ is an \textbf{\textit{invariant set}} under the flow $\phi_t$ if it satisfies:
$\phi(\RR,S)=S$. In particular, if an invariant set consists of a single point $x\in X$, we call the $x$ a \bi{fixed point}.
\end{dfn}
\par For a point $x\in X$, the \bi{orbit} of $x$ is $\GO(x):=\phi(\RR,x)$. We write $x\cdot t:=\phi(t,x)$ and $x\cdot I:=\phi(I,x)$ with $I\subset \RR$ for simplicity.
\par For a subset $Y\subset X$, in order to study behaviors of orbits from $Y$ when the time $t\to \pm\infty$, we define the \textbf{\textit{alpha limit set}} $\alpha(Y)$ and the \textbf{\textit{omega limit set}}  $\omega(Y)$.
\begin{dfn}\label{alphaomegasets} For a subset $Y\subset X$, the \textbf{\textit{alpha limit set}} of $Y$ is defined as:
	$$\alpha(Y):=\bigcap_{t>0}\cl(\phi((-\infty,-t],Y));$$
similarly, the \textbf{\textit{omega limit set}} of $Y$ is defined as:
 $$\omega(Y):=\bigcap_{t>0}\cl(\phi([t,\infty),Y)).$$	
\end{dfn}
\par These limit sets $\alpha(Y)$ and $\omega(Y)$ are closed and invariant under the flow $\phi_t$. Moreover, if $Y$ is connected in $X$, $\alpha(Y)$ and $\omega(Y)$ are also connected in $X$.
\par For any subset $Y\subset X$, let $\inv(Y,\phi_t):=\mset{x\in Y}{\phi(\RR,x)\subset Y}$ denotes the maximal invariant set contained in $Y$ under the flow $\phi_t$. When there is no confusion, we omit the dependence on $\phi_t$ and denote it as $\inv(Y)$. 
\begin{dfn}[Isolating neighborhood]\label{IsolatingNeighborhood}
	A compact subset $N\subset X$ is called an \bi{isolating neighborhood} if $\inv(N)\subset \inte(N)$. We call $S=\inv(N)$ the \bi{isolated invariant set} isolated by $N$.
\end{dfn}
\par An isolated invariant set is robust when the flow is perturbed. More precisely, let $\phi_t^\lambda:\RR\times X \to X$ be a parameterized flow on $X$ with a parameter $\lam\in \RR$. Then, we have the following proposition.
 \begin{prop}[Robustness of isolated invariant sets; \cite{Mischaikow2002ConleyIndex}, Proposition 1.1]
 	Let $N$ be an isolated neighborhood for the flow $\phi_t^{\lam_*}$ at some parameter value $\lam_*\in \RR$. Then there is an open interval $O\subset \RR$ such that $\lam_*\in O$ and $N$ is an isolating neighborhood of $\phi^\lam_t$ for all $\lam \in O$.
 \end{prop}
\begin{dfn}
Let $N\subset X$ be a compact set, and let $S^\lam:=\inv(N,\phi^\lam_t)$ with $\lam \in [\lam_1,\lam_2]$. Assume that at $\lam_1$ and $\lam_2$, $N$ is an isolating neighborhood, then the two isolated invariant sets $S^{\lam_1}$ and $S^{\lam_2}$ are \bi{related by continuation} or \bi{$S^{\lam_1}$ continues to $S^{\lam_2}$} if $N$ is an isolating neighborhood of $\phi^\lam_t$ for all $\lam\in [\lam_1,\lam_2]$.
\end{dfn}
\par In order to study internal structures of isolated invariant sets, we introduce, in the following subsections, decompositions of an isolated invariant set into several isolated invariant subsets called an attractor-repeller decomposition or, more generally, a Morse decomposition.

\subsection{Attractor-Repeller Decomposition}\label{Attractor-Repeller Decomposition}
\begin{dfn}[Attractor-repeller decomposition]
Let a compact set $S\subset X$ be an isolated invariant set under the flow $\phi_t$. We say $A\subset S$ is an \bi{attractor} in $S$ if there is an open neighborhood $U\subset X$ of $A$ with 
$$ \omega(U \cap S) = A.$$
The \bi{dual repeller} of $A$ in $S$ is  defined by:
$$R:=\mset{x\in S}{\omega(x)\cap A=\OO}.$$
We call the pair $(A, R)$ the \bi{attractor-repeller decomposition} of the invariant set $S$.
\end{dfn}
\par From the closedness of omega and alpha limit sets, we have the following proposition.
\begin{prop}[\cite{Salamon1985ConnectedSimpleSys}, Lemma $3.2$]
Let $(A, R)$ be an attractor-repeller pair of an isolated invariant set $S$. Then,
\begin{enumerate}
	\item $A$ and $R$ are isolated invariant subsets of $S$.	
	\item $R=\alpha(S-A)$ and $A=\omega(S-R)$.
	\item $A\cap R =\OO$.
\end{enumerate} 
\end{prop}
\begin{dfn}[Connecting orbits]
	The \bi{set of connecting orbits} for an attractor-repeller pair $(A,R)$ of $S$ is defined by:
	$$
	\DC(A,R; S):=\mset{x\in S}{\omega(x)\subset A \text{ and } \alpha(x)\subset R}.
	$$
	For any $x\in \DC(A,R; S)$, we call the orbit $\GO(x)$ a \bi{connecting orbit} from $R$ to $A$ for the pair $(A,R)$ of $S$.
\end{dfn}
\begin{thm}[Attractor-repeller decomposition;\cite{Mischaikow2002ConleyIndex}, Theorem $2.4$]
	Let the pair $(A, R)$ be an attractor-repeller pair of an isolated invariant set $S$. Then,
	$$
	S=A\sqcup R\sqcup \DC(A,R; S).
	$$
\end{thm}

\par In general, the attractor-repeller pair $(A,R)$ is robust under perturbation because of the robustness of isolated invariant sets. However, the set of connecting orbits may change discontinuously under perturbation, which will be discussed in the section \ref{Transition Matrix without Continuation}.

\begin{prop}[Robustness of attractor-repeller decomposition; \cite{Mischaikow2002ConleyIndex}, Theorem 2.5] For a parameterized flow $\phi_t^\lambda$ with a parameter $\lam \in \RR$, let $N$, $N_A\subset N$ and $N_R\subset N$ be isolating neighborhoods under the flow $\phi_t^{\lam_*}$ for some $\lam_*\in \RR$ and let $S_\lam:=\inv(N,\phi_t^\lambda), A_\lambda:=\inv(N_A,\phi_t^\lambda),$ and $R_\lambda:=\inv(N_R,\phi_t^\lambda)$ for $\lam\in \RR$. 
\par If $(A_{\lam_*}, R_{\lam_*})$ is an attractor-repeller decomposition for $S_{\lam_*}$, then there is an open interval $O\subset \RR$ such that $\lam_*\in O$ and the pair $(A_\lam ,R_\lam)$ is an attractor-repeller decomposition for $S_\lambda$ under the flow $\phi_t^\lambda$ for all $\lam \in O$.
\end{prop}
\subsection{Morse Decomposition}\label{Morse Decomposition}

\par Generalizing an attractor-repeller decomposition of an isolated invariant set $S$, we now study decomposing $S$ into finitely many isolated invariant subsets and connecting orbits between them.
\begin{dfn}[Morse decomposition]
A \bi{Morse decomposition} of an isolated invariant set $S$ is a finite collection of mutually disjoint isolated invariant sets $\DM(S):=\mset{M_p\subset S}{p\in (\DP,<)}$ which are indexed by a partially ordered set $(\DP,<)$ satisfying that
	 for any $$x\in S-\bigcup_{p\in \DP}M_p$$ there exist $p,q\in (\DP,<)$ with $p<q$ such that $$ \omega(x)\subset M_p \quad \text{and} \quad \alpha(x)\subset M_q.$$
We denote a Morse decomposition as $(\DM(S),(\DP,<))$, and call each isolated invariant set $M_p\in \DM(S)$ a \bi{Morse set}.
\end{dfn}
\par Because $\DM(S)$ is a finite collection, $\DP$ is a finite set. A partial order $<^\Gf$ on $\DP$ is defined as: $p<^\Gf q$ if and only if there exists a point $x\in S$ such that $\omega(x)\subset M_p$ and $\alpha(x)\subset M_q$. Then $(\DM(S),(\DP,<^\Gf))$ is a Morse decomposition of $S$. We call the order $<^\Gf$ the \bi{flow-defined order} for $\DM(S)$. If we take any extension $<^{'}$ of the flow-defined order $<^\Gf$, then $(\DM(S),(\DP,<{'}))$ is also a Morse decomposition of $S$. Any extension of the flow-defined order is called an \bi{admissible order}.

\par By making Morse sets and the connecting orbits between them to a new isolated invariant set, a coarser Morse decomposition can be created. In order to make such construction precise, we define an interval in a partially ordered set.
\begin{dfn}[Interval]\label{Interval}
	For a partially ordered set $\DP$, a subset $I \subset \DP$ is called an \bi{interval} in $\DP$ if $p,q\in I$, $p<r<q$ implies $r\in I$. The collection of all intervals in $\DP$ is denoted as $\DI(\DP)$.
\end{dfn}
\begin{dfn}[Attracting interval]\label{attractinginterval}
	$A\in \DI(\DP)$ is called an \bi{attracting interval} if $p<a$ implies $p\in A$, for all $a\in A$ and $p\in \DP.$ Let $\DA(\DP)$ be the set of attracting intervals in $\DP$.
\end{dfn}

\begin{dfn}[Adjacent pair] For a partially ordered set $\DP$, the ordered pair $(I,J)$ of mutually disjoint intervals in $\DP$ is called an \bi{adjacent pair} if it satisfies:
\begin{enumerate}
	\item $I\cup J\in \SOI$.
	\item $p\in I, q\in J$ implies $p \ngtr q$.
\end{enumerate}
	The collection of all adjacent pairs in $\DP$ is denoted as $\DI_2(\DP)$, and we write $I\cup J$ as $I\!J$, if $(I,J)\in \DI_2(\DP)$.
\end{dfn}

\par For a Morse decomposition $\GMD$ and $I\in \DI(\DP)$, we define:
$$
M_I:=\left(\bigcup_{p\in I} M_p\right)\cup \left(\bigcup_{p,q\in I}\DC(M_q,M_p) \right).
$$
\begin{prop}[\cite{Mischaikow2002ConleyIndex}, Proposition $2.12$]
$M_I$, as defined above, is a compact isolated invariant set.	
\end{prop}

\par Morse decomposition decomposes an isolated invariant set into finitely many Morse sets and connecting orbits between them. However, we still lack information on each Morse set in a Morse decomposition, and we are unable to determine the existence or the structure of connecting orbits between these Morse sets, which will be discussed in the following sections. 

\section{Conley Index}\label{Conleyindex}
\par In this section, we introduce the notion of the Conley index for isolated invariant sets, which can tell us the existence, stability, and other topological properties of invariant sets in dynamical systems. The presentation of this section is based on the works of Conley \cite{Conley1978isolated} and Salamon \cite{Salamon1985ConnectedSimpleSys}.

\subsection{Index Pairs}\label{IndexPairs}
\par In this subsection, we introduce index pairs 
to define the Conley index. 
 A \bi{pointed space} $(Y,y)$ is a pair of a topological space $Y$ and a distinguished point $y\in Y$. 
\begin{dfn}[Quotient space]
	Given a pair of subspaces $(N,L)$ with $\OO\neq L\subset N\subset Y$, define a pointed space as follows:
	$$
	(N/L,[L]):=((N-L)\cup [L],[L])
	$$
	in which we collapse the subspace $L$ into a point, denoted as $[L]$. $N/L$ is equipped with the quotient topology. 
	
When $L=\OO$, $N/L=N/\OO:=N\cup \sset{*}$, $(N/\OO,[\OO])=(N\cup \sset{*},\sset{*})$ in which a new point $\sset{*}$ is added to the space $N$. The topology of $N/\OO$ is given by that a subset $O\subset N/\OO $ is open if it satisfies one of the following:
\begin{enumerate}
		\item $O$ is open in $N$ with $\sset{*} \notin O$.
		\item The set $(O - \sset{*})$ is open in $N$ with $\sset{*} \in O$.
	\end{enumerate}

\end{dfn}

\begin{dfn}[Index pairs]\label{indexpairsdef}
	Let $\phi_t$ be a flow on $X$, and let $S$ be an isolated invariant set. Then a pair $(N,L)$ of compact sets in $X$ is an \bi{index pair} for $S$ with respect to the flow $\phi_t$, if it satisfies the following:
	\begin{enumerate}
		\item $L\subset N$.
		\item $N-L$ is a neighborhood of $S$ and $S=\inv(\cl(N-L))$.
		\item $L$ is positively invariant in $N$, namely, for any $l\in L$ and any $t\geq 0$, $l\cdot[0,t]\subset N$ implies $l\cdot t\in L$.
		\item If there is a point $l\in N$ that $l\cdot[0,\infty)\not\subset N$, then there exist a $t_0\geq 0$ such that $l\cdot[0,t_0]\subset N$ and $l\cdot t_0 \in L$.
	\end{enumerate}
\end{dfn}
\begin{note} The statement $4$ in Definition \ref{indexpairsdef} indicates that every point will enter $L$ before leaving $N$. Therefore, the set $L$ in an index pair $(N,L)$ is also called an \bi{exit set} for $N$.	
\end{note}

\begin{thm}[Existence of index pairs; \cite{Salamon1985ConnectedSimpleSys}, Theorem $4.3$]
	Let $N\subset X$ be an isolating neighborhood of the isolated invariant set $S$ with respect to the flow $\phi_t$, and let $U$ be any neighborhood of $S$. Then, there exists an index pair $(N_1, N_0)$ for $S$ with respect to $\phi_t$, such that $N_1,N_0$ are positively invariant in $N$ and $\cl(N_1-N_0)\subset U$.
\end{thm}
\par In general, the index pair $(N_1,N_0)$ for an isolated invariant set $S$ may not be unique, but the homotopy type of $(N_1/N_0, [N_0])$ for any index pair $(N_1,N_0)$ of $S$ keeps the same.
\begin{prop}[Equivalence of index pairs; \cite{Salamon1985ConnectedSimpleSys} Lemma $4.6,4.7,4.8$]\label{Equviance of Index Pairs}
	Let $(N,L)$ and $(N^{'},L^{'})$ be index pairs for an isolated invariant set $S$. Then $(N/L,[L])$ and $(N^{'}/L^{'},[L^{'}])$ are homopoty equivalent.
\end{prop}
\par Let $(N,L)$ be an index pair for $S$, then we know that the homotopy type of the quotient space $(N/L,[L])$ is a topological invariant of the isolated invariant set $S$. We will use this to define the Conley index in the next subsection.
\subsection{Conley Index}\label{ConleyIndexSubsection}
\par Let $(N,L)$ be an index pair for an isolated invaiant set $S$ with respect to the flow $\phi_t$. We define the \bi{homotopy Conley index} of $S$ in the following way.
\begin{dfn}[Conley index]\label{Homotopy Conley Index}
	The \bi{homotopy Conley index} of $S$ with respect to the flow $\phi_t$ is the homotopy type of $(N/L,[L])$, denoted as $h(S,\phi_t)$.
\end{dfn}
\begin{note}
By Proposition \ref{Equviance of Index Pairs}, the Conly index of $S$ is well-defined. In practice, we usually consider the \bi{homological Conley index}, defined as the homology of $h(S,\phi_t)$:
$$
C\!H_{*}(S,\phi_t):=H_{*}(h(S,\phi_t))=H_{*}(H/L,[L]).
$$
If the flow $\phi_t$ is clear, we denote it as $C\!H_{*}(S)$. It is known that there is always a good index pair $(H,L)$ for $S$ such that 
$$
C\!H_{*}(S) \approx H_{*}(N,L).
$$
From now on, we always consider the homological Conley index, and for simplicity, we consider a finite field coefficient, such as $\ZZ_2$.
\end{note}

\par Taking the contrapositive of the preceding example, we have the \bi{Wa\.zewski property}.
\begin{thm}[Wa\.zewski property]\label{Wazewski property}
	Let $N$ be an isolating neighborhood in the flow $\phi_t$, if $\CH_*(\inv(N))\not\approx 0$ then $\inv(N)\neq \OO$.
\end{thm}
\par Conley index remains the same if two invariant sets are related by continuation.
\begin{thm}[Continuation property; cf. \cite{Mischaikow2002ConleyIndex}, Theorem $3.10$]\label{Continuation Property}
	For a parameterized flow $\phi_t^\lambda$ with a parameter $\lam \in \RR$,
	let $S^\lz$ and $S^{\lam_1}$ be isolated invariant sets under the flows $\phi_t^\lz$ and $\phi_t^{\lam_1}$ respectively. Let $S^\lz$ and $S^{\lam_1}$ be related by continuation. Then $\CH_*(S^\lz) \approx \CH_*(S^{\lam_1})$.
\end{thm}

\par The Conley index provides information about the isolated invariant set, but to fully understand the dynamics in the system, we also need to know about connecting orbits. In the next section, we will use an algebraic tool to relate the Conley indices of Morse sets to know more about connecting orbits.


\section{Connection Matrix}\label{Connectionmatrix}
\par This section presents the notion of the connection matrix for a Morse decomposition, which can tell the existence of a connecting orbit between some pair of the Morse sets. We begin with the case of an attractor-repeller pair and then extend it to a more general Morse decomposition. Presentation of this section relies on Franzosa \cite{Franzosa1989connectionmatrixtheory}.
\subsection{The Structure of Attractor Repeller Pairs}\label{TheStructure ofAttractorRepellerpairs}
\par In this subsection, we study the existence of connecting orbits
between the attractor and repeller in an attractor-repeller decomposition. It can be measured by a homomorphism from the homological Conley index of the repeller to that of the attractor, which is called the connecting homomorphism for the attractor-repeller decomposition. Firstly, we begin with the \bi{index triple}. 
\begin{dfn}[Index triple]\label{IndexTriple}
	Let $(A,R)$ be an attractor-repeller decomposition for an isolated invariant set $S$ with respect to the flow $\phi_t$. An \bi{index triple} for the attractor-repeller decomposition $(A,R)$ is a collection of compact sets $(N_2,N_1,N_0)$ with $N_0\subset N_1 \subset N_2$ that satisfies the following:
	\begin{enumerate}
		\item $(N_2,N_0)$ is an index pair for $S$;
		\item $(N_2,N_1)$ is an index pair for $R$;
		\item $(N_1,N_0)$ is an index pair for $A$.
	\end{enumerate}
\end{dfn}

\begin{thm}[Existence of index triple; \cite{Mischaikow2002ConleyIndex}, Theorem 4.2]\label{ExistenceofIndexTriple}
	Let $(A,R)$ be an attractor-repeller decomposition for an isolated invariant set $S$ with respect to the flow $\phi_t$. Then there exists an index triple $(N_2,N_1,N_0)$ for $(A,R)$.
\end{thm}
\par Let $C_*(X):=\sset{(C_n(X),\pa_n)}_\for{n}$ denote a chain complex of topological space $X$, and $C_n(X,Y)$ denote the quotient group $C_n(X,Y):=C_n(X)/C_n(Y)$ for topological spaces $X$ and $Y\subset X$. 
It naturally defines a boundary homomorphism $\pa_n(X,Y):C_n(X,Y)\to C_{n-1}(X,Y)$, which makes $C_*(X,Y)=\sset{(C_n(X,Y),\pa_n(X,Y))}_\for{n}$ a chain complex.
\par For an index triple $(N_2,N_1,N_0)$, we have the following short exact sequence:
\begin{equation*}
\begin{CD}
	0    @>>> C_*(N_1,N_0)   @>i>>   C_*(N_2,N_0)   @>j>>   C_*(N_2,N_1)   @>>>   0
\end{CD}
\end{equation*}
of chain complexes, more precisely, presented by the following commutative diagrams:
\begin{equation*}
\begin{CD}
     @.   \vdots           @.         \vdots             @.      \vdots           @.       \\
@.        @VVV                       @VVV                       @VVV                   \\
0    @>>> C_{n+1}(N_1,N_0)   @>i>>   C_{n+1}(N_2,N_0)   @>j>>   C_{n+1}(N_2,N_1)   @>>>   0\\
@.        @V\pa_{n+1}VV                       @V\pa_{n+1}VV                        @V\pa_{n+1}VV                   \\
0    @>>> C_{n}(N_1,N_0)   @>i>>   C_{n}(N_2,N_0)   @>j>>   C_{n}(N_2,N_1)   @>>>   0\\
@.        @V\pa_{n}VV                      @V\pa_{n}VV                       @V\pa_{n}VV                   \\
0    @>>> C_{n-1}(N_1,N_0)   @>i>>   C_{n-1}(N_2,N_0)   @>j>>   C_{n-1}(N_2,N_1)   @>>>   0\\
@.        @VVV                        @VVV                      @VVV                   \\
     @.   \vdots            @.         \vdots            @.      \vdots            @.       
\end{CD}
\end{equation*}
This gives rise to the long exact sequence of corresponding homology groups:
\begin{equation*}
	\begin{CD}
\cdots @>>> H_{n}(N_1,N_0) @>i^{\sharp}_{n}>> H_{n}(N_2,N_0) @>j^\sharp_{n}>>  H_{n}(N_2,N_1) \\
 @. @. @>{\ti{\partial}_{n}}>> H_{n-1}(N_1,N_0) @>>> \cdots
\end{CD}
\end{equation*}
which gives a long exact sequence of the homological Conley indices of the attractor-repeller pair of the isolated invariant set $S$
\begin{equation*}
	\begin{CD}
\cdots @>>> \CH _{n}(A) @>i^\sharp_{n}>> \CH _{n}(S) @>j^\sharp_{n}>>  \CH _{n}(R) \\
 @. @. @>{\ti{\partial}_{n}}>> \CH _{n-1}(A) @>>> \cdots
\end{CD}
\end{equation*}

\par By the excision theorem, if there are no connecting orbits, then $\ti{\partial}_*=0$ as in the next proposition.
\begin{prop}[\cite{Mischaikow2002ConleyIndex}, Theorem 4.3] Let $(A,R)$ be an attractor-repeller decomposition for the isolated invariant set $S$ with respect to the flow $\phi_t$. If $\DC(A,R)=\OO$, i.e., $S=A\cup R$, then $\ti{\pa}_*=0$.
\end{prop}

\par Take the contrapositive of the preceding theorem, we can infer the existence of connecting orbits from the map $\ti{\pa}_*$.
\begin{thm}[Existence of connecting orbits]\label{ExistenceofconnectingorbitsARPair}
Let $(A,R)$ be an attractor-repeller decomposition for the isolated invariant set $S$ with respect to the flow $\phi_t$. If there is an $n\in \NN$ such that $\ti{\pa}_n\neq 0$ then $\DC(A,R)\neq \OO$.
\end{thm}
\begin{note}
It should be noted that if $\ti{\pa}_n=0$, there may also exist a connecting orbit. An example is given by a sink and a source in a vector field on the sphere $S^2$, where all remaining orbits connect $R$ and $A$.
Considering the degree of homological Conley indices for the sink and source, we have $\ti{\pa}_*=0$.
\end{note}
\par The map $\ti{\pa}_*$ can be used to restore $\CH_*(S)$ from the direct sum $\CH_*(A)\oplus \CH_*(R)$ in the following way.
\par We define the map:
$$
\Delta_n= 
\begin{bmatrix}
0 &\ti{\pa}_n(A,R) \\
0 &0 \\
\end{bmatrix}: \CH_n(A)\oplus \CH_n(R) \to \CH_{n-1}(A)\oplus \CH_{n-1}(R)
$$
Obviously, it follows that  $\Delta_n\circ \Delta_{n+1}=0$, which makes $\Delta:=\sset{\Delta_n}_{\for{n}}$ a boundary operator such that $(\CH_*(A)\oplus\CH_*(R),\Delta)$ forms a chain complex. The homology groups $H_n\Delta$ of the chain complex are given by:
$$
H_n\Delta:=\frac{\ker(\Delta_n)}{\mathrm{Im}(\Delta_{n+1})}, \qquad \text{for all } \for{n}
$$

\begin{prop}[\cite{Mischaikow2002ConleyIndex}; Propostion $4.6$]
	$H_*\Delta \approx \CH_*(S).$
\end{prop}


\par The map $\Delta_*$ defined above for an attractor-repeller pair is the simplest case of the connection matrix, and we will extend the connection matrix for a general Morse decomposition in the next subsection.

\subsection{Connection Matrix}\label{ConnectionMatrixSubsection}
\par In this subsection, we define the connection matrix $\Delta_*$ for a Morse decomposition $(\MD{},\DP)$, in which $\MD{}=\mset{M_p}{p\in\DP}$. 
\begin{dfn}[Index filtration]\label{Indexfiltration}
	Let $(\MD{},\DP)$ be a Morse decomposition for $S$.  A collection of compact subsets $\mset{N(J)}{J\in \DA(\DP)}$ indexed by the attracting intervals in $\DP$ is called an \bi{index filtration} for $(\MD{},\DP)$ if it satisfies the following:
	\begin{enumerate}
		\item $(N(J),N(\OO))$ is an index pair for $M_J$, for all $J\in \DA(\DP)$.
		\item $N(J_1\cap J_2)=N(J_1)\cap N(J_2)$, for all $J_1,J_2\in \DA(\DP)$.
		\item $N(J_1\cup J_2)=N(J_1)\cup N(J_2)$, for all $J_1,J_2\in \DA(\DP)$.
	\end{enumerate}
\end{dfn}
\par For any Morse decomposition $(\MD{},\DP)$, Franzosa showed the existence of an index filtration for it in Theorem $3.8$ in \cite{Franzosa1986IndexFiltration}.
\begin{prop}[\cite{Franzosa1986IndexFiltration}, Proposition $3.5$]\label{indexfiltrationprop}\
\begin{enumerate}
	\item For any $J\in \DI(\DP)$, there is an interval $K\in \DA(\DP)$ such that $(K,J)\in \DI_2(\DP)$, $K\!J\in \DA(\DP)$. 
	\item For any $J\in \DI(\DP)$, take the interval $K$ above then $(N(K\!J),N(K)))$ is an index pair for $M_J$. 
	\item For any $J\in \DI(\DP)$, if there are two intervals $K_1$, $K_2$ satisfy the statement $1$ for $J$ then the spaces $N(K_1\!J)/N(K_1)$ and $N(K_2\!J)/N(K_2)$ are homeomorphic.
\end{enumerate}
\end{prop}
\par Therefore, we have index pairs for all $J\in \DI(\DP)$. For any $J\in \DI(\DP)$, we define the chains $C_*(J):=C_*(N(K\!J)/N(K))$ for any $K$ that satisfies statement $1$ in Proposition \ref{indexfiltrationprop}. Then for any adjacent pairs $(I,J)\in \DI_2(\DP)$, we have the following short exact sequence:
$$
\begin{CD}
	0 @>>> C_*(I) @>i>> C_*(I\!J) @>j>> C_*(J) @>>> 0
\end{CD}
$$
This gives rise to the following long exact sequence of corresponding homology groups:
\begin{equation} \label{homology braid}
\begin{CD}
\cdots @>>> \CH_n(M_I) @>{i^\sharp_n}>> \CH_n(M_{I\!J}) @>{j^\sharp_n}>> \CH_n(M_J) \\
@. @. @>{\ti{\pa}_n}>> \CH_{n-1}(M_I) @>>> \cdots	
\end{CD}
\end{equation}

\begin{dfn}[Homology braid]\label{HomologyBraid}
	Let $(\MD{},\DP)$ be a Morse decomposition for $S$. A collection 
	$$
	\mset{(\CH_*(M_I),i_*^\sharp(I,I\!J),j_*^\sharp(I\!J,J),\ti{\pa}_*(I,J))}{I\in \DI(\DP) \text{ and } (I,J)\in \DI_2(\DP)}
	$$
	of homological Conley indices $\CH_*(M_I)$ and degree $0$ maps $i^\sharp_*(I,I\!J):\CH_*(M_I)\to \CH_*(M_{I\!J}), j^\sharp_*(I\!J,J): \CH_*(M_{I\!J})\to \CH_*(M_J)$ and degree $-1$ map $\ti{\pa}_*(I,J):\CH_*(M_J)\to \CH_*(M_I)$ is a \bi{homology braid} for the Morse decomposition $(\MD{},\DP)$ if the maps such that the sequence (\ref{homology braid}) is exact for all adjacent pairs $(I,J)\in \DI_2(\DP)$. The homology braid for $(\MD{},\DP)$ is denoted as $\SH(\MD{},\DP)$, and as $\SH$ for short if it causes no confusion.
\end{dfn}
\par Let $C_*\Delta(I):=\oplus_{p\in I}\CH_*(M_p)$ for $I\in\DI(\DP)$. In particular, $C_*\Delta(\DP):=\oplus_{p\in \DP}\CH_*(M_p)$ when $I=\DP$.

\begin{dfn}[Upper triangular boundary map] A homomorphism $\Delta_*:C_*\Delta(\DP)\to C_*\Delta(\DP)$ is an \bi{upper triangular boundary map} if it satisfies the following:
\begin{enumerate}
	\item $\Delta_*$ has degree $-1$.
	\item $\Delta_{n}\circ\Delta_{n+1}=0$.
	\item Let $\Delta_n(p,q)$ be the restriction of $\Delta_n$ on the element $(p,q)$, $\Delta_n(p,q):\CH_{n}(M_q) \to \CH_{n-1}(M_p)$, then $\Delta(p,q)=0$ if $q\not<p$.
\end{enumerate}
\end{dfn}
\par Let $\Delta_*(I)$ be the restriction of $\Delta_*$ on an interval $I\in \DI(\DP)$, $\Delta_*(I):C_*\Delta(I)\to C_*\Delta(I)$, if $\Delta_*$ is an upper triangular boundary map, then the restriction $\Delta_*(I)$ is also an upper triangular boundary map. Thus if such boundary map $\Delta_*$ exists, for each $I\in \DI(\DP)$, letting $\Delta(I):=\sset{\Delta_n (I)}_{\for{n}}$ we have a chain complex $(C_*\Delta(I),\Delta(I))$ and the homology of the chain complex is denoted by $H_*\Delta(I)$. In particular, $H_*\Delta(p)=\CH_*(M_p)$ for $p\in \DP$.
\par For any adjacent pair $(I,J)\in \DI_2(\DP)$, we have the following short exact sequence:
$$
\begin{CD}
	0 @>>> C_*\Delta(I) @>i>> C_*\Delta(I\!J) @>j>> C_*\Delta(J) @>>> 0
\end{CD}
$$
This gives rise to the long exact sequence as follows:
\begin{equation}\label{longexactseqconnection}
	\begin{CD}
	\cdots @>>> H_n\Delta(I) @>{i^\sharp_n}>> H_n\Delta(I\!J) @>{j^\sharp_n}>> H_n\Delta(J) \\
@. @.	@>{\ti{\Delta}_n}>> H_{n-1}\Delta(I) @>>> \cdots	\end{CD}
\end{equation}

\par Compare the long exact sequence (\ref{homology braid}) and (\ref{longexactseqconnection}) and we define a \bi{connection matrix}.

\begin{dfn}[Connection matrix]\label{dfnofConnectionMatrix}
Let $(\MD{},\DP)$ be a Morse decomposition, and the homology braid for the Morse decomposition is given by $\SH(\MD{},\DP)$. Then an upper triangular boundary map $\Delta_*$ is a \bi{connection matrix} for the Morse decomposition $(\MD{},\DP)$ if for each $I\in \DI(\DP)$ there is an isomorphism $\GJ_*(I):H_*\Delta(I)\to H_*(I)$ such that:
\begin{enumerate}
	\item $\GJ_*(p):H_*\Delta(p)\to \CH_*(M_p)$ is the indentity map of $\CH_*(M_p)$ for all $p\in \DP$.
	\item For all adjacent pairs $(I,J)\in \DI_2(\DP)$, the following diagram commutes:
\end{enumerate} 
$$
\begin{CD}
 @>>> H_n\Delta(I) @>{i^\sharp_n}>> H_n\Delta(I\!J) @>{j^\sharp_n}>> H_n\Delta(J) @>{\ti{\Delta}_n}>> H_{n-1}\Delta(I) @>>>	\\
  @.    @V{\GJ_n(I)}VV                 @V{\GJ_n(I\!J)}VV                @V{\GJ_n(J)}VV                 @V{\GJ_{n-1}(I)}VV  @.    \\
 @>>> \CH_n(M_I) @>{i^\sharp_n}>> \CH_n(M_{I\!J}) @>{j^\sharp_n}>> \CH_n(M_J) @>{\ti{\pa}_n}>> \CH_{n-1}(M_I) @>>> 
\end{CD}
$$
\end{dfn}
\begin{note}
When $(p,q)\in \DI_2(\DP)$, then the $\ti{\Delta}_*(p,q)$ is the operator $\ti{\pa}_*$ defined in the last section for the attractor-repeller decomposition $(M_p, M_q)$ of the isolated invariant set $M_{\sset{p,q}}$.
\end{note}
\begin{thm}[\cite{Franzosa1989connectionmatrixtheory}, Theorem $3.8$]\label{existenceofconnectionmatrix}
For any isolated invariant set $S$ with a Morse decomposition $(\MD{},\DP)$, the set of connection matrices for the Morse decomposition is not empty.
\end{thm}
\par For $(p,q)\in \DI_2(\DP)$, by Theorem \ref{ExistenceofconnectingorbitsARPair} if $\Delta_n(p,q)\neq 0$ for some $n$, then there is a connecting orbit from $M_q$ to $M_p$. We conclude it as the following proposition.
\begin{prop}[Existence of connecting orbits]\label{Existenceofaconnectingorbit}
Let $(\MD{},\DP)$ be a Morse decomposition of $S$ and let $\Delta_*$ be a connection matrix for the Morse decomposition. For any $(p,q)\in \DI_2(\DP)$, if there is a $n\in \NN$ such that $\Delta_n(p,q)\neq 0$, then $\DC(M_p,M_q)\neq \OO$.
\end{prop}
\par At last, we introduce the definition of an isomorphism between homology braids. This will be used in the next section to link the homological Conley indices of Morse sets under the fast-slow flow with the homological Conley indices of Morse sets under the original parameterized flow.

\begin{dfn}[Homology braid isomorphism]\label{HomologyBraid}Let $(\MD{},\DP)$ and $(\MD{'},{\DP}^{'})$ be two Morse decompositions, and let $\SH=\mset{(\CH_*(M_I),i_*^\sharp,j_*^\sharp,\ti{\pa}_*)}{I\in \DI(\DP)}$ and $\SH^{'}=\mset{(\CH_*(M_J^{'}),(i^\sharp_*)^{'},(j^\sharp_*)^{'},\ti{\pa}_*^{'})}{J\in\DI(\DP^{'})}$ be two homology braids for the two Morse decompositions respectively. We say $\SH$ and $\SH^{'}$ are isomorphic if $\DP=\DP^{'}$ and there are isomorphisms $\GK_*(I):\CH_*(M_I)\to \CH_*(M_I^{'})$ for all $I\in \DI(\DP)$ such that for all $(I,J)\in \DI_2(\DP)$ the following diagram commutes:
$$
\begin{CD}
@>>> \CH_n(M_I) @>{i^\sharp_n}>> \CH_n(M_{I\!J}) @>{j^\sharp_n}>> \CH_n(M_J) @>{\ti{\pa}_n}>> \CH_{n-1}(M_I) @>>>  	\\
@.       @V{\GK_n(I)}VV        @V{\GK_n(I\!J)}VV           @V{\GK_n(J)}VV         @V{\GK_{n-1}(I)}VV  @.      \\
 @>>> \CH_n(M_I^{'}) @>{(i^\sharp_n)^{'}}>> \CH_n(M_{I\!J}^{'}) @>{(j^\sharp_n)^{'}}>> \CH_n(M_J^{'}) @>{\ti{\pa}_n^{'}}>> \CH_{n-1}(M_I^{'}) @>>> 
\end{CD}
$$
\end{dfn}

\section{Transition Matrix without Continuation}\label{Transition Matrix without Continuation}

\par In this section, we deal with a parameterized family of vector fields (ODEs) and study how its connecting orbit structure changes when a parameter varies. For such problems, the notion of transition matrix has been used (see Reineck \cite{Reineck1998connection} for the singular transition matrix, and McCord-Mischaikow \cite{C.McCordK.Mischaikow1992Connectedsimple} for the topological transition matrix) when a Morse decomposition continues over a parameter interval of interest.

\par However, if a Morse set undergoes a bifurcation at some point in the parameter interval, the continuation of Morse decompositions may break down at some parameter point because of the bifurcation, in which case, the conventional definition of the transition matrix does not apply.

\par As the main result of this paper, we extend the notion of the transition matrix in cases where the continuation of Morse decompositions breaks down at a parameter point. We mainly take the approach of Reineck \cite{Reineck1998connection} by studying a parameterized family of ODEs with a slow parameter drift and extend the notion of the singular transition matrix without assuming the continuation of Morse decompositions.

\subsection{The Basic Setting: Parameterized Family of Vector Fields (ODEs)}\label{setting}
\par We consider the following continuous-time dynamical system generated by the following systems of ordinary differential equations with a parameter $\lambda$:
\begin{equation}\label{ODEn}\tag{$*_\eps$}
\left\{
	\begin{aligned}
		&\dot{x}=f(x,\lambda) \\
		&\dot{\lambda}=g_{\eps}(\lambda)=\eps \lambda (\lambda-1),
	\end{aligned}
\qquad x\in \RR^d,\quad \lambda \in \Lambda:=(-2\delta, 1+2\delta),\quad\eps\geq0.	
\right.
\end{equation}
In (\ref{ODEn}), $f(x,\lam)$ is a Lipschitz continuous function on $\RR^d \times \Lam$ and $\delta$ is a small positive number.
\par There are two kinds of dynamical systems related to (\ref{ODEn}). 
\par Let the flow $\phi_t^{\lambda}$ be the continuous-time dynamical system defined on $\RR^d$ generated by the ODEs $\dot{x}=f(x,\lambda)$ with parameter $\lambda$. When we consider the family of flows $\sset{\phi_t^\lambda}$ with $\lam\in \Lam$, we call it a \bi{parameterized flow}.
\par We then add the slow parameter drift $\dot{\lambda}=g_{\eps}(\lambda)=\eps \lambda (\lambda-1)$ to the parameterized family of ODEs as in (\ref{ODEn}) above and consider it as an ODE system on $\RR^d\times \Lam$. Let the flow $\Phi_t^\eps$ be the continuous-time dynamical system defined on $\RR^d \times \Lambda$ generated by (\ref{ODEn}), which depends on the choice of $\eps$. We call $\Phi_t^\eps$ the \bi{extended slow-fast flow} for the parameterized family of ODEs $\dot{x}=f(x,\lambda)$ at $\eps$.

\par In particular, when $\eps=0$, we have another continuous-time dynamical system defined on the product space $\RR^d \times \Lam$. Let the flow $\Phi_t^{0}$ on $\RR^d \times \Lambda$ be generated by ($*_0$), namely $\Phi_t^{0}(x,\lambda)=(\phi_t^{\lambda}(x),\lambda)$.
\par For a subset $V \subset \RR^d \times \Lambda$, we denote the intersection of $V$ with the slice $\RR^d \times \{\lambda\}$ be ${V}^\lambda := {V} \cap (\RR^d \times \{\lambda\})$. When there is no risk of confusion, we regard $V^\lambda$ as the corresponding subset in $\RR^d$.
\par Now, we assume $N\subset \RR^d$ is an isolating neighborhood of the flow $\phi_t^{\lambda}$ for all $\lambda\in \Lambda$. Let $S^\lambda := \mathrm{Inv}(N,\phi_t^{\lambda})$ denote the maximal invariant set in $N$ under the flow $\phi_t^{\lambda}$. Let $(\DM(S^\lambda ),\DP_{\lambda})$ be a Morse decomposition for the invariant set $S^{\lambda}$ of the flow $\phi_t^{\lambda}$, in which the set of Morse sets at each parameter $\lam$ is given by $\DM(S^\lam)=\mset{M_p^\lambda}{p\in \DP_\lambda}$ with $\DP_\lam$ an order at slice $\lam$. We use $\DP^{\Gf}_\lam$ to denote the flow-defined order and use $\DP_\lam$ to denote an admissible order, namely, $\DP_\lambda$ is an extension of $\DP^\Gf_\lambda$.

\subsection{The Finest Decomposition of a Continuable Interval Pair}\label{The Finest Decomposition of a Continuable Interval Pair}
\par In this subsection, we discuss how the collection of Morse decompositions $\{(\MD{\lam},$ $\DP_\lambda)\}_{\lam\in \Lam}$ may be related with respect to the parameter $\lam$. Firstly, we define the \textbf{\textit{continuation of Morse decompositions over an interval}}. 
\begin{dfn}[Continuation]\label{continuation}
For an interval $I \subset \Lam$, any $\lam \in I$ and a collection of Morse decompositions $\{(\DM(S^\lam),\DP_\lam)\}_{\lam\in I}$ for each flow $\phi_t^\lambda$ 
where $\MD{\lambda}:=\mset{M_p^\lambda}{p\in\DP_\lambda}$, we say \textit{\textbf{the Morse decomposition continues over the interval}} $I\subset \Lam$ \textit{\textbf{for the flow}} $\phi_t^\lam$, if the following two statements hold.
\begin{enumerate}[label=\textnormal{(\arabic*).}]
     \item There is a partially ordered set $\DP$, such that $\DP_\lam=\DP$ for all $\lam \in I$.
      \item For every $p\in \DP$ and for any $\gamma \in I$, there is a relative open interval $O_\gamma \subset I$ (with respect to the relative topology of $I$) such that there exists a neighborhood $N_p(O_\gamma)$ which is an isolating neighborhood for $M_p^\lam$ under the flow $\phi_t^\lambda$ for all $\lam\in O_\gamma$.   
\end{enumerate}
\end{dfn}

\begin{apt}[Breakdown of continuation at one parameter point]\label{breakdown}
Let $N\subset \RR^d$ be an isolating neighborhood for each $\lam\in \Lam$ under the flow $\phi_t^\lam$. Let $S^\lam :=\mathrm{Inv}(N,\phi_t^\lam)$ be the maximal invariant set in $N$ under the flow $\phi_t^\lam$. Assume there is a Morse decomposition $(\MD{\lambda},\DP_\lambda)$ for each $\lam\in\Lam$ and assume there is a $\lam_0\in (0,1)$, such that Morse decompostions continue under the flow $\phi_t^\lam$ over $[0,\lam_0)$ and $(\lam_0,1]$, but not over $[0,1]$.
\end{apt}

\par Under Assumption \ref{breakdown}, we have two admissible order sets $\DP_{[0,\lam_0)}$ and $\DP_{(\lam_0,1]}$ with the fact $\DP_0=\DP_{[0,\lam_0)}$ and $\DP_1=\DP_{(\lam_0,1]}$. For simplicity, we write $\DP_{[0,\lam_0)}$ and $\DP_{(\lam_0,1]}$ as $\DP_0$ and $\DP_1$. Let $\DI(\DP)$ be the set of intervals in a partially ordered set $\DP$, i.e., $\DI(\DP):=\{J\subset \DP \mid J$  is an interval in  $\DP \}$.

\begin{dfn}[Local continuation]
Under Assumption \ref{breakdown}, for an interval pair $(J,J^{'})$, we say there is a \textbf{\textit{local continuation between}} $\DP_0$ and $\DP_1$ \textbf{\textit{from}} $J$ \textbf{\textit{to}} $J^{'}$, if it satisfies the following:
\begin{enumerate}[label=\text{(\arabic*).}]
    \item $J\in \DI(\DP_0), J^{'}\in \DI(\DP_1)$, with either interval $J$ is nonempty or interval $J^{'}$ is nonempty.

	\item There is $a\in [0,\lam_0)$, $b\in (\lam_0,1]$, and an isolating neighborhood $N(J,J^{'})\subset \RR^d$, such that $N(J,J^{'})$ is an isolating neighborhood under the flow $\phi_t^\lam$ for all $\lam\in [a,b]$, $N(J,J^{'})$ isolates the invariant set $M_J^{\lam}$ under the flow $\phi_t^\lam$ over the interval $[a, \lam_0)$, and $N(J,J^{'})$ isolates the invariant set $M_{J^{'}}^\lam$ for all $\lam\in (\lam_0, b]$.
\end{enumerate}

\par If there is a {local continuation between} $\DP_0$ and $\DP_1$ from $J$ to $J^{'}$, we call the pair $(J,J^{'})$ a \textbf{\textit{continuable interval pair}}.

\par Since we assume $N\subset \RR^d$ is an isolating neighborhood of the flow $\phi_t^\lambda$ for all $\lam\in \Lam$, the pair $(\DP_0,\DP_1)$ is obviously a continuable interval pair; hence it is called the \bi{trivial continuable interval pair}. 
 
\end{dfn}

\par Then, we decompose Morse decompositions into several continuable interval pairs.
\begin{dfn}[Decomposition]\label{decomposition}
Under Assumption \ref{breakdown}, we call the two collections of intervals $\DDJ_0:=\{J_i\in \DI(\DP_0)\mid i=1,\cdots,k \}$ and $\DDJ_1:=\{J_i^{'}\in \DI(\DP_1)\mid i=1,\cdots,k \}$ a \textbf{\textit{decomposition}} of the continuable interval pair $(\DP_0, \DP_1)$, if they satisfy the following conditions.
\begin{enumerate}[label=\textnormal{(\arabic*).}]
	\item  $J_u\cap J_v=\OO$ for all $J_u,J_v\in\DDJ_0$, $J_u^{'}\cap J_v^{'}=\OO$ for all $J_u^{'},J_v^{'}\in\DDJ_1$, and $$
        \sqcup_{i=1}^k J_i = \DP_0 \text{  and  } \sqcup_{i=1}^k J_i^{'} = \DP_1.
        $$
     \item $(J_i,J^{'}_i)$ is a continuable interval pair for all $i\in \sset{1,2,\cdots,k}$.
\end{enumerate}
We denote the \textbf{\textit{decomposition}} of the continuable interval pair $(\DP_0, \DP_1)$ as $\sset{(J_i,J_i^{'})}_{i\in A}$ with $A=\{1,2,\cdots,k\}$.
\end{dfn}

\par Suppose a decomposition $\sset{(J_i,J_i^{'})}_{i\in A}$ is given, we consider the homological Conley index of $\h{S}_\eps:=\text{Inv}(N\times \Lam,\Phi_t^\epsilon)$ and link it to the decomposition $\sset{(J_i,J_i^{'})}_{i\in A}$. From Definition \ref{decomposition}, we have the isolating neighborhoods $\h{N}(J_i,J_i^{'})\subset \RR^d\times \Lam$ that isolates $\mset{M_i^0\times\sset{0}}{i\in J}\cup \mset{M_j^1\times\sset{1}}{j\in J^{'}}$ under the flow $\Phi_t^\epsilon$ for all continuable interval pairs in the decomposition $\sset{(J_i,J_i^{'})}_{i\in A}$. In other words, $\mset{M_i^0\times\sset{0}}{i\in J}\cup \mset{M_j^1\times\sset{1}}{j\in J^{'}} \subset \text{Inv}(\h{N}(J_i,J_i^{'}),\Phi_t^\epsilon)$. We let $\h{S}_\epsilon(J_i,J_i^{'}):=\text{Inv}(\h{N}(J_i,J_i^{'}),\Phi_t^\epsilon)$, for all $i\in A$.
\begin{prop}[\cite{Mischaikow1988existenceofgeneralized}, Proposition $2.3$]
	The homological Conley index of the invariant set $\h{S}_\eps$ for the extended slow-fast flow $\Phi_t^\epsilon$, denoted as $C\!H_*(\h{S}_\eps)$, is trivial, namely $C\!H_*(\h{S}_\eps)=0$.
\end{prop}
\par This can also be applied to $\h{S}_\epsilon(J_i,J_i^{'})$ for all $i\in A$. 
\begin{prop}[Trivial homological Conley index of each continuable interval pair]\label{TrivialConleyIndexofEachContinuableIntervalPair}
	The homological Conley index of the invariant set $\h{S}_\epsilon(J_i,J_i^{'})$ for the extended slow-fast flow $\Phi_t^\epsilon$, denoted as $C\!H_*(\h{S}_\epsilon(J_i,J_i^{'}))$, is trivial, namely $C\!H_*(\h{S}_\epsilon(J_i,J_i^{'}))=0$ for all $i\in A$.
\end{prop}

\par Now we assume there is more than one decomposition of $(\DP_0, \DP_1)$ and show that their intersection is again a decomposition.
\begin{prop}[Intersection of two decompositions]\label{Reduced}
	Assume there are two decompositions of the continuable interval pair $(\DP_0, \DP_1)$ which are denoted as $\sset{(J_i,J_i^{'})}_{i\in A}$ and $\{(K_j,K_j^{'})\}_{j\in B}$. Then, we define the intersection of these two decompositions as the following:
$$
		\sset{(J_i,J_i^{'})}_{i\in A} \cap \{(K_j,K_j^{'})\}_{j\in B}:= \sset{(J_i\cap K_j, J^{'}_i\cap K_j^{'})}_{i\in A, j\in B}
$$
Futhermore, we take away these sets that are $J_i\cap K_j = J_i^{'}\cap K_j^{'} =\OO$, for all $i\in A, j\in B$. Taking away all such empty intersections, the set of remaining interval pairs denoted as $\sset{(Q_l,Q_l^{'})}_{l\in C}$, then $\sset{(Q_l,Q_l^{'})}_{l\in C}$ is also a decomposition of the continuable interval pair $(\DP_0, \DP_1)$.
\par We call $\sset{(Q_l,Q_l^{'})}_{l\in C}$ the \textbf{\textit{reduced intersection of decompositions $\sset{(J_i,J_i^{'})}_{i\in A}$ and $\sset{(K_j,K_j^{'})}_{j\in B}$}}.
\end{prop}
\begin{proof}
\par Take a pair $(V_{ij},V_{ij}^{'})$ with $V_{ij}=J_i\cap K_j$ and $V_{ij}^{'}=J_i ^{'}\cap K_j ^{'}$, in the reduced intersection of these two decompositions. Then either $V_{ij}\neq \OO$ or $V_{ij}^{'}\neq \OO $, and both $V_{ij}$ and $V_{ij}^{'}$ are intervals.
\par We first check that there is a local continuation from $V_{ij}$ to $V_{ij}^{'}$. Because there is a local continuation from $J_i$ to $J_i^{'}$, and from $K_j$ to $K_j^{'}$, there are two isolating neighborhood $N(J_i)$ and $N(K_j)$, and two intervals $[a_J,b_J],[a_K,b_K]\subset [0,1]$, with $a_J<\lz <b_J; a_K<\lz <b_K$, letting $a_{J\!K}:=\max{(a_J,a_K)}; b_{J\!K}:=\min{(b_J,b_k)}$, satisfy the following:
\begin{enumerate}[label=\textnormal{\arabic*).}]
      \item $N(J_i)$ isolates $M^\lam_{J_{i}}$ under the flow $\phi_t^\lam$ for all $\lam\in [a_J,\lam_0)$, and $N(J_i)$ isolates $M^\lam_{J_{i}^{'}}$ under the flow $\phi_t^\lam$ for all $\lam\in (\lam_0, b_J]$;
      \item $N(J_i)$  is an isolating neighborhood under the flow $\phi_t^\lam$ for all $\lam\in [a_J, b_J]$;
      \item $N(K_j)$ isolates $M^\lam_{K_{j}}$ under the flow $\phi_t^\lam$ for all $\lam\in [a_K,\lam_0)$, and $N(K_j)$ isolates $M^\lam_{K_{j}^{'}}$ under the flow $\phi_t^\lam$ for all $\lam\in (\lam_0, b_K]$;
      \item $N(K_j)$  is an isolating neighborhood under the flow $\phi_t^\lam$ for all $\lam\in [a_K, b_K]$.
\end{enumerate}
\par We let $N(V_{ij}):=N(J_i)\cap N(K_j)$, and we check that $N(V_{ij})$ isolates $M^\lam_{V_{ij}}$ for $\lam \in [a_{J\!K},\lam_0)$  and isolates $M^\lam_{V^{'}_{ij}}$ for $\lam \in (\lam_0, b_{J\!K}] $ under the flow $\phi_t^\lam$.
\par Because $M^\lam_{J_{i}}=\mathrm{Inv}({N(J_i))}$ and $M^\lam_{K_{j}}=\mathrm{Inv}( N(K_j))$ for all $\lam \in [a_{J\!K},\lam_0)$, we have ${M^{\lam}_{V_{ij}}} \subset \inte{(N(J_i))}$ and ${M^{\lam}_{V_{ij}}} \subset \inte{(N(K_j))}$ for $\lam \in [a_{J\!K},\lam_0)$. Thus, we have ${M^{\lam}_{V_{ij}}}\subset \mathrm{Inv}(N(V_{ij}))$ for all $\lam \in [a_{J\!K},\lam_0)$. Now we prove $ \mathrm{Inv}(N(V_{ij})) \subset {M^{\lam}_{V_{ij}}}$ for all $\lam \in [a_{J\!K},\lam_0)$. Suppose it does not hold, then there is a $p\in \DP_0-V_{ij}$, such that $M_p^\lam \subset \mathrm{Inv}(N(V_{ij}))$. But this is contradicted to $M^\lam_{J_{i}}$ and $M^\lam_{K_{j}}$ are the maximal invariant sets in $N(J_i)$ and $N(K_j)$ respectively. Therefore, we have the fact that $\mathrm{Inv}(N(V_{ij})) = {M^{\lam}_{V_{ij}}}$ for all $\lam \in [a_{J\!K},\lam_0)$. Similarly, we have the result $M^\lam_{V^{'}_{ij}}=\mathrm{Inv}(N(V_{ij}))$ under the flow $\phi_t^\lam$ for all $\lam \in ({\lam_0,b_{J\!K}}]$. 
\par At the slice $\lam= \lam_0$, $N(J_i)$ and $N(K_j)$ isolate some invariant sets and the connecting orbits between them under the flow $\phi_t^{\lam_0}$. Thus, we can also use a similar way to prove that $N(V_{ij})$ is an isolating neighborhood for $\lam= \lam_0$ under the flow $\phi_t^{\lam_0}$. As a result, we have the result that there is a local continuation from $V_{ij}$ to $V_{ij}^{'}$.
\par We check:
$\sqcup_{i\in A, j\in B}V_{ij}=\DP_0$ and $\sqcup_{i\in A, j\in B}V^{'}_{ij}=\DP_1$. 
Firstly, for $i_1\neq i_2$ or $j_1\neq j_2$, we have $V_{i_1j_1}\cap V_{i_2j_2}= \OO$ by $J_{i_{1}}\cap J_{i_{2}}= \OO$ and $K_{j_{1}}\cap K_{j_{2}}= \OO$. Appearently, $\cup_{i\in A, j\in B}V_{ij}=\DP_0$ and $\cup_{i\in A, j\in B}V^{'}_{ij}=\DP_1$. Thus, $\sqcup_{i\in A, j\in B}V_{ij}=\DP_0$ and $\sqcup_{i\in A, j\in B}V^{'}_{ij}=\DP_1$.
\par Finally, we consider the reduced intersection. Because we only take away empty sets for both sides, the result above also holds. Therefore, the reduced intersection of two decompositions is also a decomposition of the continuable interval pair $(\DP_0, \DP_1)$.
\end{proof}
\par By the definition of local continuation, we can decompose Morse decompositions into several continuable interval pairs that have a local continuation in each pair. In order to find the finest such decomposition, we define the \textbf{\textit{indecomposability}} of continuable interval pairs.

\begin{dfn}[Indecomposability]\label{LTD}
Under Assumption \ref{breakdown}, a continuable interval pair $(J,J^{'})$ is called \bi{indecomposable} if there are no strict subintervals $K\subsetneqq J$ and $K^{'}\subsetneqq J^{'}$ such that $(K, K^{'})$ is a continuable interval pair.
\end{dfn}


\begin{thm}[Existence and uniqueness of the finest decomposition]\label{finest}
	Under Assumption \ref{breakdown}, there is a unique decomposition of the continuable interval pair $(\DP_0, \DP_1)$ denoted as $\sset{(J_i,J_i^{'})}_{i\in A}$, such that the continuable pair $(J_i, J_i^{'})$ is indecomposable for all $i\in A$. We call the decomposition $\sset{(J_i,J_i^{'})}_{i\in A}$ as the \textit{\textbf{finest decomposition}} of the continuable interval pair $(\DP_0, \DP_1)$.
\end{thm}
\begin{proof}
	\par We first prove the existence, which is to prove that there is a decomposition of the continuable interval pair $(\DP_0, \DP_1)$, denoted as $\sset{(J_i,J_i^{'})}_{i\in A}$, such that the continuable pair $(J_i, J_i^{'})$ is indecomposable for all $i\in A$.
	\par Suppose for any $i\in A$, if the continuable pair $(J_i, J_i^{'})$ is not indecomposable, then there is a $I\subset J_i$ and $I^{'}\subset J_i^{'}$, with either $I\neq \OO$ or $I^{'}\neq \OO$ and either $I\neq J_i$ or $I^{'}\neq J_i^{'}$, such that there is a local continuation from $I$ to $I^{'}$.
	\par By the fact that $J_i,J_i^{'},I$ and $I^{'}$ are intervals, we have $J_i-I$ and $J_i^{'}-I^{'}$ are also intervals. We use the same discussion as in Proposition \ref{Reduced}, there is a local continuation from $J_i-I$ to $J_i^{'}-I^{'}$. Therefore, we have a finer decomposition of the original one. By the finiteness of a Morse decomposition, we can get a decomposition of $(\DP_0, \DP_1)$, denoted as $\sset{(K_j,K_j^{'})}_{j\in B}$, such that the continuable pair $(K_j, K_j^{'})$ is indecomposable for all $j\in B$.
	\par Now we check the uniqueness of the finest decomposition by contradiction. Suppose that we have two finest decompositions are $\sset{(U_i,U_i^{'})}_{i\in A}$ and $\sset{(V_j,V_j^{'})}_{j\in B}$. Then, the reduced intersection of these two decompositions is also a decomposition. From the proof of Proposition \ref{Reduced}, there is a $i_0\in A$ such that there are two intervals $L\subset U_{i_{0}}$ and $L^{'}\subset U_{i_{0}}^{'}$ with with either $L\neq \OO$ or $L^{'}\neq \OO$ and either $L\neq U_{i_0}$ or $L^{'}\neq U_{i_0}^{'}$, satisfying that there is a local continuation from $L$ to $L^{'}$ which contradicted to the indecomposability of continuable pairs in $\sset{(U_i,U_i^{'})}_{i\in A}$.
\end{proof}
\begin{ex}
\par Here, we present an example of the preceding theorem. Figure \ref{theFinestDecomposition} is a bifurcation diagram of a one-parameter family of one-dimensional ODEs over the parameter interval $\Lam=(-2\delta, 1+2\delta)\subset \RR$, in which stable fixed points are denoted by solid lines and the unstable ones are denoted by the dotted lines. Here $0<\lz<1$ is the unique bifurcation point in $\Lam$, and hence Assumption \ref{breakdown} is satisfied. The arrows in the light grey denote the phase portrait of each slice. Let $\DP_0=\sset{1^0,2^0,3^0,4^0,5^0}$ and $\DP_1=\sset{1^1,2^1,3^1,4^1,5^1}$ be the index sets of corresponding Morse decomposition as shown in Figure \ref{theFinestDecomposition}. Then finest decomposition of the pair $(\DP_0, \DP_1)$ is given by 

$$\sset{(\sset{1^0},\sset{1^1}),(\sset{2^0},\sset{2^1}),(\sset{3^0,4^0},\OO),(\sset{5^0},\sset{3^1, 4^1, 5^1})}
$$

\begin{figure}[H] 
\centering 
\includegraphics[width=1 \textwidth]{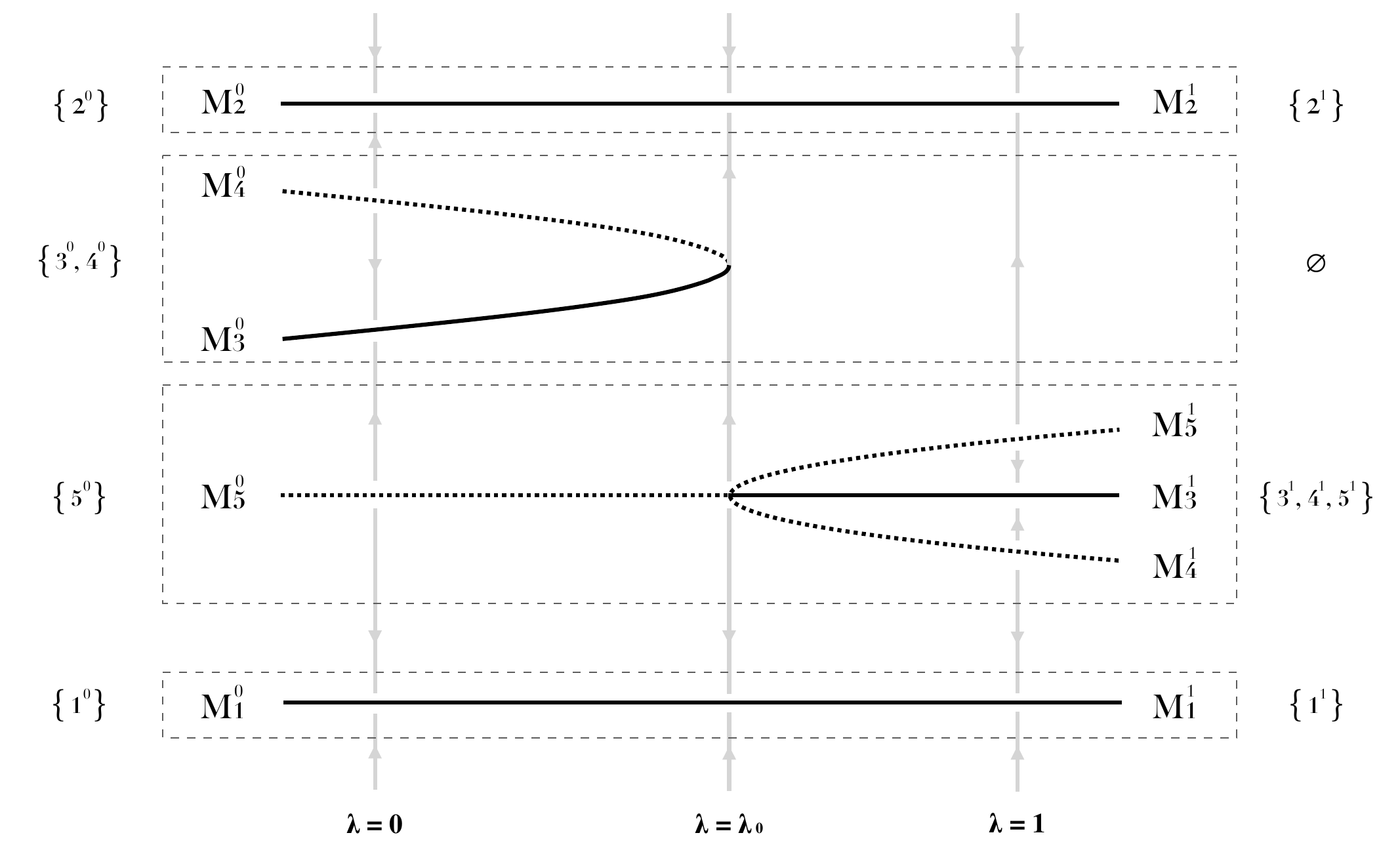} 
\caption{An example of the finest decomposition} 
\label{theFinestDecomposition} 
\end{figure}
\end{ex}
\par Back to the theorem above, we have the finest decomposition $\sset{(J_i,J_i^{'})}_{i\in A}$ of the continuable interval pair $(\DP_0, \DP_1)$ under Assumption \ref{breakdown}. We note here that the collections of invariant sets with admissible orders $(\DM(S^{\lam};\{J_i\}_{i\in A}\}),\DP(\{J_i\}_{i\in A}\})):=(\{M^{\lam}_{I} \mid I\in \{J_i\}_{i\in A}\},\DP(\{J_i\}_{i\in A}\}))$, with $\DP(\{J_i\}_{i\in A}\})$ an admissible order of $\{M^{\lam}_{I} \mid I\in \{J_i\}_{i\in A}\}$ is also a Morse decomposition for all $\lambda \in [0,\lam_0)$. At the slice $\lam = \lam_0$, if we denote the Morse decomposition as $(\DM(S^{\lam_0}),\DP_{\lam_0})$ and consider the isolating neighborhoods  $\{N(J_i)\}_{i\in A}$ crossing $\lam _0$. Let $\DP_{\lam_0}(N(J_i)):=\{\pi \in \DP_{\lam_0} \mid M_{\pi}^{\lam_0} \cap N(J_i)\neq \OO \}$ then each $\DP_{\lam_0}(N(J_i))$ may not always be an interval in $\DP_{\lam_0}$. As a result, the collection of invariant sets $\{M_{ \DP_{\lam_0}(N(J_i)) }^{\lam_0} \}_{i\in A}$ may not always be a Morse decomposition for $\lam = \lam_0$. 
\par If there is a partially ordered set $\DP$ which $\DP$ is an extension for $\DP_0$ and $\DP_1$, and if $\{M_{ \DP(N(p)) }^{\lam_0} \}_{p \in \DP_0}$ is a Morse decomposition for $\lam = \lam_0$, then the Morse decomposition conintues over $[0,1]$, thus we can directly apply the result of Reineck \cite{Reineck1998connection}. If there is not such a $\DP$ nor $\{M_{ \DP(N(p)) }^{\lam_0} \}_{p \in \DP_0}$ is a Morse decomposition for $\lam = \lam_0$, we can still use the finest decomposition to extend the result into continuable interval pairs, which we will see in Section \ref{CO}.

\subsection{The Hausdorff limit of Morse Decompositions}\label{limit}
\par In this section, we use the Hausdorff metric to define the Hausdorff limit of a Morse set under the flow $\phi_t^\lambda$ when $\lambda$ limits to some point in the parameter space. 

\par Let $(X,d)$ be a compact metric space with the metric $d(\cdot , \cdot):X\times X \to \RR$, and $\GC(X):=\{C\subset X\mid C  \text{ is closed } \}$ be the set of all closed subsets in $X$. Let $\rho(\cdot , \cdot):\GC(X) \times \GC(X) \to \RR$ be the \textbf{\textit{Hausdorff metric}} on $\GC(X)$. We have the following well-known facts of the Hausdorff metric.

\begin{prop}[cf. \cite{Reineck1998connection}, Lamma $1.1$]\label{HausdorffLimit}
	Suppose the sequence of closed sets $\{C_n\}_{n\in \NN}\subset \GC(X)$ limit to $C$ in the Hausdorff metric, then:
	$$
	\DC=\{ x\in X
	\mid
	\text{there is a sequence $x_n \in C_n$ such that $x_n \to x$, as $n \to \infty$}
	\}.
	$$
\end{prop}

\begin{prop}[cf. \cite{Reineck1998connection}, Lamma $1.2$]\label{connectedness}
For a connected sequence of subsets $\{C_n\}_{n\in\NN}\subset \GC(X)$, and $\{C_n\}_{n\in\NN}$ limits to $C$ in Hausdorff metric, then $C$ is also connected.
\end{prop}

\par Now we consider the Morse decomposition $(\DM(S^\lam),\DP_\lam)=(\{M_p^{\lam} \mid p\in \DP_\lam \},\DP_\lam)$ of $S^\lam$ for all $\lam\in [0,1]$. Under Assumption \ref{breakdown}, we consider the admissible orders $\DP_0$ and $\DP_1$, and consider the Morse decomposition $(\DM(S^\lz),\DP_\lz)$ at the slice $\lam = \lz$.
We take any sequence $\{\lam_{n,-}\}_{n\in \NN}\subset [0,\lam_0)$ with $\lam_{n,-} \to \lam_{0}$, or $\{\lam_{n,+}\}_{n\in \NN}\subset (\lam_0, 1]$ with $\lam_{n,+} \to \lam_{0}$, and consider the sequence of Morse sets $\{M_p^{\lam_{n,-}}\}_{n\in \NN}$ with $p\in \DP_0$ and $\{M_p^{\lam_{n,+}}\}_{n\in \NN}$ with $p\in \DP_1$. Because the space $\GC(N\times [0,1])$ with topology generated by Hausdorff metric is still compact, there is a convergent subsequence of $\{M_p^{\lam_{n,-}}\}_{n\in \NN}$ with $p\in \DP_0$ and $\{M_p^{\lam_{n,+}} \}_{n\in \NN}$ with $p\in \DP_1$ in Hausdorff metric.
\par The following statements hold for any choice of sequence $\{\lam_{n,-}\}_{n\in \NN}\subset [0,\lam_0)$ with $\lam_{n,-} \to \lam_{0}$, or $\{\lam_{n,+}\}_{n\in \NN}\subset (\lam_0, 1]$ with $\lam_{n,+} \to \lam_{0}$ and any choice of convergent subsequences of $\{M_p^{\lam_{n,-}}\}_{n\in \NN}$ with $p\in \DP_0$ and $\{M_p^{\lam_{n,+}}\}_{n\in \NN}$ with $p\in \DP_1$ respectively. For simplicity, we denote the Hausdorff limit of any convergent subsequence of $\sset{M_p^{\lam_n}}$ as $\lim_{\lam_0^{-}} M_p^{\lambda_n} $ for $p\in \DP_0$ and as $\lim_{\lam_0^{+}}M_p^{\lambda_n} $ for $p\in \DP_1$. If the set that the index $p$ belongs to is clear, we denote the Hausdorff limit as $\lim_{\lam_0} M_p^{\lambda_n}$ for short.

\begin{thm}	\label{limitofset}
	For any $p\in \DP_0$ or $\DP_1$, the Hausdorff limit $\lim_{\lam_0} M_p^{\lam_n}$ is compact in $N$ and invariant under the flow $\phi_t^{\lam_0}$.
\end{thm}
\begin{proof} The compactness is obtained by the closedness of the Hausdorff limit $\lim_{\lam_0} M_p^{\lam_n}$ in the compact set $N$.
\par We show the property of invariance by contradiction. Suppose the limit set $\lim_{\lam_0} M_p^{\lam_n}$ is not invariant, then there is a point $x\in \lim_{\lam_0} M_p^{\lam_n}$, a real number $\delta >0$, and a time $t_0>0$, such that $d(\phi_{t_0}^{\lam_0}(x), \lim_{\lam_0} M_p^{\lam_n})=\delta$. By the uniformly continuity of $\phi_t^{\lam_0}$, there is a $\eta >0$ such that $d(\phi_{t_0}^{\lam_0}(x), \phi_{t_0}^{\lam_0}(y))<\frac{\delta}{3}$ for any $y\in X$ with $d(x, y)<\eta$. 
\par For sufficiently large $n$, we have $d(\phi_{t_0}^{\lam_0}(x), \phi_{t_0}^{\lam_n}(x))< \frac{\delta}{3}$ for any $x\in X$. Also by $M_p^{\lam_n}\to \lim_{\lam_0} M_p^{\lam_n}$, we have $\rho(M_p^{\lam_n}, \lim_{\lam_0} M_p^{\lam_n})< \min \{\frac{\delta}{3}, \eta \}$, if $n$ is sufficiently large.
\par Now we choose a sufficiently large $n$, and by $M_p^{\lam_n}\to \lim_{\lam_0} M_p^{\lam_n}$, we can choose $y\in M_p^{\lam_n}$ with $d(x,y)< \eta$.
\par Then, we have:
\begin{align*}
d(\phi_{t_0}^{\lam_0}(x), \lim_{\lam_0} M_p^{\lam_n}) &\leq	d(\phi_{t_0}^{\lam_0}(x), \phi_{t_0}^{\lam_0}(y)) + d(\phi_{t_0}^{\lam_0}(y), \phi_{t_0}^{\lam_n}(y)) + d(\phi_{t_0}^{\lam_n}(y), M_p^{\lam_n})\\
&\quad +\rho(M_p^{\lam_n}, \lim_{\lam_0} M_p^{\lam_n})\\
&<\frac{\delta}{3} + \frac{\delta}{3} + \frac{\delta}{3} = \delta.
\end{align*}
But this contradicts to $d(\phi_{t_0}^{\lam_0}(x), \lim_{\lam_0} M_p^{\lam_n})=\delta$. Therefore, the Hausdorff limit set $\lim_{\lam_0} M_p^{\lam_n}$ is compact and invariant under the flow $\phi_t^{\lam_0}$.
\end{proof}

\begin{cor}
	For any interval $J\in \DI(\DP_0)$ or $\DI(\DP_1)$, the Hausdorff limit of the invariant set $M_J^{\lam_n}$, denoted as $\lim _{\lam_0}M_J^{\lam_n}$, is also invariant under $\phi_t^{\lam_0}$.
\end{cor}

\par Let $(\DM(S^{\lam_0}),(\DP_{\lam_0}, <_{\lam_0}))$ be a Morse decomposition at the slice $\lam=\lam_0$ with the admissible order $(\DP_{\lam_0}, <_{\lam_0})$. Since the Hausdorff limit $\lim_{\lam_0} M_p^{\lam_n}$ is invariant under the flow $\phi_t^{\lam_0}$, it is contained in some Morse sets and the connecting orbits between them.

\begin{cor}\label{limitint}
For any $p\in \DP_0$ or $\DP_1$, there is an interval $J_p \in \DI(\DP_{\lam_0})$ such that $\lim_{\lam_0} M_p^{\lam_n}\subset M_{J_p}^{\lam_0}$.	
\end{cor}
\par However, the interval $J_p \in \DI(\DP_{\lam_0})$ may be too coarse to measure the Hausdorff limit  $\lim_{\lam_0} M_p^{\lam_n}$ in an efficient way, hence we present the following notations.

\begin{dfn}
We denote the intersection of the Hausdorff limit $\lim_{\lam_0} M_p^{\lam_n}$ and the admissible order set $\DP_{\lam_0}$ as,
$$
\DP_{\lam_0}(\lim_{\lam_0} M_p^{\lam_n}):=\{ \pi \in \DP_{\lam_0} \mid \lim_{\lam_0} M_p^{\lam_n} \cap M_\pi^{\lam_0} \neq \OO \}. 
$$
\end{dfn}


\par The set $\DP_{\lam_0}(\lim_{\lam_0} M_p^{\lam_n})$ may not always be an interval in $\DI(\DP_{\lam_0})$, even under a stricter condition as follows.

\begin{apt}\label{oneside}
For a Morse decomposition $(\DM(S^{\lam_0}), \DP_{\lam_0})$ at $\lam = \lam_0$ and for any $p\in \DP_{\lam_0}$, there is a $q\in \DP_0$ or $q\in \DP_1$, such that $\lim_{\lam_0} M_q^{\lam_n} \cap M_p^{\lam_0} \neq \OO$.
\end{apt}
\par The assumption above requires that each Morse set in the Morse decomposition at slice $\lam_0$ either contains a Hausdorff limit of some Morse set from the left side or from the right side.


For a subset $\DQ\subset \DP$, with $\DP$ an admissible order for a Morse decomposition $(\DM(S),\DP)$, we define:
$$
\GM_\DQ:=\{ M_\pi 
\mid
\pi \in \DQ 
\} \quad
\cup \quad 
\cup_{\pi,\theta \in \DQ } 
\DC(M_\pi, M_\theta).
$$
with $\DC(M_\pi, M_\theta):=\{x\in X \mid \alpha(x)\subset M_\theta \text{ and } \omega (x)\in M_\pi \}$, the set of connections from $M_\theta$ to $M_\pi$.
We denote the set $\GM_{\DP_{\lam_0}(\lim_{\lam_0} M_p^{\lam_n})}^{\lam_0}$ as $\GM^{\lam_0}_{\DP(\lim M_p)}$ for simplicity.

\begin{note}
$\GM_\DQ$ defined above is invariant but since $\DQ$ may not be an interval in $\DP$, $\GM_\DQ$ is not always closed.	
\end{note}

\begin{prop}
	$\lim_{\lam_0} M_p^{\lam_n} \subset \GM^{\lam_0}_{\DP(\lim M_p)}$.
\end{prop}

\par Now, we consider Assumption \ref{breakdown}, and by Theorem \ref{finest},  there is a finest decomposition of $(\DP_0, \DP_1)$ denoted as $\sset{(J_i,J_i^{'})}_{i\in A}$. Then we have a collection of mutually disjoint isolating neighborhoods $\{ N({J_i},{J_i^{'}})\}_{i\in A}$ such that $N({J_i},{J_i^{'}})$ is an isolating neighborhood under the flow $\phi_t^\lam$ over some interval $\lam \in [a_i,b_i]$ for all $i\in A$. $ N({J_i},{J_i^{'}}) $ isolates $M_{J_i}^\lam $ under the flow $\phi_t^\lam$ over the interval $\lam \in [a_i,\lam_0)$, and isolates $M_{J_i^{'}}^\lam $ under the flow $\phi_t^\lam$ over the interval $\lam \in (\lam_0,b_i]$. We let $\DP_{\lam_0}(N({J_i},{J_i^{'}})):=\{\pi \in  \DP_{\lam_0} \mid M_\pi^{\lam_0} \subset N({J_i},{J_i^{'}}) \}$ for all $i\in A$. Then $\GM^{\lam_0}_{\DP_{\lam_0}(N({J_i},{J_i^{'}}))}=\inv (N({J_i},{J_i^{'}}),\phi_t^\lz)$ for all $i\in A$.
\begin{prop}\label{decompositionlimit}
 Under Assumption \ref{breakdown}, we have for all $i \in A$:
 \begin{align*}
 	 &\lim_{\lam_0} M_{J_i}^{\lam_n} \subset \GM^{\lam_0}_{\DP_{\lam_0}(N({J_i},{J_i^{'}}))},\\
 	 &\lim_{\lam_0} M_{J_i^{'}}^{\lam_n} \subset \GM^{\lam_0}_{\DP_{\lam_0}(N({J_i},{J_i^{'}}))}.
 \end{align*}	
\end{prop}
\begin{proof}
By Corollary \ref{limitint}, $\lim_{\lam_0} M_{J_i}^{\lam_n}$ and $\lim_{\lam_0} M_{J_i^{'}}^{\lam_n}$ are invariant under the flow $\phi_t^{\lam_0}$ for all $i\in A$. Because $M_{J_i}^{\lam}\subset N({J_i},{J_i^{'}})$ over $\lam \in [a,\lam_0)$, for some $a<\lam_0$, and by Proposition \ref{HausdorffLimit} and the closedness of $N({J_i},{J_i^{'}})$, we have $\lim_{\lam_0} M_{J_i}^{\lam_n} \subset N({J_i},{J_i^{'}})$. The compact set $N({J_i},{J_i^{'}})$ is an isolating neighborhood under the flow $\phi_t^{\lam_0}$ and $\lim_{\lam_0} M_{J_i}^{\lam_n}$ is invariant under the flow $\phi_t^{\lam_0}$, we have $\lim_{\lam_0} M_{J_i}^{\lam_n}\subset \inv (N({J_i},{J_i^{'}}),\phi_t^\lz) = \GM^{\lam_0}_{\DP_{\lam_0}(N({J_i},{J_i^{'}}))}$. Similarly for $\lim_{\lam_0} M_{J_i^{'}}^{\lam_n}$, we have $\lim_{\lam_0} M_{J_i^{'}}^{\lam_n} \subset \GM^{\lam_0}_{\DP_{\lam_0}(N({J_i},{J_i^{'}}))}$.
\end{proof}

\par We should note here that $\cup_{i\in A}	\DP_{\lam_0}(N({J_i},{J_i^{'}}))$ may not always be $ \DP_{\lam_0}$, but the following holds under Assumption \ref{oneside}.
\begin{prop}
Under Assumption \ref{breakdown} and Assumption \ref{oneside}, we have
$$ \cup_{i\in A} \DP_{\lam_0}(N({J_i},{J_i^{'}})) = \DP_{\lam_0}.$$ 
\end{prop}
\begin{proof} By definition, $\cup_{i\in A} \DP_{\lam_0}(N({J_i},{J_i^{'}})) \subset \DP_{\lam_0}$ is evident. We prove the inverse inclusion. If there is a $p \in  \DP_{\lam_0}$ such that $p \notin \cup_{i\in A} \DP_{\lam_0}(N({J_i},{J_i^{'}}))$, then there is an isolating neighborhood $N(M_p^{\lam_0})$ isolating $M_p^{\lam_0}$ under the flow $\phi_t^{\lam_0}$, satisfying $N(M_p^{\lam_0})$ isolates $\OO$ under the flow $\phi_t^{\lam}$ over intervals $[a,\lam_0)$ and $(\lam_0, b]$, for some $a$ and $b$ with $a<\lam_0 <b$. But this contradicts Assumption \ref{oneside}, and we get the inverse inclusion.
\end{proof}
\begin{note}
The collection of mutually disjoint invariant sets $\{\GM^{\lam_0}_{\DP_{\lam_0}(N({J_i},{J_i^{'}}))} \mid i\in A\}$ may not always be a Morse decomposition for $S^{\lam_0}$.
\end{note}

\par In the next section, we can see that the Theorem $3.13$ in Reineck \cite{Reineck1998connection} could be directly extended between continuable interval pairs, even if there is no partially ordered set $\DP$ that extends $\DP_0$ and $\DP_1$, and even under the situation  $\{\GM^{\lam_0}_{\DP_{\lam_0}(N({J_i},{J_i^{'}}))} \mid i\in A\}$ is not a Morse decomposition for $S^{\lam_0}$.


\subsection{Connecting Orbits}\label{CO}
\par In this section, we consider the continuous-time dynamical system $\Phi_t^{\epsilon}$ generated by (\ref{ODEn}) with sufficiently small $\epsilon>0$. We suppose the existence of a connecting orbit between two specified Morse sets $M_\rho^1$ and $M_\pi^0$ of a Morse decomposition of $\Phi_t^\epsilon$ for all $\epsilon$, as $\epsilon \to 0$. We consider Hausdorff limits of these connecting orbits as $\epsilon \to 0$ and extend Theorem $3.13$ in Reineck \cite{Reineck1998connection} in the situation that the continuation of Morse decompositions over $[0,1]$ does not exist. We first extend it for continuable interval pairs and then extend it to a more general situation under Assumption \ref{breakdown}. Firstly, we shall assume the existence of a connecting orbit from $M_\rho^1$ to $M_\pi^0$ for all $\epsilon$ with $\eps \to 0$.
\begin{apt}\label{connexist}
Assume there exist $\eps_0\in (0,1)$, $\rho\in \DP_1$ and $\pi\in \DP_0$ such that for all positive $\eps < \eps_0$, there is a connecting orbit $c_\eps$ from $M_\rho^1$ to $M_\pi^0$ under the flow $\Phi_t^{\epsilon}$.	
\end{apt}

\par Under Assumption \ref{connexist}, we let: $\ol{c}_\epsilon:=\cl(c_\epsilon)=c_\epsilon \cup \omega(c_\epsilon) \cup \alpha(c_\epsilon)$. For any sequence $\{\eps_n\}_{n\in \NN}$ with $\eps_n\to 0$ as $n\to \infty$, by the compactness of $\GC(N\times [0,1])$, there is a convergent subsequence of $\{ \ol{c}_{\eps_n} \}$ in Hausdorff metric. We denote the Hausdorff limit of the subsequence $\{\ol{c}_{\epsilon_{n_k}}\}_{k\in \NN}$ as $\ol{c}^{\{\eps_{n_k}\}}$. Because the following results hold for any choice of $\{\eps_n\}_{n\in \NN}$ and any convergent subsequence of $\{\ol{c}_{\eps_n}\}$, we denote a convergent subsequence also as $\cc_n$, and the Hausdorff limit as $\ol{c}$, with $\cc_n \to \cc$ as $n\to \infty$ for simplicity.

\par Similarly to the proof of Theorem \ref{limitofset} and by Lemma \ref{connectedness}, we have the invariance and the connectedness of the limit set $\cc$.

\begin{lma}\label{invariantofconnection}
$\cc \subset N\times [0,1] $ is compact, connected, and invariant under the flow $\Phi_t^0$.
\end{lma}
\newcommand{\ccl}[0]{\cc ^\lam}
\par We now take the Hausdorff limit of the connecting orbit $\cc$ and intersect it with the slice $\lambda$, denoted as $\ccl :=\cc \cap (\RR^d\times \{\lam\}).$ 
\begin{prop}
	For $\lam\in [0,1]$, $\ccl$ is a non-empty set which is compact, connected, and invariant under the flow $\Phi_t^0$.
\end{prop}
\begin{proof}
Because $\cc$ is not an empty set and $\cc$ is connected, $\ccl$ is also non-empty. The compactness, connectedness, and invariance under the flow $\Phi_t^0$ of $\ccl$ follow directly from the compactness, connectedness, and invariance under the flow $\Phi_t^0$ of $\cc$ in $\pspa$.
\end{proof}

\par Therefore, like $\HHL{p}$, let $(\MD{\lambda},\DP_\lambda)$ denote the Morse decomposition for $S^\lam$, then $\ccl$ either intersects only one Morse set or multiple Morse sets. In the former case, $\ccl$ will be contained in a single Morse set, and in the latter case, it will be contained in these multiple Morse sets and a connecting orbit between them. Let $$I^\lambda:=\{p\in \DP_\lam \mid \ccl \cap M_p^\lam \neq \OO \}$$ for all $\lam\in [0,1]$. Under Assumption \ref{breakdown}, we only need to consider $\DP_\lam=\DP_0$, $\DP_\lam=\DP_1$ and $\DP_\lam=\DP_\lz$. 
\par Because $\ccl$ is compact and invariant under the flow $\phi_t^\lam$, $\ccl$ must intersect with a Morse set in $\MD{\lam}$. As a result, $I^\lam \neq \OO$. Moreover, $I^\lam$ is shown to be totally ordered under the flow-defined order $\FO{\lambda}$.

\begin{lma}\label{orderpreserve} 
For $\for{n}$, let $\cc_n$ be a connecting orbit under the flow $\Phi_t^{\eps_n}$. Let $\cc_n\to \cc$ as $\eps_n \to 0$.
Suppose there is a sequence  $\sset{(x_n, \lam_n)}_{\for{n}}$ with $(x_n, \lam_n)\in \cc_n$ and $(x_n, \lam_n)\to (x,\lam)\in M_p^\lam$ for some $p\in \DP_\lam$, and suppose there is a sequence of time $\sset{t_n}_{\for{n}}$ with $t_n \geq 0$ such that $\Phi_{t_n}^{\eps_n}(x_n, \lam_n) = (y_n, \mu_n)$, and $(y_n, \mu_n)$ converges to a point $(y,\lam)$ in the same slice $\lam$, namely, $(y_n, \mu_n)\to (y, \lam)\in M_{q}^\lam$ for some $q\in \DP_\lam$. Then we have $q \leq_\lam^{\Gf} p$ under the flow-defined order $\FO{\lam}$. In particular, if $q\neq p$, then $q <_\lam^{\Gf} p$.
\end{lma}
\begin{proof}
 Firstly, if the time sequence $\sset{t_n}_{n\in \NN}$ is bounded, there is a convergent subsequence $t_{n_k}\to t_0$,  then we take the limit as $k\to \infty$ and get the result $\Phi_{t_0}^0(x,\lambda)=(y,\lambda)$. From this, we know that $p =_\lam^{\Gf} q$ in the flow-defined order $\FO{\lambda}$.
 \par In the situation that the time sequence $\sset{t_n}_{n\in \NN}$ is not bounded, we show it by contradiction. When $t_n \to \infty $, we suppose that $ p \nleq_\lambda^\Gf q $. We collect all Morse sets that are less than $p$ in the flow-defined order, letting:
 $$
 I_{p}:=\mset{\pi \in \FO{\lam}}{\pi \leq_\lambda^\Gf p}.
 $$
 Then $I_p$ is an attracting interval in the flow-defined order $\FO{\lam}$, and the invariant set $M_{I_p}^\lambda$ is an attractor under the flow $\phi_t^\lambda$. Let $U\subset \RR^d$ is an isolating neighborhood under the flow $\phi_t^\lam$ isolating $M_{I_p}^\lambda$, and let $\hat{U}:=U\times [0,1]$. Then for sufficiently large $n$, we have $(x_n, \lam_n)\in \hat{U}$, but $(y_n, \mu_n)\notin \cl(\hat{U})$, by our supposition. As a result, we can find a sequence of time $s_n\in (0, t_n)$ such that:
 $$
 \Phi_{s_n}^{\epsilon_n}(x_n, \lambda_n) = (z_n, \nu_n) \in \partial \hat{U} \quad \text{and} \quad \Phi_{t}^{\epsilon_n}(x_n, \lam_n)\in \hat{U} \text{ for } t\in [0, s_n]. 
 $$
 \par Then, we let $(z,\nu)$ be a limit point of the sequence $(z_n, \nu_n)$ in the compact space $\pspa$. By $\lam_n \to \lam$, $\mu_n \to \lam$, as $n \to \infty$, and $\lam_n  \leq \nu_n \leq \mu_n$ for all $n$, we have $\nu_n \to \lambda$ as $n \to \infty$, and therefore $\nu = \lambda$. 
 \par Then we claim that the time sequence $\sset{s_n}_{n \in \NN}$ is unbounded. Suppose if $\sset{s_n}_{n \in \NN}$ is bounded, and we get a limit point of the sequence $\sset{s_n}_{n \in \NN}$, denoted as $s$. We have that $\Phi^{0}_{s}(x, \lam) \in \partial \hat{U}$, which makes $\hat{U}^\lam=U$ not an isolating neighburhood. Because this contradicts the definition of $U$, we know that $\sset{s_n}_{n\in \NN}$ is unbounded.
 \par We claim that for any time $\tau>0$, $\Phi^0_{-\tau}(z,\nu)=\Phi^0_{-\tau}(z,\lam)\in \cl (U)$. If this is not true, then there is a $\tau_0 > 0$ such that $\Phi^{0}_{-\tau_0}(z, \nu) \notin \cl (U)$. Then for sufficiently large $n$, we have that $\Phi_{-\tau_0}^{\eps}(z_n, \nu_n)\notin \cl (U)$. However, we have $\Phi_{-\tau_0}^{\eps_n}(z_n, \nu_n) = \Phi_{s_n-\tau_0}^{\eps_n}(x_n, \lam_n)$ and if $n$ is sufficiently large, we have $0<s_n-\tau_0 <s_n$, and $\Phi_{-\tau_0}^{\eps_n}(z_n, \nu_n) = \Phi_{s_n-\tau_0}^{\eps_n}(x_n, \lam_n)\in \hat{U}$. And then we take the limit as $n\to \infty$, we get that $\Phi_{-\tau_0}^{0}(z, \nu) \in \cl (U)$, which is a contradiction.
 \par Since for any time $t>0$, $\Phi^0_{-t}(z,\nu)=\Phi^0_{-t}(z,\lam)\in \cl (U)$, we have under the flow $\Phi^0_t$ the alpha limit set $\alpha(z,\nu) \subset M_{I_p}^\lam \times \sset{\lam}$. We consider the omega limit set $\omega(z, \nu)$ under the flow $\Phi^0_t$, then we have there is a $p^{'}$ such that $\omega(z, \nu)\subset M_{p^{'}}^\lam$. If $p^{'} \in I_p$, we have $U$ is not an isolating neighborhood for $M_{I_p}^{\lam}$, because the point $(z, \nu)\in \partial \hat{U}$. If $p^{'} \not\in I_p$, we have that $I_p$ is not an attracting interval. In each case, we get a contradiction; thus, the lemma is proved.
\end{proof}

\par Similar holds for the backward flow, that is, if $t_n\leq0$ for all $n\in \NN$, we have $p \leq_\lam^{\Gf} q$.
\begin{lma}\label{totallyordered}
	For any $\lam \in [0,1]$, $I^\lam$ is a totally ordered subset of the flow-defined order set $\FO{\lam}$.
\end{lma}
\begin{proof}
Suppose there are $p,q\in I^\lam$, and let $(x,\lam)\in M_p^{\lam}$ and $(y, \lam)\in M_q^\lambda	$. Because $\cc_n \to \cc$ as $n\to \infty$, we have two seqences $(x_n, \lam_n)\in \cc_n$ and $(y_n, \mu_n)\in \cc_n$, such that $(x_n, \lam_n)\to (x, \lam)$ and $(y_n, \mu_n)\to (y, \lam)$. Because $\cc_n$ is a connecting orbit under the flow $\Phi^{\eps_n}_t$, there are a sequence of time $\sset{t_n}_{n\in \NN}$ such that $\Phi_{t_n}^{\eps_n}(x_n, \lam_n) = (y_n, \mu_n)$ in $\cc_n$ for all $n\in \NN$. Therefore, there are three cases that could happen for the time sequence $\sset{t_n}_{n\in \NN}$:
\begin{enumerate}
	\item There are infinitely many positive times $(t_i)$s and finitely many negative times $(t_j)$s, and by Lemma \ref{orderpreserve}, we have $q\leq_\lam^\Gf p$, in the flow-defined order.
	\item There are infinitely many negative times $(t_i)$s and finitely many positive times $(t_j)$s, and by Lemma \ref{orderpreserve}, we have $p\leq_\lam^\Gf q$, in the flow-defined order.
	\item There are infinitely many positive times $(t_i)$s and infinitely many negative $(t_j)$s, and by Lemma \ref{orderpreserve}, we have $q\leq_\lam^\Gf p$ and $p\leq_\lam^\Gf q$, which means $p =_\lam^\Gf q$ in the flow-defined order.
\end{enumerate}
Therefore, in each case, we know that for any $p,q\in I^\lam$, $p$ and $q$ are comparable. As a result, $I^\lam$ is a totally ordered set.
\end{proof}

\begin{note}
	We should note here that Lemma \ref{orderpreserve} and Lemma \ref{totallyordered} hold no matter whether there is a continuation or not over the interval $[0,1]$ for the Morse decomposition. Therefore, under Assumption \ref{breakdown}, for the slice $\lam = \lam_0$, we have that $I^\lz$ is also a totally ordered set in $\FO{\lz}$.
	\end{note}

\par Now we assume the Morse decomposition continues over an interval $L\subset [0,1]$, and we denote the Morse decomposition as $(\MD{}, \DP_L)=(\mset{M_\pi}{\pi\in \DP_L},\DP_L)$, with an admissible order $\DP_L$. For each slice $\lam \in L$, the Morse decomposition is given by $(\MD{\lambda},\DP_L):=(\mset{M_p^\lambda}{p\in \DP_L},\DP_L)$
\par In order to determine whether the limit connecting orbit $\cc$ is entirely contained in some Morse set or not, we define the following for a $\pi \in \DP_L$:
$$
\Lam^L_\pi := \mset{\lam\in L}{\cc^\lam \subset M_\pi } \quad \text{ and, } \quad \wt{\Lam}_\pi^L := \mset{\lam\in L}{\cc^\lam \cap M_\pi \neq \OO }.
$$
Obviously, we have $\Lam^L_\pi \subset \wt{\Lam}^L_\pi$.
\begin{lma}\label{openclosed}
For any $\pi\in \DP_L$, the set $\Lam_\pi^L$ is open and the set $\wt{\Lam}_\pi^L$ is closed in the relative topology of $L$.	
\end{lma}
\begin{proof}
	We define the projection map $\Pi_\Lam:N\times \Lam \to \Lam$, as $\Pi_\Lam:(x,\lam) \mapsto \lam$. Then, because $N$ is compact, the projection map $\Pi_\Lam$ is a closed map. 
\par We take the relative topology of $N\times L$ and $L$, then the restriction map $\Pi_\Lam \mid _{N \times L}: N\times L \to L$ is also a closed map. Since for any closed set $K\subset N\times L$, there is a $K^{'}\subset N\times \Lam$ such that $K^{'}\cap (N\times L) = K$, and $\Pi_\Lam(K^{'})$ is closed in $\Lam$. Because $\Pi_\Lam\! \mid _{N \times L}(K) = \Pi_\Lam (K) =\Pi_\Lam (K^{'}) \cap L$, we have that $\Pi_\Lam \mid _{N \times L}(K)$ is also closed. Also, $\cc$ is compact in the $N\times [0,1]$, then for the relative topology, $\cc \cap (N\times L)$ is also closed. 
\par For any $\pi \in \DP_L$, we consider the set $\cup_{\lam \in L} M_\pi^\lam$. Take the closure of  $\cup_{\lam \in L} M_\pi^\lam$ in $N\times[0,1]$, and denote it as $\cl(\cup_{\lam \in L} M_\pi^\lam)$. By Lemma $3.10$ in \cite{Reineck1998connection}, for any $[l_1,l_2]\subset L, \cup_{\lam_\in[l_1,l_2]}M_\pi^\lambda$ is compact in $N\times [0,1]$. We have $\cup_{\lam \in L} M_\pi^\lam=\cl(\cup_{\lam \in L} M_\pi^\lam) \cap (N\times L)$, therefore $\cup_{\lam \in L} M_\pi^\lam$ is relatively closed in $N\times L$.
\par Therefore, $\wt{\Lam}_\pi^L= \Pi_\Lam((\cc \cap (N\times L))\cap  \cup_{\lam \in L} M_\pi^\lam )$ is also closed in $L$. Then, we have $\Lam_\pi^L= L - \cup_{\rho\in \DP_L, \rho \neq \pi} \wt{\Lam}_\rho^L$ and thus $\Lam_\pi^L$ is open in $L$.
\end{proof}
\par Now, we present the property of the limit orbit $\cc$ in the situation that $\cc^\lam$ intersects more than one Morse set in $\MD{\lam}$.
\begin{prop}\label{twoside}
	Let $\lam \in \inte{L}$ with the topology of $\Lam$, and we let $p= \inf I^\lam, q=\sup I^\lam$. Then there is an $\veps>0$ such that $(\lam -\veps, \lam )\subset \Lam_p^L$ and $(\lam, \lam +\veps)\subset \Lam_q^L$, with $(\lam-\veps, \lam +\veps)\subset L$.
\end{prop}
\begin{proof}
	We prove the claim for $\Lam_p^L$ by contradiction. The proof for $\Lam_q^L$ is the same by considering the backward flow.
\par We suppose that if the statement for $\Lam_p^L$ is not true, then there is a monotone increasing sequence of $\sset{\lam_n}_{n\in \NN}$, with $\lam_n\to \lam_-$, such that $\lam_n \in \wt{\Lam}_{r_n}^L$ for some $r_n\in \DP_L-\sset{p}$. Because $\DP_L$ is finite, we have a subsequence of $\sset{\lam_n}_{n\in \NN}$, for simplicity we also denote it as $\sset{\lam_n}_{n\in \NN}$, such that $\lam_n \in \wt{\Lam}_r^L$ for some $r\in \DP_L-\sset{p}$ for all $\for{n}$. Since $\lam \in \inte{L}$ with the topology of $\Lam$, and $\wt{\Lam}_r^L$ is closed with respect to the relative topology of $L$, we have that the limit $\lam\in \wt{\Lam}_r^L$. Therefore $r\in I^\lam$.
\par We consider the $\cc_n \to \cc$ as $n \to \infty$. We take the sequences $\sset{(x_n, \mu_n)}_{n\in \NN}$ with $(x_n, \mu_n) \in  \cc_n$ and $\sset{(y_n^m,\nu_n^m)}_{n\in \NN}$ with $(y_n^m,\nu_n^m)\in \cc_n$ for each $m\in \NN$, such that $(x_n, \mu_n) \to (x,\lam)\in M_p^\lambda$ as $n\to \infty$, $(y_n^m, \nu_n^m)\to (y^m, \lam^m)\in M_r^{\lam_m}$, as $n\to \infty$ for each $m\in \NN$. We note that $(x_n, \mu_n)\in \cc_n$ converges to $(x,\lam)\in M_p^\lambda$ in the slice $\lam$ as $n\to \infty$. And for a fixed $m\in \NN$,  $(y_n^m, \nu_n^m)\in \cc_n$ converges to $(y^m, \lam^m)\in M_r^{\lam_m}$ in the slice $\lam_m$ as $n\to \infty$. Then, we are going to find a subsequence of $\NN$, denoted as $\sset{n_k}_{k\in \NN}$ such that,
\begin{enumerate}[label=\textnormal{(\arabic*).}]
	\item $n_{k+1} > n_k$,
	\item $d((y_{n_k}^{m_k},\nu_{n_k}^{m_k}),M_r^{\lam_{m_k}})<\frac{1}{m_k}$,
	\item $\mu_{n_k}<\nu_{n_k}^{m_k}$.
\end{enumerate}
We can choose such a subsequence of $\NN$. Because the sequence $(y_{n}^{m},\nu_{n}^{m}) \to M_r^{\lam_{m}}$ as $n\to \infty$, we firstly take a subsequence such that for each $m\in \NN$,  $d((y_{n_v}^{m},\nu_{n_v}^{m}),M_r^{\lam_{m}})<\frac{1}{m}$ for all $v\in \NN$, and relabel them as before $\sset{(y_{v}^{m},\nu_{v}^{m})}_{v\in \NN}$ for each $m\in \NN$. Then, no matter which $v$ we choose, the requirement $(2)$ is automatically satisfied if we consider the new subsequence. For the requirement $(3)$, because $\lam_m\to \lam_-$ and we choose a sufficient large $m=D$, we can get $\mu_1< \lam_D<\lam$, and consider a subsequence $\sset{(y_{v_l}^{D},\nu_{v_l}^{D})}_{l\in \NN}$ of the new $\sset{(y_{v}^{D},\nu_{v}^{D})}_{v\in \NN}$ such that $\mu_1<\nu_{v_1}^{D}$, and we relabel the $\sset{(y_{v_l}^{D},\nu_{v_l}^{D})}_{l\in \NN}$ as $\sset{(y_{l}^{D},\nu_{l}^{D})}_{l\in \NN}$ for $m=D$, then we can let $n_1=1$ and $m_k =D$. We can continue to do this for $n_2=2$, and then we get the subsequence of $\NN$.
\par Now we have two subsequences $\sset{(x_{n_k}, \mu_{n_k})}_{k\in \NN}$ and $(y_{n_k}^{m_k},\nu_{n_k}^{m_k})_{k\in \NN}$. And we find a convergent subsequence of $(y_{n_k}^{m_k},\nu_{n_k}^{m_k})_{k\in \NN}$, such that $(y_{n_{k_d}}^{m_{k_d}},\nu_{n_{k_d}}^{m_{k_d}}) \to (y, \lambda)\in M_r^\lambda$ as $d\to \infty$, by there is a compact $L_0\subset \inte{L}$ and the $\cup_{{\lam}\in L_0}M_r^\lambda$ is compact in $N\times \Lam$. 
\par We relabel the two sequences $\sset{(x_{n_{k_d}}, \mu_{n_{k_d}})}_{d\in \NN}$ and $\sset{(y_{n_{k_d}}^{m_{k_d}},\nu_{n_{k_d}}^{m_{k_d}})}_{d\in \NN}$ as $\sset{(x_{d}, \mu_{d})}_{d\in \NN}$ and $\sset{(y_{d}^{d},\nu_{d}^{d})}_{d\in \NN}$ respectively. Because $(x_{d}, \mu_{d})\in \cc_d$ and $(y_{d}^{d},\nu_{d}^{d})\in \cc_d$, and $\mu_d<\nu_t^d$, there is a time $T_d=\nu_d^d - \mu_d>0$ such that $\Phi_{T_d}^n((x_d,\mu_d))=(y_d,\nu_d)$ for all $d\in \NN$. Since $T_d >0$ for all $\for{d}$, by Lemma \ref{orderpreserve}, we have $r\leq_\lam^\Gf p$, and $r\neq p$, we have $r<_\lam^\Gf p$. But considering $r\in I^\lam$ and $p=\inf I^\lam$, there is a contradiction. As a result, the proposition is proved.
\end{proof}
\begin{prop}\label{interval}
	 For any $\pi \in \DP_L$, $\wt{\Lam}_\pi^L$ and $\Lam_\pi^L$ are both intervals in $L$.
\end{prop}
\begin{proof}
	To prove $\wt{\Lam}_\pi^L$ is an interval, we need to prove that $\wt{\Lam}_\pi^L$ is connected for any $\pi \in \DP_L$. We suppose that $\wt{\Lam}_\pi^L$ is not connected in the interval $L$ for a $\pi \in \DP_L$. Then there is a $x\not \in \wt{\Lam}_\pi^L$, and two closed connected components $K_1, K_2$ of $\wt{\Lam}_\pi^L$ in the topology of $L$, such that  $\sup K_1<x$ and $x<\inf K_2$. Because $L$ is an interval, $\sup K_1=\max K_1$ and $\inf K_2= \min K_2$. Let $\eta_1:=\max K_1$ and $\eta_2 :=\min K_2$ then $\eta_1, \eta_2\notin \Lam_\pi^L$. We have $\pi = \inf I^{\eta_1}$ and $\pi = \sup I^{\eta_2}$. And by Proposition \ref{twoside}, there are finite mutually disjoint open intervals $L_1, L_2, \cdots, L_m$ with relative topology in $L$ between $K_1$ and $K_2$, such that $\cup_{i=1}^m \cl(L_i) = [\eta_1, \eta_2]$ with $\sup L_i = \inf L_{i+1}$ for all $i = 1,2,\cdots,m-1$  and $\sup I^{\inf L_{i}} = \inf I^{\sup L_i}$ for all $i = 1,2,\cdots,m$. Thus, we have there is a $\rho\in \DP_L$ such that $\pi <_{\eta_1}^\Gf \rho <_{\eta_2}^\Gf \pi$, which is a contradiction. As a result, we have $\wt{\Lam}_\pi^L$ is connected for any $\pi \in \DP_L$, and therefore $\wt{\Lam}_\pi^L$ is an interval for any $\pi \in \DP_L$.
	\par Similarly, we can use Proposition \ref{twoside} to get the connectedness of $\Lam_\pi^L$, and therefore $\Lam_\pi^L \subset \wt{\Lam}_\pi^L$ is an interval in $\wt{\Lam}_\pi^L$.
\end{proof}
\begin{prop}
	 For any $\pi \in \DP_L$, $\wt{\Lam}^L_\pi-\Lam_\pi^L$ consists of at most two points in the interval $L$.
\end{prop}
\begin{proof}
If not, there is a $\pi \in \DP_L$, and there are three points $\gamma_1,\gamma_2,\gamma_3\in \wt{\Lam}^L_\pi-\Lam_\pi^L$ with $\gamma_1<\gamma_2<\gamma_3$. By Proposition \ref{interval}, we have $(\gamma_1,\gamma_3)\subset \wt{\Lam}^L_\pi$. Considering $\inf I^{\gamma_2}$ and $\sup I^{\gamma_2}$, we have $\inf I^{\gamma_2} \neq \sup I^{\gamma_2}$. Therefore, either $\inf I^{\gamma_2}\neq \pi$ or $\sup I^{\gamma_2}\neq \pi$. We may assume $\sup I^{\gamma_2}=\xi \neq \pi$, then by Proposition \ref{twoside}, there is an $\veps>0$ such that $(\gamma_2,\gamma_2+\veps)\subset \Lam^L_\xi$. From the definition, we have $\Lam^L_\xi \cap \wt{\Lam}^L_\pi =\OO$, but $(\gamma_2,\gamma_2+\veps) \cap (\gamma_1,\gamma_3) \neq \OO$, which is a contradiction. Thus, the proposition is proved.
\end{proof}
\par Here we let the interval $L$ be the interval $[0,1]$ and assume Assumption \ref{connexist}. It is the same situation as in Reineck (\cite{Reineck1998connection}, Theorem $3.13$). There is a little mistake in Theorem $3.13$ in Reineck \cite{Reineck1998connection}. Here, we present its correction and give a proof in the following.

\begin{thm}[\cite{Reineck1998connection}, Theorem $3.13$]\label{Reineck}
	Let $L=[0,1]$. Under Assumption \ref{connexist}, there exist a finite sequence of parameters $1=\gam_0 > \gamma_1 >\gamma_2 >\cdots > \gamma_{k-1}>\gamma_{k} > \gam_{k+1} =0$ and a sequence of indices $ \rho = \xi_0, \xi_1, \xi_2,\cdots,$ $\xi_{k}, \xi_{k+1} , \xi_{k+2}= \pi$ such that $M_{\xi_{i+1}}^{\gamma_i} <_{\gamma_i}^\Gf M_{\xi_i}^{\gamma_i}$, for all $i\in \sset{1,\cdots,k}$, where  $<_{\gamma_i}^\Gf$ is the flow-defined order on $\DP_{\gam_i}^\Gf$ for the flow $\phi_t^{\gamma_i}$. At the slice $\lam = 1$, we have $\rho =\xi_0 = \sup I^1$, $\xi_1 = \inf I^1$, and at the slice $\lam = 0$, we have $\xi_{k+1} = \sup I^0$, $\pi = \xi_{k+2} = \inf I^0$.
\end{thm}
\begin{note}
	Reineck (\cite{Reineck1998connection}, Theorem $3.13$), mistakenly excluded the possibility that $1\not \in \Lam^L_\rho$ or $0\not \in \Lam^L_\pi$, that is, the limit $\ccl$ meets more than one Morse set at the slice $\lam=0,1$, as illustrated in Example \ref{ErrRec}.
\end{note}
\begin{proof}
By Lemma \ref{openclosed} and Proposition \ref{twoside}, we have $L-\cup_{p \in \DP_L} \Lam_p^L$ is a discrete set, and $L=[0,1]$ we have $L-\cup_{p \in \DP_L} \Lam_p^L$ is also compact. Therefore, $L-\cup_{p \in \DP_L} \Lam_p^L$ consists of finite points. We label the finite set $L-\cup_{p \in \DP_L} \Lam_p^L -\sset{0,1}$ by $\gam_i$, $i\in \sset{1,2,\cdots,k}$. By the connectedness of $\Lam_p^L$ and $\wt{\Lam}_p^L$ for all $p \in \DP_L$, we have that for all $i\in \sset{0,1,2,\cdots,k} $, $( \gamma_{i+1}, \gamma_i)=\Lam_{q(i)}^L$ with some $q(i)\in \DP_L$. Then we let $\xi_i=\inf I^{\gamma_{i-1}} =\sup I^{\gamma_{i}}$ for all $i\in \sset{1,2, \cdots,k,k+1}$, then we have $\xi_1 = \inf I^1$ and $\xi_{k+1} = \sup I^0$. Because for any $\lam \in L$, $I^\lam$ is totally ordered by the flow-defined order from Lemma \ref{totallyordered}, we have that $M_{\xi_{i+1}}^{\gamma_i} <_{\gamma_i}^\Gf M_{\xi_i}^{\gamma_i}$, for all $i\in \sset{1,\cdots,k}$.
\end{proof}
\par Now, we consider the situation that there is not a continuation. Under Assumption \ref{breakdown}, we get the following result.
\begin{thm}\label{withoutconn}
	Under Assumption \ref{breakdown} and Assumption \ref{connexist}, there exists a finite sequence of parameters 
	$$1=\gam_0 > \gamma_1 >\gamma_2 >\cdots > \gamma_{l-1}> \gamma_l=\lam_0 > \gamma_{l+1} > \cdots >\gamma_{k-1}>\gamma_{k} > \gam_{k+1} =0$$
	 and sequences of indices 
	$$\sset{ \rho = \xi_0, \xi_1, \xi_2,\cdots,\xi_{l}}\subset \DP_1, \sset{\xi_{l}^{'},\xi_{l+1}^{'}}\subset \DP_\lz, \sset{\xi_{l+1},\cdots, \xi_{k}, \xi_{k+1} , \pi=\xi_{k+2}}\subset \DP_0$$
	 such that 
	 $$M_{\xi_{i+1}}^{\gamma_i} <_{\gamma_i}^\Gf M_{\xi_i}^{\gamma_i} \quad \text{ for all } \quad i\in \sset{1,\cdots,k}-\sset{l}$$ 
	 and $M_{\xi_{l+1}^{'}}^{\lz} \leq_{\lz}^\Gf M_{\xi_{l}^{'}}^{\lz}$, where  $<_{\gamma_i}^\Gf$ and $\leq_{\lz}^\Gf$  are flow-defined orders on $\DP_{\gam_i}^\Gf$ and $\DP_{\lz}^\Gf$ for flows $\phi_t^{\gamma_i}$ and $\phi_t^\lz$, respectively. 
	 
	 \par At the slice $\lam = 1$, we have $\rho =\xi_0 = \sup I^1$, $\xi_1 = \inf I^1$, and at the slice $\lam = 0$, we have $\xi_{k+1} = \sup I^0$, $\pi = \xi_{k+2} = \inf I^0$. At the slice $\lam = \lam_0$, $M_{\xi_{l}^{'}}^\lz \cap \HL{\xi_{l}} \neq \OO $ and $M_{\xi_{l+1}^{'}}^\lz \cap \HL{\xi_{l+1}} \neq \OO$.
\end{thm}
\begin{proof}
Firstly, we apply Theorem \ref{Reineck} for intervals $[0,\lz)$ and $(\lz,1]$, and we get the 	$1=\gam_0 > \gamma_1 >\gamma_2 >\cdots > \gamma_{l-1} > \gamma_{l+1} > \cdots >\gamma_{k-1}>\gamma_{k} > \gam_{k+1} =0$ and $ \rho = \xi_0, \xi_1, \xi_2,\cdots,\xi_{l}\in \DP_1$, $\xi_{l+1},\cdots, \xi_{k}, \xi_{k+1} , \pi=\xi_{k+2}\in \DP_0$ such that $M_{\xi_{i+1}}^{\gamma_i} <_{\gamma_i}^\Gf M_{\xi_i}^{\gamma_i}$, for all $i\in \sset{1,\cdots,l-1}\cup \sset{l+1,\cdots,k}$, and where  $<_{\gamma_i}^\Gf$ is the flow-defined order on $\DP_{\gam_i}^\Gf$ for the flow $\phi_t^{\gamma_i}$. We have $(\lz, \gamma_{l-1})=\Lam_{\xi_{l}}^{(\lz,1]}$, and $(\gamma_{l+1},\lz)=\Lam_{\xi_{l+1}}^{[0,\lz)}$. Let $\xi_{l}^{'}=\sup I^\lz$ and $\xi_{l+1}^{'}=\inf I^\lz$, then because $I^\lz$ is totally ordered, we have $M_{\xi_{l+1}^{'}}^{\lz} \leq_{\lz}^\Gf M_{\xi_{l}^{'}}^{\lz}$. 
\par Let $K=I^\lz\subset \DP_\lz^\Gf$, and by the invariance of $\cc^\lz$ under the flow $\phi_t^\lz$, we have $\cc^\lz\subset M_K^\lz$. At last, by the connectedness of $\cc$, we have $\HL{\xi_{l}}$ and $\HL{\xi_{l+1}}$ connects to $M_K^\lz$, with $M_{\xi_{l}^{'}}^\lz \cap \HL{\xi_{l}} \neq \OO $ and $M_{\xi_{l+1}^{'}}^\lz \cap \HL{\xi_{l+1}} \neq \OO$.
\end{proof}

\par Notice that $\HL{\xi_{l}}\subset M_{\xi_{l}^{'}}^\lz$ and $ \HL{\xi_{l+1}}\subset M_{\xi_{l+1}^{'}}^\lz$ are not always true, but if we consider the finest decomposition of $(\DP_0, \DP_1)$, we can apply Theorem \ref{Reineck} to these indecomposable continuable interval pairs, even if the continuation of Morse decompositions over $[0,1]$ may not exist.
\begin{cor}
	Under Assumption \ref{breakdown} and Assumption \ref{connexist}, we consider the finest decomposition $\LTD{J}{A}$ of the continuable interval pair $(\DP_0, \DP_1)$, also the Morse decomposition 
	$$(\DM(S^{\lam};\{J_i\}),\DP(\{J_i\})):=(\{M^{\lam}_{I} \mid I\in \{J_i\}_{i\in A}\},\DP(\{J_i\}))$$
	 with $\DP(\{J_i\})$ the admissible order of $\{M^{\lam}_{I} \mid I\in \{J_i\}_{i\in A}\}$ continues over $[0,\lz)$, and the Morse decomposition 
	 $$(\DM(S^{\lam};\{J_i^{'}\}),\DP(\{J_i^{'}\})):=(\{M^{\lam}_{I} \mid I\in \{J_i^{'}\}_{i\in A}\},\DP(\{J_i^{'}\}))$$ 
	 with $\DP(\{J_i^{'}\})$ the admissible order of $\{M^{\lam}_{I} \mid I\in \{J_i^{'}\}_{i\in A}\}$ continues over $(\lz,1]$.
	\par Then, there exists a finite sequence of parameters 
	$$1=\gam_0 > \gamma_1 >\gamma_2 >\cdots > \gamma_{l-1}> \gamma_l=\lam_0 > \gamma_{l+1} > \cdots >\gamma_{k-1}>\gamma_{k} > \gam_{k+1} =0$$
	 and sequences of indices 
	 $$ \sset{\rho = \xi_0, \xi_1, \xi_2,\cdots,\xi_{l}}\subset \DP_1, \sset{\xi_{l}^{'},\xi_{l+1}^{'}}\subset \DP_\lz, \sset{\xi_{l+1},\cdots, \xi_{k}, \xi_{k+1} , \pi=\xi_{k+2}}\subset \DP_0$$
	  such that 
	  $$M_{\xi_{i+1}}^{\gamma_i} <_{\gamma_i}^\Gf M_{\xi_i}^{\gamma_i} \quad \text{ for all } \quad i\in \sset{1,\cdots,k}-\sset{l}$$
	  and $M_{\xi_{l+1}^{'}}^{\lz} \leq_{\lz}^\Gf M_{\xi_{l}^{'}}^{\lz}$ where  $<_{\gamma_i}^\Gf$ and $\leq_{\lz}^\Gf$  are flow-defined orders on $\DP_{\gam_i}^\Gf$ and $\DP_{\lz}^\Gf$ for flows $\phi_t^{\gamma_i}$ and $\phi_t^\lz$, respectively.
	  \par At the slice $\lam = 1$, we have $\rho =\xi_0 = \sup I^1$, $\xi_1 = \inf I^1$, and at the slice $\lam = 0$, we have $\xi_{k+1} = \sup I^0$, $\pi = \xi_{k+2} = \inf I^0$. At the slice 
	$\lam = \lam_0$, 
	$M^\lz_{\xi_{l}^{'}} \subset  \GM^{\lam_0}_{\DP_{\lam_0}(N(M_{\xi_{l}}))}$,
	 $M^\lz_{\xi_{l+1}^{'}} \subset  \GM^{\lam_0}_{\DP_{\lam_0}(N(M_{\xi_{l+1}}))}$,
	  $\HL{\xi_{l}}\subset  \GM^{\lam_0}_{\DP_{\lam_0}(N(M_{\xi_{l}}))}$
	   and
	    $\HL{\xi_{l+1}} \subset  \GM^{\lam_0}_{\DP_{\lam_0}(N(M_{\xi_{l+1}}))}$.
\end{cor}
\begin{proof}
From Theorem \ref{withoutconn}, we have the sequences of $(\gamma_i)$s, $(\xi_i)$s and $(\xi_i^{'})$s; also from Proposition \ref{decompositionlimit} we have these inclusions $\HL{\xi_{l}}\subset  \GM^{\lam_0}_{\DP_{\lam_0}(N(M_{\xi_{l}}))}$, $\HL{\xi_{l+1}} \subset  \GM^{\lam_0}_{\DP_{\lam_0}(N(M_{\xi_{l+1}}))}$ and $M^\lz_{\xi_{l}^{'}} \subset  \GM^{\lam_0}_{\DP_{\lam_0}(N(M_{\xi_{l}}))}$,
	 $M^\lz_{\xi_{l+1}^{'}} \subset  \GM^{\lam_0}_{\DP_{\lam_0}(N(M_{\xi_{l+1}}))}$.
\end{proof}
\begin{ex}[Hausdorff limit of connecting orbits]\label{Connecingorbitexample}
\par We study the pitchfork bifurcation from the point of view of this paper. Let $f(x,\lam)$ in (\ref{ODEn}) be $f(x,\lam):=(\lam - \lz)x-x^3$ which undergoes the pitchfork bifurcation at $\lz\in[0,1]$ and add the slow-drift $\dot{\lambda}=\eps \lambda (\lambda-1)$ on the parameter interval $\Lam$. Taking a sequence $\sset{\eps_n}_{\for{n}}$ with $\eps_n\to 0$, we consider the connection orbit $\cc_n$ in the extended slow-fast flow $\Phi_t^{\eps_n}$ and its Hausdorff limit $\cc$ (see Figure \ref{Connecting orbits in slow-drift for the pitchfork bifurcation and its Hausdorff limit}). This dynamical system satisfies Assumption \ref{breakdown} with the Morse decompositions as follows:

\begin{enumerate}
	\item  $(\MD{\zeta},\DP_\zeta)=(\sset{M_1^\zeta},\sset{1})$ for all $\zeta \in [0,\lz)$;
	\item $(\MD{\eta},\DP_\eta)=(\sset{M_1^\eta,M_2^\eta,M_3^\eta},\sset{1<_\eta 2<_\eta 3})$ for all $\eta \in (\lz,1]$, in which $M_1^\eta$, $M_2^\eta$ denote sinks under the flow $\phi_t^\eta$, and $M_3^\eta$ denotes the source under the flow $\phi_t^\eta$ for all $\eta\in (\lz,1]$ (see Figure \ref{Connecting orbits in slow-drift for the pitchfork bifurcation and its Hausdorff limit});
	\item $(\MD{\lz},\DP_\lz)=(\sset{M_1^\lz},\sset{1})$.
\end{enumerate}

\par In Figure \ref{Connecting orbits in slow-drift for the pitchfork bifurcation and its Hausdorff limit}, the bifurcation diagram over parameter interval $\Lambda$ and phase portraits of parameterized flow $\phi_t^\lam$ at some slices are denoted in the light grey. The dotted line denotes the unstable fixed point for $\phi_t^\lambda$, and the solid line denotes the stable ones. For a suitable choice of $\sset{\epsilon_n}_{\for{n}}$ with $\eps_n\to 0$, there is a connecting orbit $\cc_n$ from $M_2^1$ to $M_1^0$ in the flow $\Phi_t^{\eps_n}$ on the left side of \ref{COf2t1} in Figure \ref{Connecting orbits in slow-drift for the pitchfork bifurcation and its Hausdorff limit} and a connecting orbit $\cc_n$ from $M_3^1$ to $M_1^0$ in the flow $\Phi_t^{\eps_n}$ on the left side of \ref{COf3t1} in Figure \ref{Connecting orbits in slow-drift for the pitchfork bifurcation and its Hausdorff limit} for each $\for{n}$ and the two sequences $\sset{\cc_n}_{\for{n}}$ converge in Hausdorff metric as $\eps_n \to 0$. The Hausdorff limit $\cc$ of $\sset{\cc_n}_{\for{n}}$ is illustrated on the right side of \ref{COf2t1} and \ref{COf3t1} in Figure \ref{Connecting orbits in slow-drift for the pitchfork bifurcation and its Hausdorff limit}. 
\par On the right side of \ref{COf3t1} in Figure \ref{Connecting orbits in slow-drift for the pitchfork bifurcation and its Hausdorff limit}, there is a $\gamma\in(\lz,1]$ such that $\cc^{\gamma}$ intersects with two Morse sets $M_3^\gamma$, $M_1^\gamma$ and connecting orbits between them. 
The set $I^\gamma$ discussed in Lemma \ref{totallyordered} is given by $I^\gamma = \sset{M_3^\gamma,M_1^\gamma}$ which is a totally ordered set in $\DP_\gamma^\Gf$.  
\begin{figure}[H]
\centering  
\subfigure[A Connecting orbit from $M_2^1$ to $M_1^0$]{
\label{COf2t1}
\includegraphics[width=0.45\textwidth]{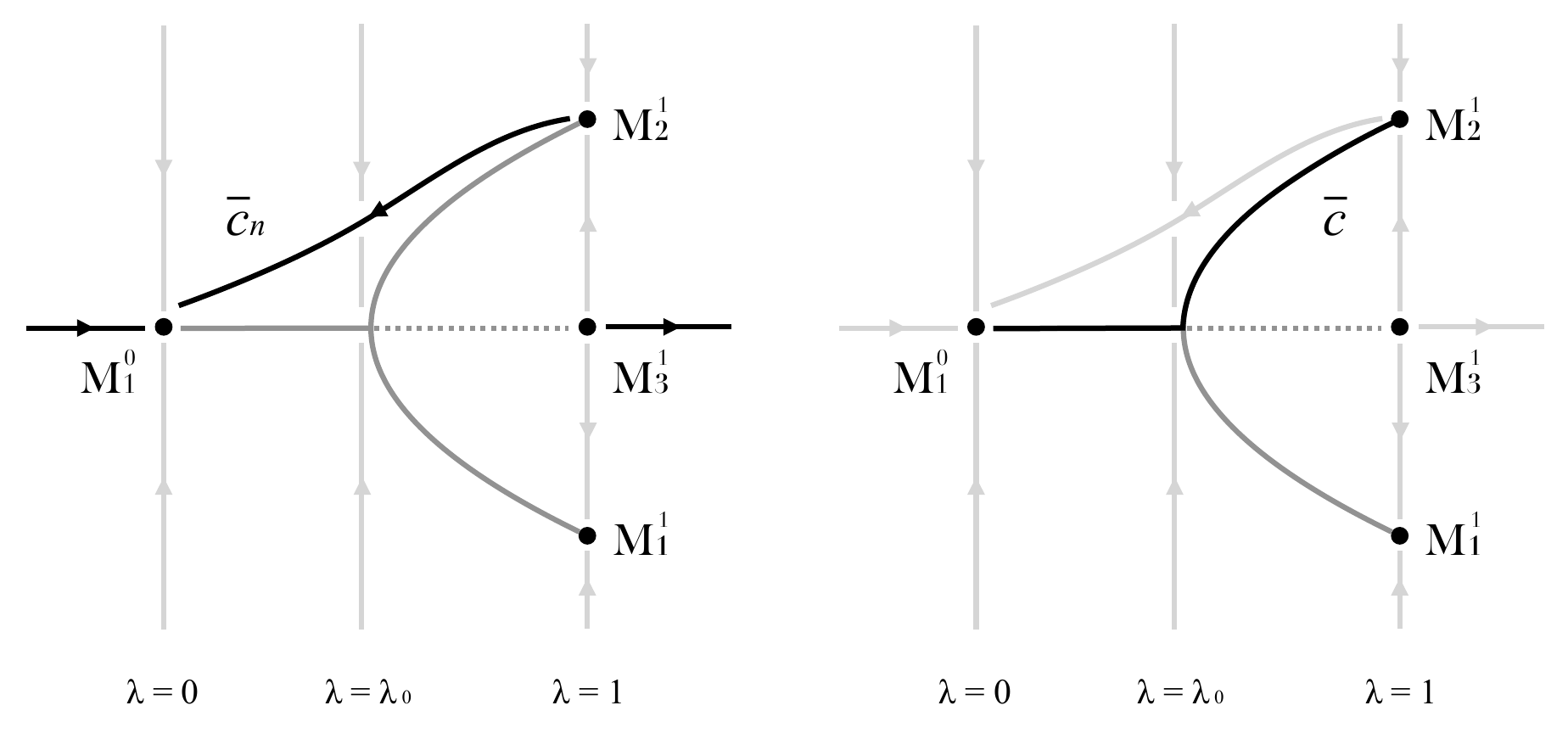}}
\subfigure[A Connecting orbit from $M_3^1$ to $M_1^0$]{
\label{COf3t1}
\includegraphics[width=0.45\textwidth]{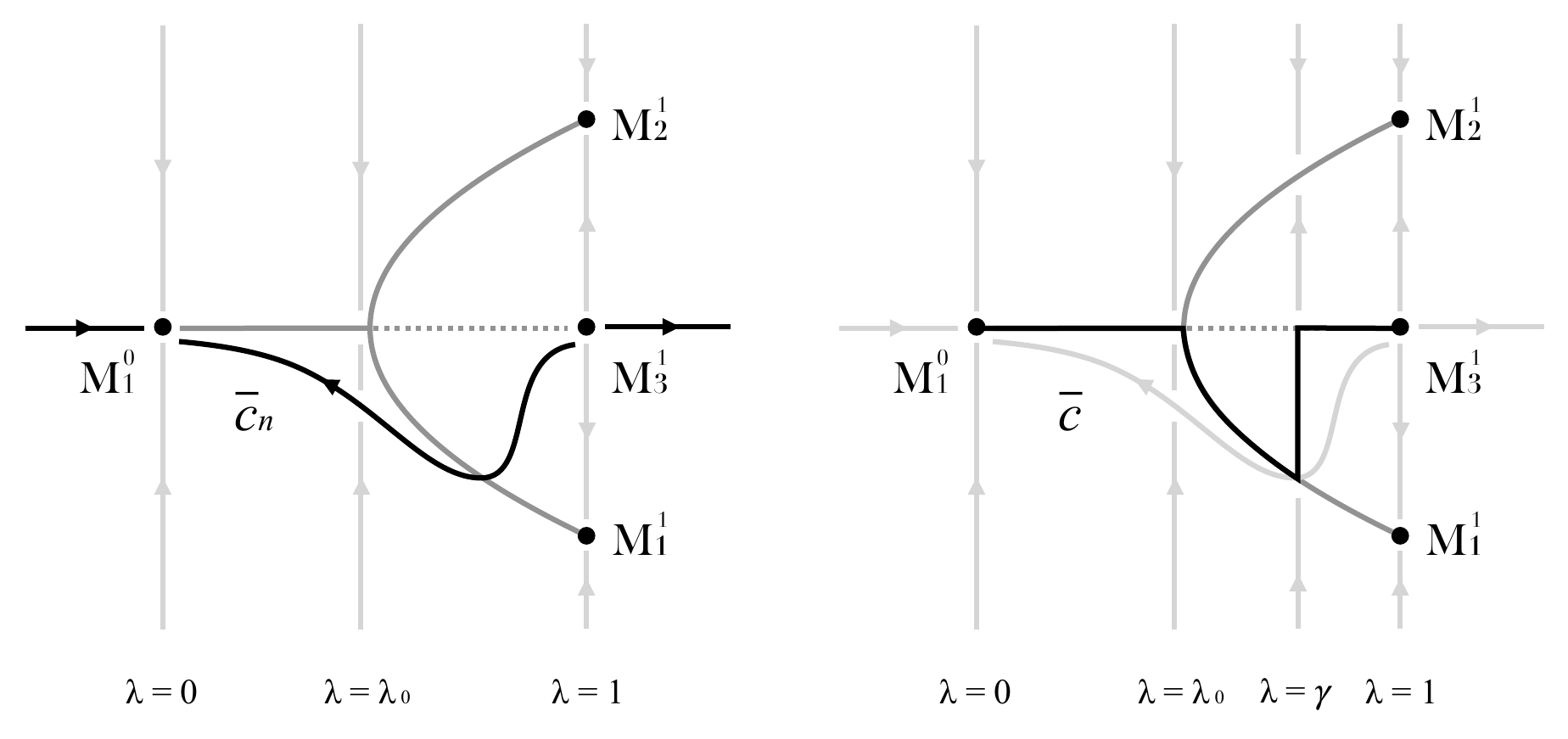}}
\caption{Connecting orbits and their Hausdorff limits (Pitchfork bif.)}
\label{Connecting orbits in slow-drift for the pitchfork bifurcation and its Hausdorff limit}
\end{figure}

\par We then consider a perturbed pitchfork bifurcation consisting of a saddle-node bifurcation and a stable fixed point continuing over $[0,1]$ (see Figure \ref{Connecting orbits and their Hausdorff limits (Perturbed Pitchfork Bifur.)}, in which grey and black lines denoted in the same as in Figure \ref{Connecting orbits in slow-drift for the pitchfork bifurcation and its Hausdorff limit}). This parameterized flow satisfies Assumption \ref{breakdown} with the Morse decompositions as follows:
\begin{enumerate}
	\item  $(\MD{\zeta},\DP_\zeta)=(\sset{M_1^\zeta},\sset{1})$ for all $\zeta \in [0,\lz)$;
	\item $(\MD{\eta},\DP_\eta)=(\sset{M_1^\eta,M_2^\eta,M_3^\eta},\sset{1<_\eta 2<_\eta 3})$ for all $\eta\in (\lz,1]$, in which $M_1^\eta$, $M_2^\eta$ denote sinks under the flow $\phi_t^\eta$, and $M_3^\eta$ denotes the source under the flow $\phi_t^\eta$ for all $\eta\in (\lz,1]$ (see Figure \ref{Connecting orbits and their Hausdorff limits (Perturbed Pitchfork Bifur.)});
	\item $(\MD{\lz},\DP_\lz)=(\sset{M_1^\lz, M_2^\lz},\sset{1<_\lz 2})$, in which $ M_1^\lz=\HLL{1}{-}, M_2^\lz=\HLL{2}{+}=\HLL{3}{+}$.
\end{enumerate}

\par Similarly to the above, we consider the suitable sequence $\sset{\eps_n}_{\for{n}}$ such that there is a connecting orbit $\cc_n$ from $M_3^1$ to $M_1^0$ in the flow $\Phi_t^{\eps_n}$ on the left side of \ref{pCOf2t1} in Figure \ref{Connecting orbits and their Hausdorff limits (Perturbed Pitchfork Bifur.)} and a connecting orbit $\cc_n$ from $M_3^1$ to $M_1^0$ in the flow $\Phi_t^{\eps_n}$ on the left side of \ref{pCOf3t1} in Figure \ref{Connecting orbits and their Hausdorff limits (Perturbed Pitchfork Bifur.)} for each $\for{n}$ and the two sequences $\sset{\cc_n}_{\for{n}}$ converge in Hausdorff metric as $\eps_n \to 0$. The Hausdorff limit $\cc$ of $\sset{\cc_n}_{\for{n}}$ is illustrated on the right side of \ref{pCOf2t1} and \ref{pCOf3t1} in Figure \ref{Connecting orbits and their Hausdorff limits (Perturbed Pitchfork Bifur.)}. 
\par On the right side of \ref{pCOf2t1} in Figure \ref{Connecting orbits and their Hausdorff limits (Perturbed Pitchfork Bifur.)}, there is a $\gamma\in(\lz,1]$ such that $\cc^{\gamma}$ intersects with two Morse sets $M_3^\gamma$, $M_2^\gamma$ and connecting orbits between them. Also, on the right side of \ref{pCOf2t1} in Figure \ref{Connecting orbits and their Hausdorff limits (Perturbed Pitchfork Bifur.)}, $\cc^{\lz}$ intersects with two invariant sets $\HLL{2}{+}=M_2^\lz$, $\HLL{1}{+}=M_1^\lz$ and connecting orbits between them. The sets $I^\gamma, I^\lz$ discussed in Lemma \ref{totallyordered} are totally ordered set in $\DP_\gamma^\Gf$ and $\DP_\lz^\Gf$, respectively.
\par  On the right side of \ref{pCOf3t1} in Figure \ref{Connecting orbits and their Hausdorff limits (Perturbed Pitchfork Bifur.)}, there is a $\gamma\in(\lz,1]$ such that $\cc^{\gamma}$ intersects with two Morse sets $M_3^\gamma$, $M_1^\gamma$ and connecting orbits between them. Also, the set $I^\gamma$ is a totally ordered set in $\DP_\gamma^\Gf$.
\begin{figure}[H]
\centering  
\subfigure[A Connecting orbit from $M_2^1$ to $M_1^0$]{
\label{pCOf2t1}
\includegraphics[width=0.45\textwidth]{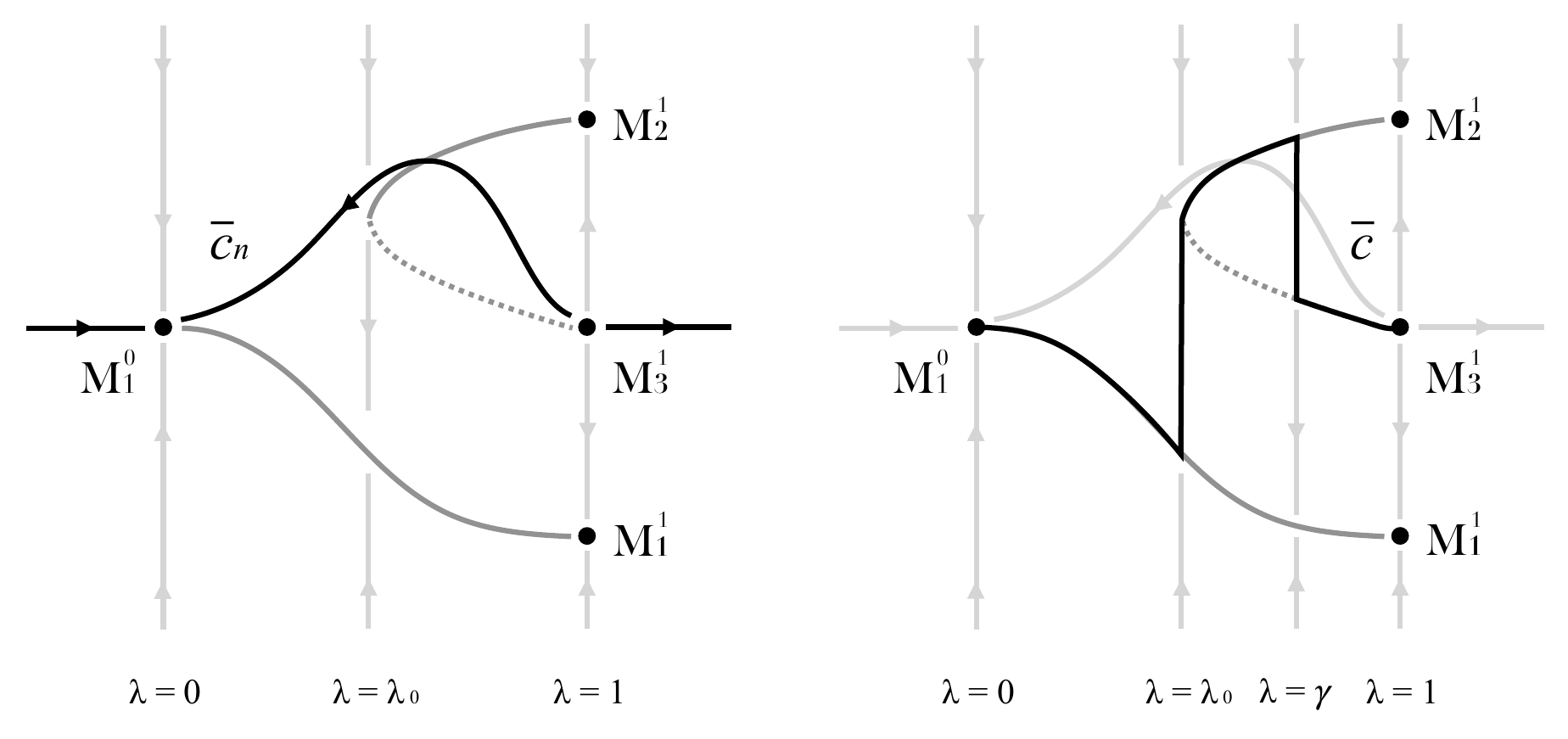}}
\subfigure[A Connecting orbit from $M_3^1$ to $M_1^0$]{
\label{pCOf3t1}
\includegraphics[width=0.45\textwidth]{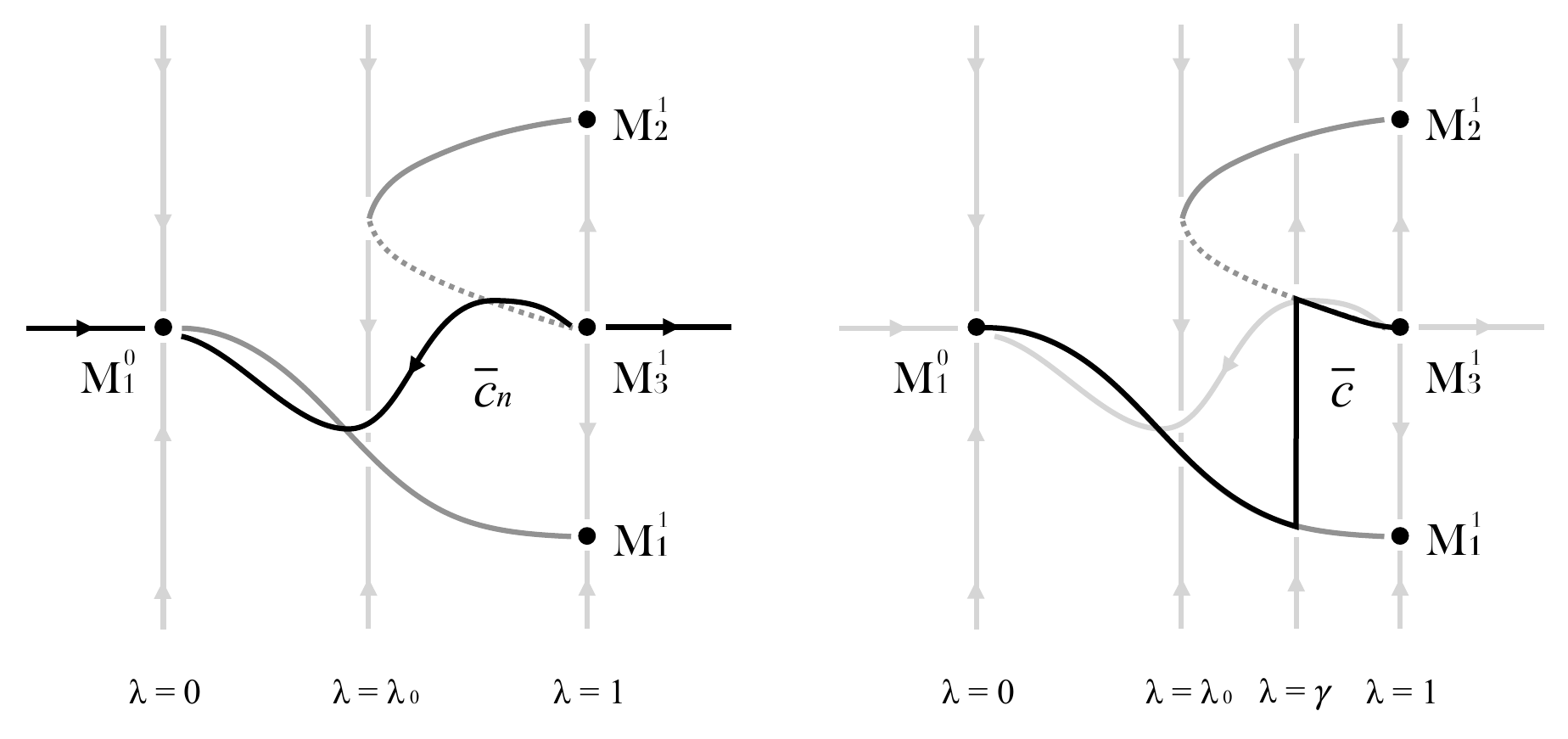}}
\caption{Connecting orbits and their Hausdorff limits (Perturbed pitchfork bif.)}
\label{Connecting orbits and their Hausdorff limits (Perturbed Pitchfork Bifur.)}
\end{figure}
\end{ex}
\begin{ex}\label{ErrRec}
We consider the same pitchfork bifurcation and perturbed pitchfork bifurcation with a slow drift on parameter space as in Example \ref{Connecingorbitexample}. This time, we choose a suitable sequence $\sset{\cc_n}_{\for{n}}$ such that there is a point $(x_n,y_n)\in \cc_n$ with $x_n\to 1$ as $n\to \infty$, and the sequence $\sset{\cc_n}_{\for{n}}$ converges in Hausdorff metric as $n \to \infty$. The Hausdorff limit $\cc$ meets more than one Morse set at the slice $\lam=1$, namely, in \ref{meets3,1} of Figure \ref{meets more than one Morse set}, $\cc^1=M_3^1\cup M_1^1 \cup \DC(M_1^1,M_3^1)$, and in \ref{meets3,2} of Figure \ref{meets more than one Morse set}, $\cc^1=M_3^1\cup M_2^1 \cup \DC(M_2^1,M_3^1)$.
\begin{figure}[H]
\centering  
\subfigure[$\cc^1$ meets $M_3^1$ and $M_1^1$]{
\label{meets3,1}
\includegraphics[width=0.45\textwidth]{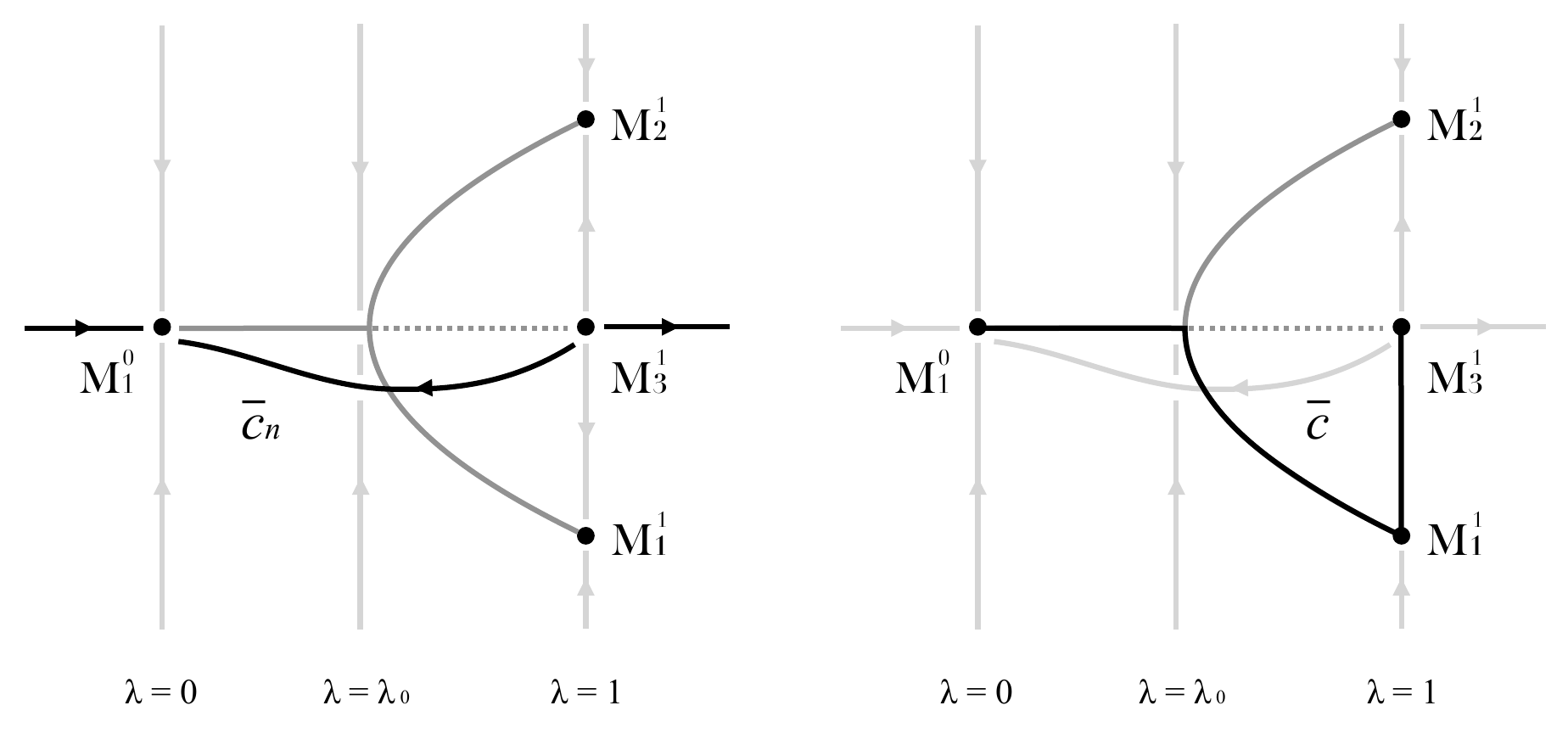}}
\subfigure[$\cc^1$ meets $M_3^1$ and $M_2^1$]{
\label{meets3,2}
\includegraphics[width=0.45\textwidth]{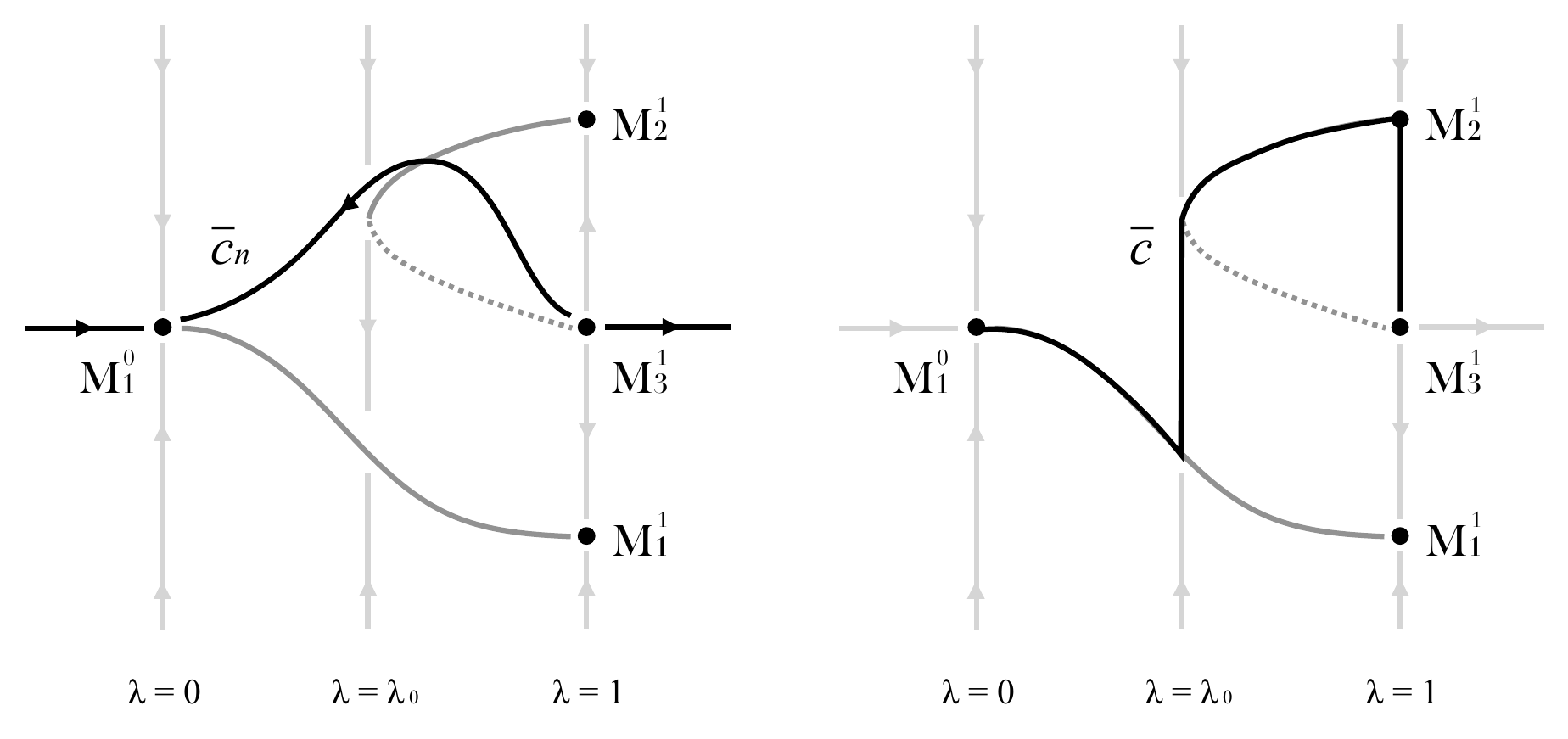}}
\caption{The Hausdorff limit $\cc$ meets more than one Morse set at $\lam=1$}
\label{meets more than one Morse set}
\end{figure}
\end{ex}

\par From the preceding discussion, we see several properties of the Hausdorff limit assuming the existence of connecting orbits for some choice of $\sset{\eps_n}_\for{n}$ with $\eps_n\to 0$.
In the next section, we use the singular transition matrix to show the existence of connecting orbits for a sequence $\sset{\epsilon_n}_{\for{n}}$ with $\eps_n \to 0$.

\subsection{Singular Transition Matrix}\label{STM}

\par Reineck \cite{Reineck1998connection} constructed the singular transition matrix in the situation where the Morse decomposition continues over the parameter interval $[0,1]$. In this section, we construct a singular transition matrix under Assumption \ref{breakdown} and show that it is a chain map between the chain complexes obtained by the homological Conley indices of the Morse sets. We then present an example to illustrate how our singular transition matrix can show the emergence of bifurcation that may not be detected by the conventional transition matrix at some parameter point under Assumption \ref{breakdown}.
\par Throughout this section, we consider the homological Conley index with a finite field coefficient, and we use the field coefficient $\ZZ_2$ in the examples below.
\par Consider the flow $\Phi_t^\eps$ for (\ref{ODEn}) with $\eps>0$ in the product space $N\times \Lam$ and let $\hse=\text{Inv}(N\times \Lam, \Phi_t^\epsilon)$. We have the following lemma of a Morse decomposition of the invariant set $\hse$ for the flow $\Phi_t^\epsilon$. 
\begin{lma}[\cite{Reineck1998connection}, Lemma $3.3$] Let $(\MD{0}, \DP_0)$ be a Morse decomposition of $S^0$ for the flow $\phi_t^0$, and $(\MD{1},\DP_1)$ be a Morse decomposition of $S^1$ for the flow $\phi_t^1$. Then the collection $\MD{0} \sqcup \MD{1}$ is a Morse decomposition of $\hse$ with any small positive $\eps$ for the flow $\Phi_t^\eps$, with index set $\h{\DP}:=\DP_0 \sqcup \DP_1$ equipped with the following admissible order:
\begin{align*}
&\pi<\rho \qquad \text{for all $\pi\in \DP_0$, $\rho\in \DP_1$};\\
&\pi<\pi^{'} \qquad \text{if $\pi<_{\DP_1}\pi^{'}$}; \\
&\rho<\rho^{'} \qquad \text{if $\rho<_{\DP_0}\rho^{'}$}.
\end{align*}
\end{lma}  
\par We denote the Morse decomposition of the invariant set $\hse$ as $(\DM(\hse),\h{\DP})$ for the flow $\Phi_t^\eps$. Let $\h{\DP}_i$ be the restriction of $\h{\DP}$ on $\DP_i$, for $i=0$ or $1$, then $\DP_i \cong \h{\DP}_i$, and under the trivial identification we do not distinguish $\DP_i$ and $\h{\DP}_i$ unless it causes confusions. We write those Morse sets for the flow $\Phi_t^\eps$ in the $\DM(\hse)$ as $M_p^0\times \{0\}$ if $p\in {\DP}_0$, and $M_q^1\times \{1\}$ if $q\in {\DP}_1$. As a result, $\DM(\hse)=\mset{M_p^0\times \{0\}}{p\in {\DP}_0}\cup \mset{M_q^1\times \{1\}}{q\in {\DP}_1}$.

\par We notice that this Morse decomposition is independent of the $\eps$. For each small positive $\eps$, Theorem \ref{existenceofconnectionmatrix} shows the existence of a connection matrix of the Morse decomposition $\MDp$. 
With the Morse decomposition $(\MD{0}, \DP_0)$ for the flow $\phi_t^0$ and the Morse decomposition $(\MD{1},\DP_1)$ for the flow $\phi_t^1$, we shall take a connection matrix for the flow $\Phi_t^\eps$, and define the singular transition matrix as follows. 
Let $\Delta^\eps_*$ be a connection matrix for the flow $\Phi_t^\eps$ with the Morse decomposition $\MDp$ and $\Delta_n^\eps$ be this connection matrix at the degree $n$ of the homological Conley index, which takes the following form:
\begin{equation}\label{connectionmatrix}\tag{$\star$}
\Delta_n^{\eps} =
\begin{bNiceMatrix}[first-row, first-col]
      &\DP_0             &\DP_1           \\
\DP_0 & \h{X}_n^\eps(0)  &\h{T}_n^\eps    \\
\DP_1 & 0                &\h{X}_n^\eps(1) \\
\end{bNiceMatrix}
\end{equation}
Here, we have the following three submatrices:
\begin{enumerate}
	\item The submatrix $\h{X}_n^\eps(0)$, a $|\DP_0|\times |\DP_0|$ matrix, is a connection matrix of the Morse decomposition $(\mset{M_p^0\times \sset{0}}{p\in \DP_0},\h{\DP}_0)$ for the flow $\Phi_t^\eps$, 
	\item The submatrix $\h{X}_n^\eps(1)$, a $|\DP_1|\times |\DP_1|$ matrix, is a connection matrix of the Morse decomposition $(\mset{M_q^1\times \sset{1}}{q\in \DP_1},\h{\DP}_1)$ for the flow $\Phi_t^\eps$, 
	\item The submatrix $\h{T}_n^\eps$ is a $|\DP_0|\times |\DP_1|$ matrix.
\end{enumerate} 

\par In application, we are only interested in the connection matrics $\Delta^\eps_*$ as $\eps \to 0$. Since we use a finite field coefficient to compute the homological Conley index, there are only finitely many possible connection matrices for each $\eps$ under the flow $\Phi_t^\eps$. By pigeonhole principle, there is a sequence $\sset{\eps_m}_\for{m}$ with $\eps_m \to 0$ as $m \to \infty$, such that these connection matrices $\Delta^{\eps_m}_*$ are identical, namely $\Delta^{\eps_m}_*=\Delta_*^{\eps_{m+1}}$ for all $\for{m}$. 
From now on, we only consider such a sequence $\sset{\eps_m}_\for{m}$. 
Let $\Delta_*:=\Delta^{\eps_m}_*$ for all $\for{m}$.  As a result, we can omit the dependence on $\eps$ and rewrite the connection matrix at the degree $n$ as:
\begin{equation}	\label{connectionmatrixinproductflow}\tag{$\star\star$}
\Delta_n=
\begin{bmatrix}
\h{X}_n(0)  &\h{T}_n  \\
0           &\h{X}_n(1)
\end{bmatrix}
\end{equation}
The submatrix $\h{T}_n$ in the above $($\ref{connectionmatrixinproductflow}$)$:
$$
\h{T}_n: \bigoplus_{q\in \DP_1}C\!H_n(M_q^1\times \sset{1}) \to \bigoplus_{p\in \DP_0}C\!H_{n-1}(M_p^0\times \sset{0})
$$
is a degree $-1$ map. We show that there is a natural corresponding degree $0$ map $T_n$ to $\h{T}_{n+1}$ for all $n\in \NN$.


\begin{lma}[\cite{Reineck1998connection}, Theorem $5.4$]\label{homologybraidisomorphism}
For $\lam=0,1$	, there is a $X_n(\lam)$ which is a connection matrix for $(\MD{\lambda},\DP_\lambda)$ under the flow $\phi_t^\lambda$ and a natural homology braid isomorphism $\Theta_\lam=\mset{\theta^\lam_n(I)}{I\in \DI(\DP_\lambda), n\in \NN}$ of degree $\lambda$, such that 
$$
\h{X}_n(\lam)(i,j)= (\theta_n^\lam(i))^{-1} \circ X_{n-\lambda}(\lam)(i,j)\circ \theta^\lam_{n+1}(j), \qquad \text{for any } i,j\in \DP_{\lam}.
$$
\end{lma}

\par In other words, the preceding lemma shows the following diagram commutes:
$$\begin{CD}
C\!H_{n+1}(M_j^\lambda\times \sset{\lambda}) @>{\h{X}_n(\lambda)(i,j)}>> C\!H_{n}(M_i^\lambda\times\sset{\lambda}) \\
	     @V{\theta_{n+1}^\lambda(j)}VV       @V{\theta_{n}^\lambda(i)}VV\\
C\!H_{n+1-\lam}(M_j^\lambda) @>{X_{n-\lambda}(\lambda)(i,j)}>> C\!H_{n-\lambda}(M_i^\lambda)		
\end{CD}$$
\begin{note}
	There is an extension of the preceding lemma by K.Mischaikow (\cite{Mischaikow1989Transitionsystems}, Theorem $2.10$), and the idea is that if the dynamics in the parameter space added a $k$-dimensional unstable manifold to each Morse set, then the homology braid isomorphism is of degree $k$.
\end{note}
\par Now there is a degree $0$ map $T_n$ corresponding to $\h{T}_{n+1}$ in a natural way. Using the $\theta^\lambda_*$ in the preceding lemma, we define the degree $0$ map $T_*$.
\begin{dfn}\label{defofchainmap}
	Define $ T_n:\oplus_{p\in\DP_1}C\!H_n(M_p^1)\to \oplus_{q\in\DP_0}C\!H_n(M_q^0) $ by
	$$
	T_n(i,j):=\theta^0_n (i)\circ\h{T}_{n+1}(i,j)\circ{(\theta^1_{n+1}(j))}^{-1}, \qquad \text{for any } i\in\DP_0, j\in \DP_1.
	$$
\end{dfn}
\begin{note}
This means the map $T$ makes the following diagram commutes:
$$\begin{CD}
	C\!H_{n+1}(M_j^1\times \sset{1}) @>{\h{T}_{n+1}(i,j)}>> C\!H_{n}(M_i^0\times\sset{0}) \\
	     @V{\theta_{n+1}^1(j)}VV       @V{\theta_n^0(i)}VV\\
C\!H_n(M_j^1) @>{T_n(i,j)}>> C\!H_n(M_i^0)	
\end{CD}$$
\end{note}

\begin{dfn}[Singular transition matrix] Under the above situation, we call the homomorphism $T_n$: 
$$ 
T_n:\bigoplus_{p\in\DP_1}C\!H_n(M_p^1)\to \bigoplus_{q\in\DP_0}C\!H_n(M_q^0) 
$$
\textit{\textbf{a singular transition matrix}} from $(\MD{1},\DP_1)$ to $(\MD{0}, \DP_0)$ at the degree $n$.
\end{dfn}
\par By the existence of connection matrix $\Delta_*$, we get the existence of the singular transition matrix as below.
\begin{thm}[Existence of singular transition matrix]
The set of singular transition matrices from $(\MD{1},\DP_1)$ to $(\MD{0}, \DP_0)$  is not empty.
\end{thm}
\begin{note}
The singular transition matrix ${T}_n$ from $(\MD{1},\DP_1)$ to $(\MD{0}, \DP_0)$ might not be unique because of the choice of the sequence $\sset{\eps_m}_\for{m}$. 
\end{note}

\par From now on, we go back to the Morse decomposition of the invariant set $\hse$ as $(\DM(\hse),\h{\DP})$ for the flow $\Phi_t^\eps$. From the submatrix $\h{T}_n^\epsilon$ in (\ref{connectionmatrix}), we can get the information of connecting orbits for any adjacent pair $(p,q)$ in $\h{\DP}$ with $p\in\DP_0$, $q\in \DP_1$ and combining with Proposition \ref{Existenceofaconnectingorbit}, we can restate the existence of a connecting orbit using the submatrix $\h{T}_n^\epsilon$.
\begin{prop} If there is an adjacent pair $(p,q)$ with $p\in\DP_0$, $q\in \DP_1$ such that the $(p,q)$ element in the submatrix $\h{T}_n^\epsilon(p,q)\neq 0$ at a degree $n$, then there is a connecting orbit from $M_q^1\times\sset{1}$ to $M_p^0\times\sset{0}$ under the flow $\Phi_t^{\eps}$.
\end{prop}
\par From the preceding proposition and the definition of the singular transition matrix, we have the following theorem.
\begin{thm}
If there is an adjacent pair $(p,q)$ with $p\in\DP_0$, $q\in \DP_1$ such that the $(p,q)$ element in the singular transition matrix ${T}_n(p,q)\neq 0$ at a degree $n$, then there is a sequence $\sset{\eps_m}_\for{m}$ with $\eps_m \to 0$ as $m\to \infty$, such that there is a connecting orbit from $M_q^1\times\sset{1}$ to $M_p^0\times\sset{0}$ under the flow $\Phi_t^{\eps_m}$.
\end{thm}
\begin{ex}
\par We consider the pitchfork bifurcation happened at $\lz\in[0,1]$ and construct a singular transition matrix for it. Let $f(x,\lam)$ in (\ref{ODEn}) be $f(x,\lam):=(\lam - \lz)x-x^3$. 
\par On the left side of Figure \ref{TransitionMatricesforthePitchforkBifurcation}, the bifurcation diagram over parameter space $\Lam$ and phase portrait of $\phi_t^\lam$ at $\lam=0,\lz,1$ are denoted in the light grey. The dotted curves denote the unstable fixed point, and the solid curves denote the stable one. The phase portrait for $\Phi_t^{\eps_m}$ is denoted by black solid curves. The right side of Figure \ref{TransitionMatricesforthePitchforkBifurcation} illustrates a perturbed pitchfork bifurcation consisting of a saddle-node bifurcation and a stable fixed point continuing over $[0,1]$, in which grey and black curves are denoted in the same way as the left one. 

\begin{figure}[H]
\centering  
\subfigure[Pitchfork Bifurcation]{
\label{PitchforkBifurcation}
\includegraphics[width=0.45\textwidth]{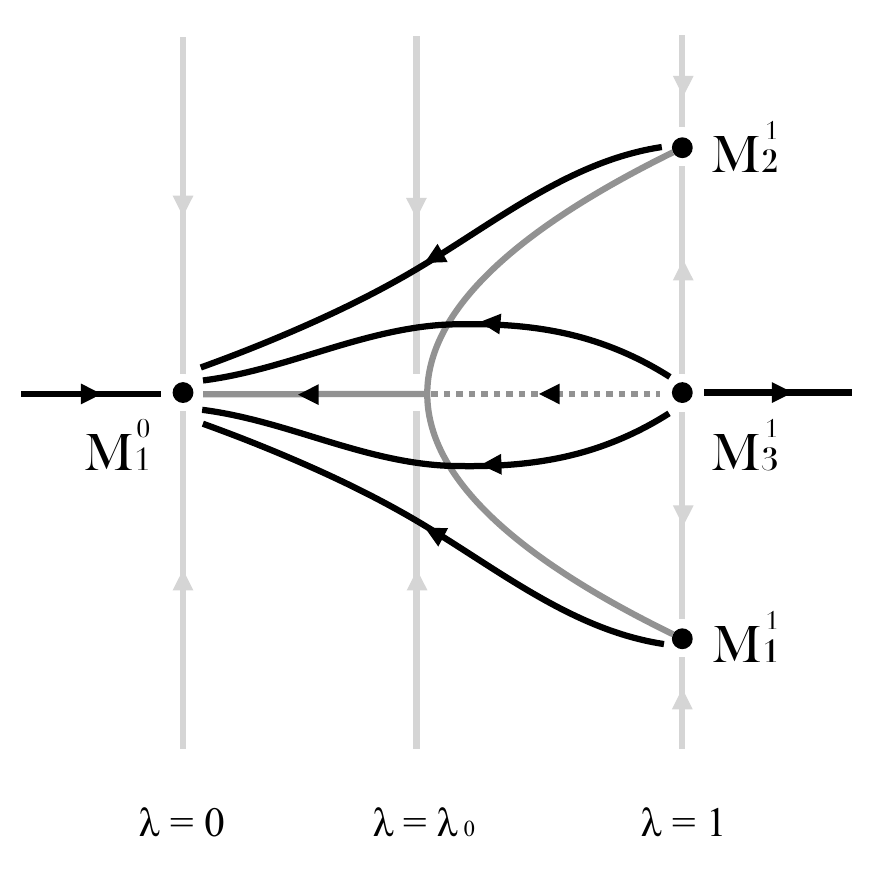}}
\subfigure[Perturbed Pitchfork Bifurcation]{
\label{PerturbedPitchforkBifurcation}
\includegraphics[width=0.45\textwidth]{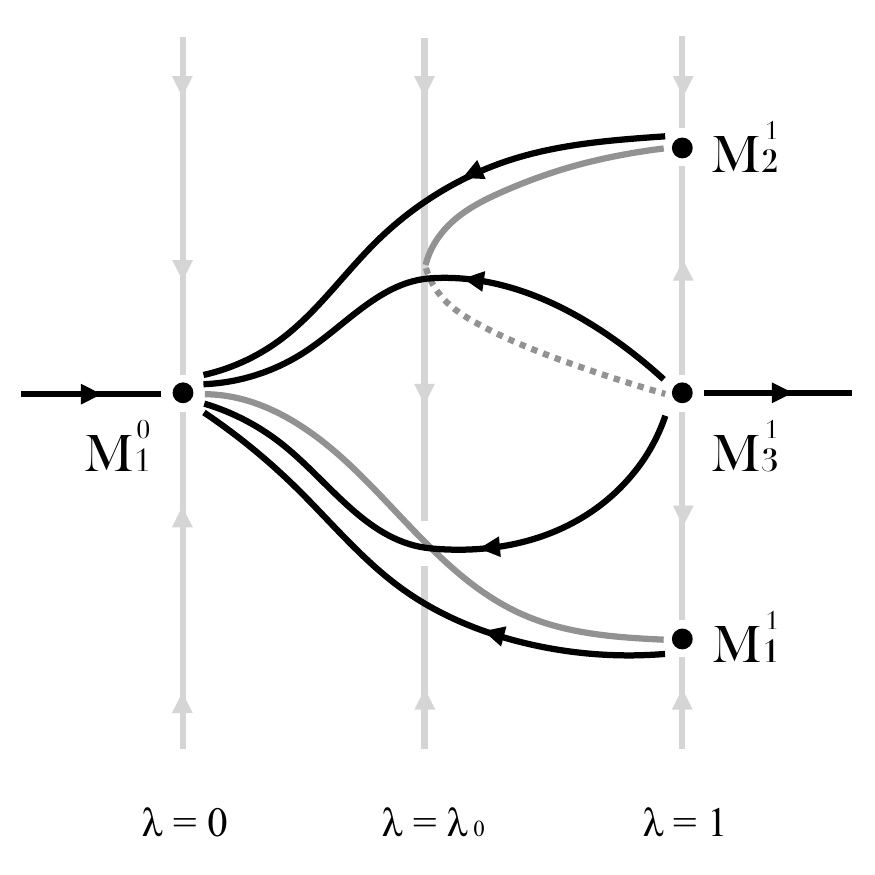}}
\caption{Slow-drift for the pitchfork bifurcation and its imperfection}
\label{TransitionMatricesforthePitchforkBifurcation}
\end{figure}
\par We take the same Morse decompositions as in Example \ref{Connecingorbitexample}. Then, in both situations $\DP_0=\sset{1^0}$ and $\DP_1=\sset{1^1,2^1,3^1}$. We compute the connection matrix $X_*(\lam)$ under the flow $\phi_t^\lam$ for $\lam=0,1$. Writing the isomorphism as $1$ and using the $\ZZ_2$ coefficient, we have the nonzero ones are:
$$X_1(0)=
\begin{bmatrix}
	0 \\
\end{bmatrix}\qquad \qquad
X_1(1)=
\begin{bmatrix}
	0 &0 &1 \\
	0 &0 &1 \\
	0 &0 &0 \\ 
\end{bmatrix}
$$
\par Considering, $(M_1^0,M_2^1)$ and $(M_1^0,M_1^1)$ are attractor-repeller pairs and the isolating neighbours $N(M_1^0,M_2^1)$ and $N(M_1^0,M_1^1)$ have trivial Conley index, we have that $\h{T}_1^{\eps_m}(M_1^0,M_2^1)=\h{T}_1^{\eps_m}(M_1^0,M_1^1)=1$. Considering the $\h{T}_1^{\eps_m}$ is a degree $-1$ map, the element $\h{T}_1^{\eps_m}(M_1^0,M_3^1)=0$. For all sufficiently small $\eps_m$, each element in $\h{T}_1^{\eps_m}$ stays the same, and we obtain the singular transition matrix:
$$
T_0= \h{T}_1=
\begin{bNiceMatrix}[first-row, first-col]
       &M_1^1  &M_2^1  &M_3^1\\
M_1^0  &1      &1      &0    \\
\end{bNiceMatrix},$$
and
$$
T_n=\begin{bNiceMatrix}[first-row, first-col]
       &M_1^1  &M_2^1  &M_3^1\\
M_1^0  &0      &0      &0    \\
\end{bNiceMatrix}\quad \text{for all $n\in \NN_+$}.
$$
\par From the singular transition matrix, we can read the information that there is a connecting orbit from $M_2^1$ to $M_1^0$ and $M_1^1$ to $M_1^0$ under the flow $\Phi_t^{\eps_m}$ for all $\for{m}$. Considering connecting orbits between Morse sets, we know the singular transition matrix of the two systems above is uniquely given by ${T}$ in this case.
\end{ex}

\par Furthermore, by taking the Hausdorff limit of the connecting orbits $\sset{\cc_{m}}_\for{m}$, and using Theorem \ref{withoutconn}, we can get the information of bifurcation between two slices at $\lam=0$ and $\lam=1$ from the singular transition matrix. This simplifies the computation of connecting orbits to a linear algebra problem. 
\par Sometimes, only by the information of the two slices can we infer the information of connecting orbits, which will be shown in the next example.
\begin{ex}\label{mainexample}
We consider the parameterized flow in Figure \ref{PhasePortraitsattwoslices}, the flow at parameter $\lam=0$ is shown on the left side, and the flow at parameter $\lam=1$ is shown on the right side in Figure \ref{PhasePortraitsattwoslices}. The shaded region $N$ in figure \ref{PhasePortraitsattwoslices} is an isolating neighborhood for all $\lam\in\Lam$. Let $S^\lambda=\inv(N,\phi_t^\lambda)$, at the two sildes $\lam=0,1$ there are two Morse decompositions $(\MD{0},\DP_0)$ and $(\MD{1},\DP_1)$ with $\MD{0}=\sset{M_1^0, M_2^0, M_5^0},(\DP_0,<_0)=(\sset{1^0,2^0,5^0},1^0<_0 2^0<_0 5^0)$ and $\MD{1}=\sset{M_1^1, M_2^1, M_3^1, M_4^1, M_5^1}, (\DP_1,<_1)=(\sset{1^1,2^1,3^1,4^1,5^1},1^1<_1 2^1<_1 3^1<_1 4^1<_1 5^1)$.
\begin{figure}[H]
\centering  
\subfigure[$\lam=0$]{
\label{lam_0}
\includegraphics[width=0.45\textwidth]{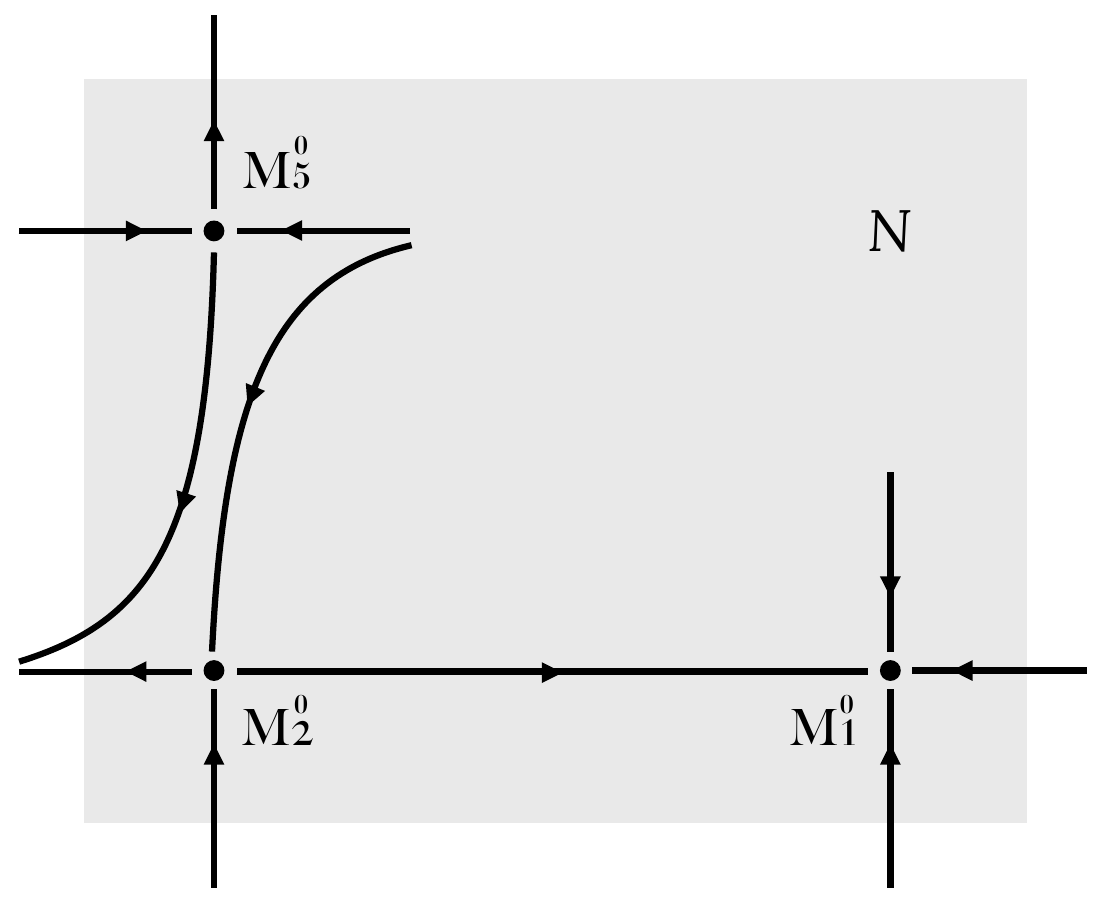}}
\subfigure[$\lam=1$]{
\label{lam_1}
\includegraphics[width=0.45\textwidth]{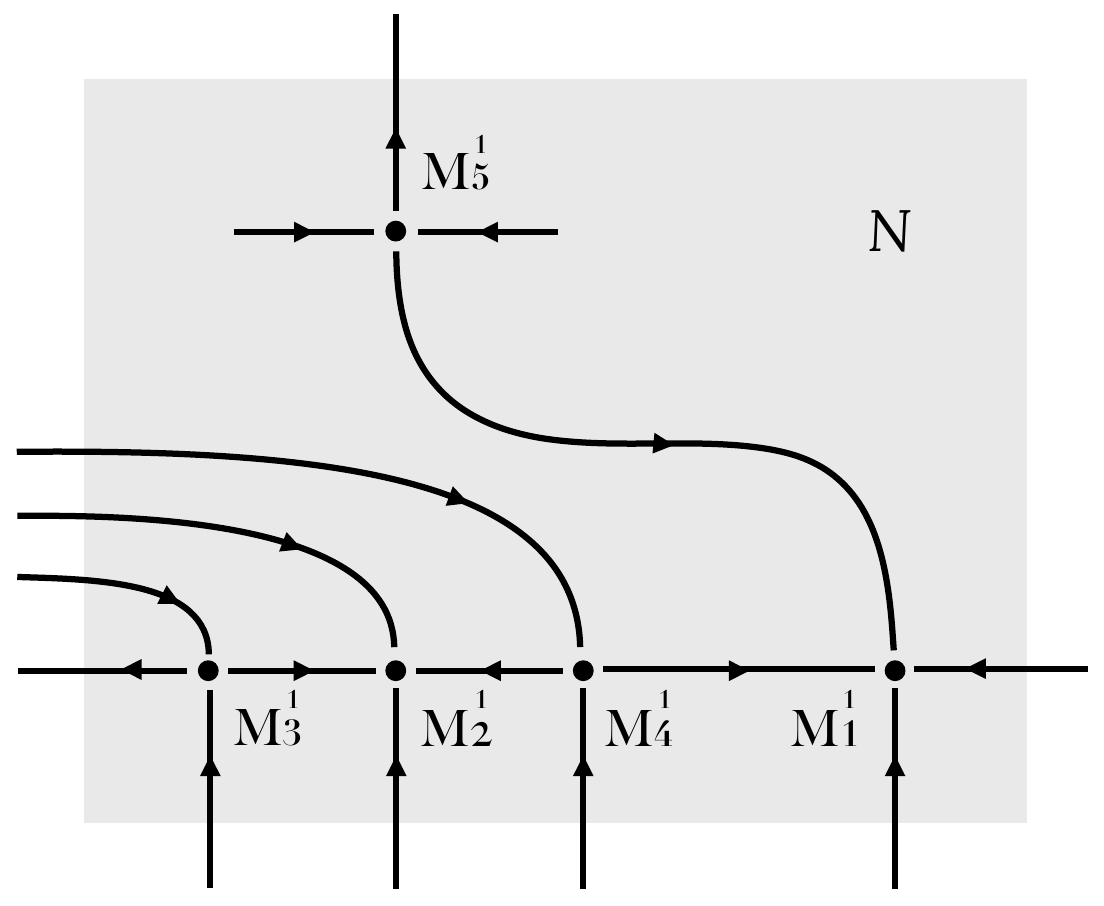}}
\caption{Phase portraits at two slices $\lam=0$ and $\lam=1$}
\label{PhasePortraitsattwoslices}
\end{figure}
\par We assume $(\MD{0},\DP_0)$ continues over $[0,\lam_0)$ and $(\MD{1},\DP_1)$ continues over $(\lam_0,1]$ with a breaking down of the continuation at $\lam=\lam_0$ and $N$ is an isolating neighborhood over the parameter space $\Lam$. Thus, it satisfies Assumption \ref{breakdown}. The finest decomposition of $(\DP_0, \DP_1)$ is given by: 
$$
\begin{aligned}
\sset{(J_i,J^{'}_i)}_{i\in\sset{1,2,3}}=\{(\sset{1^0}, \sset{1^1})&, (\sset{2^0}, \sset{2^1,3^1,4^1}), (\sset{5^0}, \sset{5^1}) \}; \\
(J_1,J_1^{'}) & = (\sset{1^0}, \sset{1^1});\\
(J_2, J_2^{'}) & = (\sset{2^0}, \sset{2^1,3^1,4^1});\\
(J_3, J_3^{'}) & = (\sset{5^0}, \sset{5^1}).
\end{aligned}
$$
We assume $M_1^\lambda$ is a sink for all $\lambda\in[0,1]$. We also assume there is no connecting orbit from $M^\lam_4$ to $M^\lam_3$ for all $\lam\in (\lam_0,1]$, and there are connecting orbits from $M^\lam_4$ to $M^\lam_2$ and from $M^\lam_3$ to $M^\lam_2$ for all $\lam\in (\lam_0,1]$.
\par By the information on the two slices, we can infer the information of connecting orbits by the singular transition matrix.
\par Firstly, we compute the connection matrix $X_*(\lam)$ under the flow $\phi_t^\lam$ for $\lam=0,1$. Writing the isomorphism as $1$ and using the $\ZZ_2$ coefficient, we have the nonzero ones are:
$$X_1(0)=
\begin{bmatrix}
	0 &1 &0\\
	0 &0 &0\\
	0 &0 &0\\
\end{bmatrix}\qquad \qquad
X_1(1)=
\begin{bmatrix}
	0 &0 &0 &1 &1\\
	0 &0 &1 &1 &0\\
	0 &0 &0 &0 &0\\
	0 &0 &0 &0 &0\\
	0 &0 &0 &0 &0\\ 
\end{bmatrix}
$$
By Lemma \ref{homologybraidisomorphism}, in the product space and under the flow $\Phi_t^{\eps_m}$, we have $\h{X}_1^{\eps_m}(0)=X_1(0)$ and $\h{X}_2^{\eps_m}(1)=X_1(1)$ for all $\for{m}$. Then, we consider the connection matrix $\Delta_{*}^{\eps_m}$ at the degree $1$ for the flow $\Phi_t^{\eps_m}$, in which we write the $\h{X}_2^{\eps_m}(1)$ and $\h{X}_1^{\eps_m}(0)$ in the same matrix: 
$$\Delta_{*}^{\eps_m} =
\begin{bNiceMatrix}[first-row, first-col]
      &\DP_0             &\DP_1           \\
\DP_0 & \h{X}_1^{\eps_m}(0)  &\h{T}_1^{\eps_m}    \\
\DP_1 & 0                &\h{X}_2^{\eps_m}(1) \\
\end{bNiceMatrix}
$$
For each $\for{m}$, because $M_1^\lambda$ is a sink for all $\lambda\in[0,1]$, by Lemma $5.8$ in \cite{Reineck1998connection}, the element $\Delta_{*}^{\eps_m}(M_1^0,M_1^1)=1$. For each $\for{m}$, $\CH_*(M_j^1)$ is nonzero only at degree $2$ for $j=3,4,5,$ $\CH_*(M_1^0)$ is nonzero only at degree $0$, and $\h{T}_1^{\eps_m}$ is a degree $-1$ map, we have $\Delta_{*}^{\eps_m}(M_1^0,M_j^1)=0$ for $j=3,4,5.$ Now, we can omit $\eps_m$ and get the following connection matrix:
\renewcommand{\arraystretch}{1.25}
$$
\Delta_{*}=
\begin{bNiceMatrix}[left-margin][first-row, first-col]
        &M_1^0  &M_2^0  &M_5^0  &M_1^1  &M_2^1  &M_3^1  &M_4^1  &M_5^1 \\
M_1^0   &0      &1      &0      &1      &       &0      &0      &0     \\   
M_2^0   &0      &0      &0      &       &       &       &       &?     \\
M_5^0   &0      &0      &0      &       &       &       &       &      \\
        
M_1^1   &\Block{5-3}<\Huge>{0}       &       &       &0      &0      &0      &1      &1\\
M_2^1   &       &       &       &0      &0      &1      &1      &0     \\
M_3^1   &       &       &       &0      &0      &0      &0      &0     \\
M_4^1   &       &       &       &0      &0      &0      &0      &0     \\
M_5^1   &       &       &       &0      &0      &0      &0      &0     \\
\end{bNiceMatrix}
$$
\renewcommand{\arraystretch}{1}
Because the connection matrix is a boundary map $\Delta_1\circ\Delta_2=0$, we have $\h{X}_1(0)\h{T}_{2}+\h{T}_1\h{X}_2(1)=0$ which implies $X_1(0)\h{T}_{2}+\h{T}_1X_1(1)=0$. Thus, it gives 
$$\h{T}_2(M_2^0,M_5^1)=1.$$
Because $M_1^\lambda$ is a sink for all $\lambda\in[0,1]$, we can conclude that there is no connecting orbit from $M_1^1$ to $M_2^0$ under the flow $\Phi_t^{\eps_m}$ for all $m\in\NN_+$. If not, we can take a Hausdorff limit of the connecting orbits and get a connecting orbit that goes from $M_1^\gamma$ at some parameter $\gamma\in [0,1]$ by the Theorem \ref{withoutconn}. But this contradicts that $M_1^\gamma$ is a sink. Since there is no connecting orbit from $M_1^1$ to $M_2^0$, $(M_2^0, M_5^1)$ is an adjacent pair.

\par  $\h{T}_2(M_2^0,M_5^1)=1$ means that $M_2^0<^{\Gf} M_5^1$ in the flow-defined order in each flow $\Phi_t^{\eps_m}$ and there is a connecting orbit from $M_5^1$ to $M_2^0$ as $\eps_m\to 0$.
\par If we only consider the finest decomposition $\sset{({J_i},{J^{'}_i})}_{i\in\sset{1,2,3}}$, we can conclude that there is a parameter $\gamma\in[0,1]$ such that there is a connecting orbit from $\GM^\gamma_{\DP_\gamma(N(M_{J_3}))}$ to $\GM^\gamma_{\DP_\gamma(N(M_{J_2}))}$.
\par There are four possibilities of the connecting orbits between these Morse sets. Let $\GM_{P\!F}^\lz:=\GM^\lz_{\DP_\lz(N(M_{J_2}))}, \GM_{S\!I}^\lz:=\GM^\lz_{\DP_\lz(N(M_{J_1}))}$ and $\GM_{S\!A}^\lz:=\GM^\lz_{\DP_\lz(N(M_{J_3}))}$ be the maximal invariant set in each isolating neighborhood of continuable interval pairs.
\begin{enumerate}[label=\textnormal{\roman*).}]
	\item $M_5^1 \to \GM_{S\!A}^{\lz}>_\lz \GM_{P\!F}^{\lz} \gets M_2^0$: there is a connecting orbit from $\GM_{S\!A}^\lz$ to $\GM_{P\!F}^\lz$ under the flow $\phi_t^\lz$.
	\item $M_5^1 \to \GM_{S\!A}^{\lz}\gets M_5^0>_{\gamma} M_2^0$: there is a $\gamma\in [0,\lz)$ such that there is a connecting orbit from $M_5^\gamma$ to $M_2^\gamma$ under the flow $\phi_t^\gamma$, $\HLL{5}{+}\subset \GM_{S\!A}^{\lz}$ and $\HLL{5}{-}\subset \GM_{S\!A}^{\lz}$. This means the two saddles $M_5^0$ and $M_2^0$ are connected before $\lz$.
	\item $M_5^1>_{\gamma} M_4^1 (\text{or } M_3^1)  \to \GM_{P\!F}^{\lz}\gets M_2^0$: there is a $\gamma\in (\lz,1]$ such that there is a connecting orbit from $M_5^\gamma$ to $ M_4^\gamma (\text{or } M_3^\gamma)$ under the flow $\phi_t^\gamma$, $\HLL{4}{+}$ $(\text{or }\HLL{3}{+}) \subset \GM_{P\!F}^{\lz}$ and $\HLL{2}{-}\subset \GM_{P\!F}^{\lz}$. This means the two saddles $M_5^1$ and $M_4^1$ (or $M_3^1$) are connected after $\lz$.
	\item $M_5^1>_{\gamma_1} M_4^1 (\text{or } M_3^1) >_{\gamma_2} M_2^1 \to \GM_{P\!F}^{\lz} \gets M_2^0$: there are $\gamma_1,\gamma_2\in (\lz,1]$ with $\gamma_1\geq \gamma_2$ such that there are connecting orbits from $M_5^{\gamma_1}$ to $ M_4^{\gamma_1} (\text{or } M_3^{\gamma_1})$ under the flow $\phi_t^{\gamma_1}$, and from $ M_4^{\gamma_2}(\text{or } M_3^{\gamma_2})$ to $ M_2^{\gamma_2}$ under the flow $\phi_t^{\gamma_2}$. $\HLL{2}{+}\subset \GM_{P\!F}^{\lz}$ and $\HLL{2}{-}\subset \GM_{P\!F}^{\lz}$. This also means that the two saddles $M_5^1$ and $M_4^1$ (or $M_3^1$) are connected after $\lz$.
\end{enumerate}
\par We can conclude that the saddles are connected at some parameter, but we should note here only from the $\h{T}_2(M_2^0, M_5^1)\neq 0$ we still do not know if the connection is before the pitchfork-like bifurcation or after it.
\end{ex}

\subsection{An Axiomatic Definition of the Transition Matrix}\label{An AxiomaticExtensionApproachtotheSingularity}
\par In this section, we show that the map $T_*$ defined by Definition \ref{defofchainmap} is a chain map between the chain complexes consisting of the direct sum of the homological Conley index of each Morse set and their connection matrices. Since the connection matrix is not always easy to compute, the method we got the non-zero element in the example \ref{mainexample} motivates us to an axiomatic definition of the transition matrix. After defining the axiomatic transition matrix, we end by discussing some open problems.  
\par Assume that we have a collection of Morse decompositions $\sset{(\MD{\lambda},\DP_\lambda)}_{\lam\in\Lam}$ under Assumption \ref{breakdown}. 
\par Firstly, let $(\oplus_{j\in\DP_1}C\!H_*(M_j^1),X(1))$ and $(\oplus_{i\in \DP_0}C\!H_*(M_i^0),X(0))$ be two chain complexes,  with $C\!H_*(M_k^\lam)$ the homological Conley index for the Morse set $M_k^\lambda$ and $X(\lambda):=\sset{X_n(\lambda)}_\for{n}$ the connection matrix under the flow $\phi_t^\lam$ for $(\MD{\lambda},\DP_\lambda)$. Let $Y_*(1):=X_*(1)$ and $Y_*(0)=-X_*(0)$. From the definition of the connection matrix, $Y_*(0)$ and $Y_*(1)$ are also connection matrices for the flow $\phi_t^0$ of $(\MD{0},\DP_0)$ and  for the flow $\phi_t^1$ of $(\MD{1},\DP_1)$, respectively. Let $Y(\lambda):=\sset{Y_n(\lambda)}_\for{n}$ for $\lam=0,1$, we have two new chain complexes 
$$(\oplus_{j\in\DP_1}C\!H_*(M_j^1),Y(1)) \qquad \text{and} \qquad (\oplus_{i\in \DP_0}C\!H_*(M_i^0),Y(0)).$$
Let $T:=\sset{T_n}_{n\in\NN}$, in which $T_n$ are degree $0$ maps defined by Definition \ref{defofchainmap}. Then $T$ is a chain map between the two chain complexes, as shown in the next proposition.
\begin{prop}\label{propofchainmap}
	$T$ is a chain map between the chain complexes $(\oplus_{j\in\DP_1}C\!H_*(M_j^1),Y(1))$ and $(\oplus_{i\in \DP_0}C\!H_*(M_i^0),Y(0))$.
\end{prop}
\begin{proof}
Since the connection matrix $\Delta_*$ for $\Phi_t^{\eps_m}$ is a boundary operator, namely, the composition $\Delta_n\circ\Delta_{n+1}=0$, we have:
\begin{align*}
\begin{bmatrix}
\h{X}_n(0)  &\h{T}_n  \\
0           &\h{X}_n(1)
\end{bmatrix}
\circ &
\begin{bmatrix}
\h{X}_{n+1}(0)  &\h{T}_{n+1}  \\
0           &\h{X}_{n+1}(1)
\end{bmatrix}
=\\
& \quad
\begin{bmatrix}
\h{X}_n(0)\circ\h{X}_{n+1}(0)  &\h{X}_n(0)\circ\h{T}_{n+1}+\h{T}_n\circ\h{X}_{n+1}(1)  \\
0                              &\h{X}_n(1)\circ\h{X}_{n+1}(1)
\end{bmatrix}=
\begin{bmatrix}
0 &0\\
0 &0	
\end{bmatrix}
.
\end{align*}
Therefore, we have $\h{X}_n(0)\circ\h{T}_{n+1}+\h{T}_n\circ\h{X}_{n+1}(1)=0$.
\par For any $(i,j)$ element in $\h{X}_n(0)\circ\h{T}_{n+1}+\h{T}_n\circ\h{X}_{n+1}(1)$ we have, from Definition \ref{defofchainmap}: 
\begin{equation*}
	\h{X}_n(0)\circ(\theta^0_n)^{-1}\circ{T}_n\circ\theta^{1}_{n+1}+(\theta^0_{n-1})^{-1}\circ{T}_{n-1}\circ\theta^{1}_{n}\circ\h{X}_{n+1}(1)=0
\end{equation*}
and from Lemma \ref{homologybraidisomorphism}, we multiply $\theta^0_{n-1}$ from left and multiply $(\theta^1_{n+1})^{-1}$ from right then we get:
\begin{align*}
	\theta^0_{n-1}\circ\h{X}_n(0)\circ(\theta^0_n)^{-1}\circ{T}_n+{T}_{n-1}\circ\theta^{1}_{n}\circ \h{X}_{n+1}(&1)\circ (\theta^1_{n+1})^{-1}=\\
	{X}_n(&0)\circ{T}_n+{T}_{n-1}\circ{X}_{n}(1)=0	
\end{align*}
\par Finally from the defintion of $Y_*(0)$ and $Y_*(1)$, we have 
$$
{Y}_n(0)\circ{T}_n-{T}_{n-1}\circ{Y}_{n}(1)=0
$$
which makes the following diagram commutes:

$$\begin{CD}
\bigoplus\limits_{j\in \DP_1} C\!H_{n}(M_j^1)      @>{T_n}>>       \bigoplus\limits_{i\in \DP_o} C\!H_{n}(M_i^0) \\
@V{{Y}_n(1)}VV                            @V{Y_n(0)}VV \\
\bigoplus\limits_{j\in \DP_1} C\!H_{n-1}(M_j^1)    @>{T_{n-1}}>> \bigoplus\limits_{i\in \DP_0} C\!H_{n-1}(M_i^0)\\
\end{CD}$$
This shows that the degree 0 map 
$T$
 is a chain map between $(\oplus_{j\in\DP_1}C\!H_*(M_j^1),Y(1))$ and $(\oplus_{i\in \DP_0}C\!H_*(M_i^0),Y(0))$.
\end{proof}

\par Enlightened by the preceding properties of the chain map and the way we got the non-zero element in the example \ref{mainexample}, we try to give an axiomatic definition of the transition matrix without a continuation based on Kokubu \cite{Kokubu2000ontransitionmatrices}, which gave an axiomatic transition matrix assuming continuation.
\par Considering that the homological Conley index $\CH_* (\h{S})=0$, and the $(S^0\times\sset{0},S^1\times\sset{1})$ is an attractor-repeller pair in $\h{S}$, we have the connection matrix $\Delta_n$ induces the isomorphism from $\CH_{n}(S^1\times\sset{1})$ to $\CH_{n-1}(S^0\times\sset{0})$ in the long exact sequence and we call it the \textit{\textbf{global continuation isomorphism}}. When without a continuation, this property can be extended to each continuable interval pair by Proposition \ref{TrivialConleyIndexofEachContinuableIntervalPair}, and we call this isomorphism between a continuable interval pair the \textit{\textbf{continuable interval continuation isomorphism}}.
\par Let $X_*(k)$ be connection matrices at slices $k=0,1$ under the flow $\phi_t^k$. Let 
$$\DC X(k)=(\oplus_{p\in\DP_k}C\!H_{*}(M^k_p),X(k))$$
be the chain complex with its boundary map $X_*(k)$. Under Assumption \ref{breakdown}, a decomposition of the pair $(\DP_0,\DP_1)$ is given by $\sset{(J_i,J_i^{'})}_{i\in A}$. 
For each continuable interval pair $(J_i,J_i^{'})$ in the decomposition $\sset{(J_i,J_i^{'})}_{i\in A}$, the restriction of the connection matrices $X_*(0)$ and $X_*(1)$ on the intervals $J_i$ and $J_i^{'}$ in the continuable interval pair $(J_i,J_i^{'})$ are denoted as $X^{J_i}_*(0)$ and $X^{J_i^{'}}_*(1)$ respectively. 
Then $(-X_*^{J_i}(0))$ and $X_*^{J_i^{'}}(1)$ are connection matrices for Morse decompositions $\DM(M^0_{J_i})=\sset{M_p^0}_{p\in J_i}$ and $\DM(M^1_{J_i^{'}})=\sset{M_q^1}_{q\in J_i^{'}}$, respectively. 
Thus, for the continuable interval pair $(J_i,J_i^{'})$ we have two chain complexes 
$$\DC X^{J_i}(0)=(\oplus_{p\in {J_i}}C\!H_{*}(M^0_p), -X^{J_i}(0)) \text{ with its boundary map $-X^{J_i}_*(0)$},$$ 
and 
$$\DC X^{J_i^{'}}(1)=(\oplus_{q\in {J_i^{'}}}C\!H_{*}(M^1_q),X^{J_i^{'}}(1)) \text{ with its boundary map $X^{J_i^{'}}_*(1)$}.$$


\begin{dfn}[Axiomatic transition matrix]\label{AxiomaicTransitionMatrix}
Under the above situation, an \textit{\textbf{axiomatic transition matrix}} is a map $T:\DC X(1) \to \DC X(0)$ satisfying the following conditions:
\begin{enumerate}[label=\textnormal{(A\arabic*).}]
    \item $T$ is a chain map from $\DC X(1)$ to $\DC X(0)$.
	\item For each continuable interval pair $(J_i,J_i^{'})$ of any decomposition $\sset{(J_i,J_i^{'})}_{i\in A}$ of the pair $(\DP_0,\DP_1)$, let $T(J_i,J^{'}_i)$ be the restriction of $T$ on the continuable interval pair $(J_i,J_i^{'})$:
  $$
  T_*(J_i,J^{'}_i):\oplus_{q\in{J_i^{'}}}\CH_*(M_q^1) \to  \oplus_{p\in{J_i}}\CH_*(M_p^0).
  $$
 We require that $T(J_i,J^{'}_i)$ is a chain map between $\DC X^{J_i^{'}}(1)$ and $\DC X^{J_i}(0)$, and the homomorphism $T^{\sharp}_*{(J_i,J_i^{'})}$ induced by $T_*{(J_i,J_i^{'})}$ at the homology level:
	 $$
	 T^{\sharp}_*{(J_i,J_i^{'})}:\CH_*(M^1_{J_{i}^{'}})\to \CH_*(M^0_{J_{i}})
	 $$
	  is an isomorphism.
	
\end{enumerate}
\par The set of such chain maps, if exists, is denoted as $\DT(X(0),X(1))$ and the union of $\DT(X(0),X(1))$ over all possible boundary connection matrices at slices $\lam=0,1$ is denoted as $\DT(0,1)$.
\end{dfn}
\par Observing that any singular transition matrix is in the $\DT(0,1)$, and singular transition matrices always exist, we have the following existence theorem.
\begin{thm}[Existence of axiomatic transition matrix]\label{ExistenceofAxiomaticTransitionMatrix}
	$\DT(0,1)$ is not empty.
\end{thm}
\begin{note}
In Example \ref{mainexample}, we only used the (A1) and (A2) in the Definition \ref{AxiomaicTransitionMatrix} to conclude the non-zero adjacent pair element. Therefore, let $(p,q)$ be an adjacent pair of a Morse decomposition for the extended slow-fast flow $\Phi_t^\epsilon$. If for any $T\in\DT(X(0),X(1))$, there is a degree $n$ such that $T_n(p,q)\neq 0$, then there is a connecting orbit from $M_q^1\times\sset{1}$ to $M_p^0\times\sset{0}$ under the extended slow-fast flow $\Phi_t^\epsilon$.	
\end{note}


However, the following problems are still unsolved:
\begin{enumerate}
	\item Although we have provided a definition of the axiomatic transition matrix, we still need a method for determining all possible transition matrices. 
	\item If we use a finite field as the coefficient for the homological Conley indices, then there are only finitely many possible axiomatic transition matrices in $\DT(X(0), X(1))$. We want to know if an axiomatic transition matrix is also a singular transition matrix. Namely, we want to know if we can construct a parameterized flow $f^{T}(x,\lam)$ for each possible axiomatic transition matrix $T$ such that $T$ is a singular transition matrix for the extended slow-fast flow of the ODE $\dot{x}=f^{T}(x,\lam)$, to classify the possible bifurcations that happened in $(0,1)$.
\end{enumerate}

\subsection*{\small Acknowledgments}
I would like to express my deepest gratitude to my advisor, Professor Hiroshi Kokubu, not only for his invaluable guidance, support, and encouragement throughout the entire research process but also for the profound impact he has had on both my academic and personal growth.

I am also deeply grateful to Professor Zin Arai for his invaluable guidance, constructive feedback, and continuous support throughout this research.
\bibliographystyle{siam}
\bibliography{ref}
 
\end{document}